\documentclass[preprint]{imsart}

\RequirePackage{amsthm,amsmath,amsfonts,amssymb}
\RequirePackage{natbib}
\RequirePackage[colorlinks,citecolor=blue,urlcolor=blue]{hyperref}
\RequirePackage{graphicx}

\arxiv{arXiv:1805.08992}

\startlocaldefs
\usepackage[T1]{fontenc}
\usepackage[utf8]{inputenc}
\usepackage[lofdepth,lotdepth]{subfig}

\newcommand{\vecteur}{\boldsymbol}
\newcommand{\matrice}{\boldsymbol}

\numberwithin{equation}{section}
\theoremstyle{plain}
\newtheorem{thm}{Theorem}[section]

\DeclareMathOperator{\Tr}{Tr}
\DeclareMathOperator{\Ker}{Ker}

\newcommand{\R}{\mathbb{R}}
\newcommand{\N}{\mathbb{N}}
\newcommand{\Z}{\mathbb{Z}}
\newcommand{\trans}{^{\top}}

\newcommand{\bs}{\boldsymbol}

\providecommand{\floor}[1]{\left \lfloor #1 \right \rfloor }

\newcommand{\co}{ \bs{ \Sigma }_{ \theta } }
\newcommand{\de}{ \frac{d}{d \theta } \co }

\newcommand{\likelihood}{L_{\vecteur{y}}(\theta)}
\newcommand{\allparamlikelihood}{L_{\vecteur{y}}(\vecteur{\beta}, \sigma^2, \theta)}
\newcommand{\nulltrendallparamlikelihood}{L_{\bs{y}} (\sigma^2, \theta)}
\newcommand{\nouveauD}{\matrice{Z}}

\newtheorem{prop}[thm]{Proposition}

\newtheorem{lem}[thm]{Lemma}
\newtheorem{defn}[thm]{Definition}

\newtheorem{rmq}{Remark}

\endlocaldefs

\begin{document}

\begin{frontmatter}

\title{Propriety of the reference posterior distribution in Gaussian Process modeling}
\runtitle{Propriety of the reference posterior distribution}

\author{\fnms{Joseph} \snm{Muré}\corref{}\ead[label=e1]{joseph.mure@edf.fr}}

\address{EDF Recherche et Développement\\
Dpt. PRISME\\
6 quai Watier\\
78401 Chatou\\
France\\
\printead{e1}\\
}
\affiliation{Université Paris Diderot and EDF R\&D}

\runauthor{J. Muré}

\begin{abstract}
In a seminal article, \citet{BDOS01} compare several objective prior 
distributions for the parameters of Gaussian Process models with 
isotropic correlation kernel. 
The reference prior distribution stands out among them insofar as 
it always leads to a proper posterior. 
They prove this result for rough correlation kernels - 
Spherical, Exponential with power $\rho<2$, 
Matérn with smoothness $\nu<1$. 
This paper provides a proof for smooth correlation kernels - 
Exponential with power $\rho=2$, 
Matérn with smoothness $\nu \geqslant 1$, Rational Quadratic -
along with tail rates of the reference prior for these kernels.
\end{abstract}



\end{frontmatter}

\section{Introduction}

Gaussian processes are often used to emulate 
unknown functions from some space $\R^r$ ($r \geqslant 1$)
to $\R$ \citep{Ras06}.
Interpreted in Bayesian terms,
this analysis 
sets the distribution of a Gaussian process as the prior
distribution of the unknown function.
The posterior is still Gaussian, but conditioned on the values of the function 
observed at specific points of $\R^r$. 
In order to show the performance of such models,
upper bounds for posterior rates have been obtained by \citet{VZ08}
and lower bounds by \citet{Cas08}.
\medskip

The Gaussian prior distribution is usually parametric,
with the parameters typically set to the maximum likelihood estimator.
Another approach is hierarchical Bayesian modeling, where
a prior on the parameters of the prior 
Gaussian process distribution is defined.
This approach is extensively discussed in Chapter 6
of \citet{BCG04}.
\medskip 

Given prior knowledge about the parameters of the Gaussian prior
is often lacking, there is value in objective elicitation 
of their prior distribution. \medskip

However, \citet{DOKS97} and \citet{Ste99} noted that
commonly used noninformative priors sometimes failed to yield proper posteriors.
\citet{BDOS01} were the first to thoroughly investigate the issue. 
Among several prior distributions -- truncated priors, vague priors, 
Jeffreys-rule and independence Jeffreys prior -- 
they showed that the reference prior \citep{Ber05} is 
the most satisfying choice for a default prior distribution. 
This is due to the fact that, for all isotropic correlation kernels 
studied by \citet{BDOS01}, the reference prior yields 
a proper posterior distribution. 
In the following, whenever we mention a ``prior'',
it is the prior on the parameters of the Gaussian process distribution,
not the Gaussian process distribution itself.\medskip

Posterior propriety is necessary to Bayesian procedures aiming to quantify parameter uncertainty: it would make no sense to express parameter uncertainty through a posterior distribution that does not integrate to a finite mass.
Moreover, in spatial models, one may want to take parameter uncertainty into account when performing prediction.
Prediction can be averaged over a proper posterior distribution on parameters,
not an improper one. \medskip

In this article, we complete the proof by \citet{BDOS01} that the reference posterior is proper.
Because of the difficulty involved in obtaining a satisfying default prior distribution which consistently yields a proper posterior, it is important to ascertain that the reference prior actually does. Indeed, a vast literature builds upon this result. 

Provided some additional conditions on the design set and the mean function 
are verified,
\citet{paulo05} states that the reference posterior is also proper 
with anisotropic product correlation kernels that have specific properties. 
As \citet{paulo05} notes, these properties are also necessary
to make the proof from \citet{BDOS01} work.

He then warns the reader that the Squared Exponential kernel (Power Exponential
with power $\rho=2$) does not satisfy one of them.

\citet{RSH12} investigate the propriety of reference posteriors for isotropic 
correlation kernels with an additional noise term (nugget effect).
They show that the reference prior is not the same 
for different parameter orders, 
but that all reference priors lead to a proper posterior.
However, the correlation kernel must have the same properties as those 
required by \citet{BDOS01}.

\citet{KP12} consider the same setting and focus on the reference prior 
with a particular parameter order. They provide an application to the analysis of zinc concentrations in the French river Meuse.

\citet{RSS13} return to the setting explored by \citet{paulo05} -- anisotropic product correlation kernels -- but relax his 
conditions on the mean function.

\citet{GWB18} go further still and also relax his conditions 
on the design set. Moreover, they establish reference posterior propriety
with anisotropic product covariance kernels and an additional noise term (nugget effect). Their work thus generalizes and combines results from \citet{paulo05} and \citet{RSS13} on one side, \citet{RSH12} and \citet{KP12} on the other.

However, all these papers prove reference posterior propriety under similar assumptions
as those made by \citet{BDOS01}.
They rely on properties that only rough correlation kernels possess.
\medskip

The main result of this work is 
Theorem \ref{THM:REFERENCE_POSTERIOR_PROPER}, 
which ensures that the reference prior leads to a 
proper posterior distribution 
for a large class of smooth isotropic kernels.
Secondary results include bounds on the reference prior density
and on the function obtained
after integrating several parameters
out of the likelihood. \medskip

The paper is organized as follows.
Section \ref{Sec:Setting} describes the Gaussian Process models considered by \citet{BDOS01}.
Section \ref{Sec:smoothness_correlation_kernel} shows that the proof of the propriety of the reference posterior provided by \citet{BDOS01} only applies to Gaussian Process models with rough correlation kernels -- Spherical, Exponential with power $\rho<2$, Matérn with smoothness $\nu<1$.
Section \ref{Sec:propriety} provides bounds on reference prior and integrated likelihood
which are then used to prove that the reference posterior is also proper for models with smoother correlation kernels, including
Exponential kernels with power $\rho=2$,
Matérn kernels with smoothness $\nu \geqslant 1$ and Rational Quadratic kernels.
While Section \ref{Sec:propriety} provides bounds applying to all cases, tighter bounds applying to specific cases only
are obtained in Appendix \ref{App:tail_rates}.
\medskip

\section{Setting} \label{Sec:Setting}
  
\cite{BDOS01} consider Gaussian Process models (also known as Universal Kriging models) with isotropic kernels. 
This article is set in their framework and borrows most of its notations from them.
Define $\| \cdot \|$ as the usual Euclidean norm if applied to a vector and as the Frobenius norm if applied to a matrix. We denote integer intervals by $[\![\cdot,\cdot]\!]$: $[\![1,4]\!]$ for instance is the set $\{1,2,3,4\}$. \medskip

In Universal Kriging, an unknown function from a domain  $\mathcal{D} \subset \R^r$ ($r \in \Z_+$) to $\R$ is assumed to be a realization of a Gaussian process $Y$. The mean function $f$ of the Gaussian process is assumed to belong to some known vector space $\mathcal{F}_p$ of dimension $p \in \N$. If $p$ is non-zero, once a basis $(f_j)_{j \in [\![1,p]\!]}$ of $\mathcal{F}_p$ has been set, $f$ can be parametrized by $\bs{\beta} = (\beta_1,...,\beta_p) \trans \in \R^p$ such that $f = \sum_{j=1}^p \beta_j f_j$.
\medskip

$Y - f$ is assumed to be a stationary Gaussian process
based on an isotropic correlation kernel. 
An isotropic correlation kernel is a function 
$K : [0,+\infty) \rightarrow [-1, 1]$
such that for any positive integer $n$
and any collection of $n$ distinct points
$(\bs{x}^{(i)})_{i \in [\![1,n]\!]}$ within $\mathcal{D}$,
the symmetric $n \times n$ matrix $\bs{\Sigma}$
with $(i,i')$-th element
$K(\| \bs{x}^{(i)}  - \bs{x}^{(i')}\| )$
is a positive definite correlation matrix.
Necessarily, $K(0)=1$.

The covariance function of the Gaussian process $Y$ is $\sigma^2 K_\theta$,
where $K_\theta$ is the correlation kernel parametrized by $\theta \in (0,+\infty)$ and defined by $K_\theta(\cdot) = K(\cdot/\theta)$, making $\sigma^2 \in (0,+\infty)$ the variance of $Y(\bs{x})$ for every $\bs{x} \in \mathcal{D}$. \medskip

Fix $n \in \Z_+$ and fix a collection of $n$ distinct points
$(\bs{x}^{(i)})_{i \in [\![1,n]\!]}$. \medskip

Let this collection be the design set, i.e. the set of points where $Y$ is observed. 
$\left(Y(\bs{x}^{(1)}),...,Y(\bs{x}^{(n)})\right) \trans$ is a Gaussian vector. 
$\left(f(\bs{x}^{(1)}),...,f(\bs{x}^{(n)}) \right) \trans$ is its mean vector and $\sigma^2 \co$ its covariance matrix, with $\co$ being the $n \times n$ matrix with $(i,i')$-th element $K_\theta(\| \bs{x}^{(i)}  - \bs{x}^{(i')}\| )$. Table \ref{Tab:Kernel_families} provides the definition of several correlation kernels. \medskip

\begin{table}[!ht]
\begin{center}
\begin{tabular}{|c|c|c|}
\hline 
\textbf{Kernel} & $K_\theta(t)$ (with $t \in (0,+\infty)$) & parameter range\\ 
\hline 
Spherical & $\left(1 - \frac{3}{2} \left( \frac{t}{\theta} \right) + \frac{1}{2} \left( \frac{t}{\theta} \right)^3 \right) \bs{1}_{ \{ t \leqslant \theta \} }$ & $\emptyset$\\
Power Exponential & $\exp \left\{ - \left( \frac{t}{\theta} \right)^\rho \right\}$ & $\rho \in (0,2]$ \\
Rational Quadratic &  $ \left( 1 + \left( \frac{t}{\theta} \right)^2 \right)^{-\nu}$ & $\nu \in (0,+\infty)$
\\
Matérn & $ \Gamma(\nu)^{-1} 2^{1-\nu} \left( 2 \sqrt{\nu} \frac{t}{\theta} \right)^{\nu} \mathcal{K}_\nu \left( 2 \sqrt{\nu} \frac{t}{\theta} \right)$ & $\nu \in (0,+\infty)$
\\
\hline
\end{tabular} 
\end{center}
\caption{Formulas for several correlation kernel families. The Spherical kernel can only be used if the design space is of dimension 1, 2 or 3, because it is not positive definite for greater dimensions \citep{Ste99}. 
$\mathcal{K}_\nu$ is the modified Bessel function of second kind with parameter $\nu$ \citep[section 9.6]{AS64}. This parametrization of the Matérn family is recommended by \citet{HW94}. To recover the one used by \citet{BDOS01}, simply replace $2 \sqrt{\nu} t$ by $t$.}
\label{Tab:Kernel_families}
\end{table}

If $p$ is non-zero, let $\bs{H}$ denote the $n \times p$ matrix with $(i,j)$-th element $f_j (\bs{x}^{(i)})$.
Then 
$\left(f(\bs{x}^{(1)}),...,f(\bs{x}^{(n)})\right) \trans = \bs{H} \bs{\beta}$. 
If $p=0$, then we adopt the convention that any term involving $\bs{H}$ can be ignored.
Let us denote the observed value of the random vector 
$\left(Y(\bs{x}^{(1)}),...,Y(\bs{x}^{(n)})\right) \trans$ 
by $\bs{y} = (y_1,...,y_n) \trans$. 
The likelihood function of the parameter triplet $(\bs{\beta},\sigma^2,\theta) \trans $ has the following expression:

\begin{equation} \label{Eq:vraisemblance}
L( \bs{y} \; | \; \bs{\beta}, \sigma^2 , \bs{ \theta } ) = 
\left( \frac{ 1 }{ 2 \pi \sigma^2 } \right) ^ { \frac{n}{2} } | \co | ^ {- \frac{1}{2} } \exp \left\{ - \frac{ 1 }{ 2 \sigma^2 } (\bs{y}-\bs{H} \bs{\beta}) \trans \co^{-1} (\bs{y} - \bs{H} \bs{\beta}) \right\} .
\end{equation}

In order for the model to be identifiable, assume that $p<n$ and that $\bs{H}$ has rank $p$. \medskip

Let us recall the general definition of the reference prior and 
how it is derived in this setting. \medskip

For smooth one-dimensional parametric families, the reference prior coincides with the Jeffreys-rule prior \citep{CB94}. For smooth finite-dimensional parametric families, the reference prior algorithm requires the user to define groups of dimensions of the parameter and rank them. The reference prior is then defined iteratively \citep{Ber05}:

\begin{enumerate}
\item Compute the Jeffreys-rule prior on the lowest-ranking group of dimensions conditionally on all others. 
\item Average the likelihood function over this prior. 
\item Compute the Jeffreys-rule prior (based on the integrated likelihood function) on the second-lowest-ranking group of dimensions conditionally on all higher-ranking dimensions. 
\item Average the integrated likelihood function over this second prior. 
\item Continue the process until the Jeffreys-rule prior on the highest-ranking group of dimensions has been computed. 
\item The reference prior is defined as the product of all successively computed priors.
\end{enumerate}

\citet{BDOS01} treat $\vecteur{\beta}$ as the lowest-ranking (possibly multidimensional) parameter
and view $(\sigma^2, \theta)$ as a group of 2 equal-ranking parameters
when they apply the reference prior algorithm.
We denote this ordering by $\vecteur{\beta} \prec (\sigma^2, \theta)$.
This choice is not arbitrary: if $\sigma^2$ and $\theta$ were known,
then the covariance matrix $\sigma^2 \co$ would be known.
Since the model is Gaussian, we would then know which linear transformation
to apply to the observations to make them mutually independent.
In the case of Simple Kriging, $\vecteur{\beta}$ is not used
and their reference prior is simply the Jeffreys-rule prior.
Note that $\theta$ is one-dimensional because all correlation kernels
considered by \citet{BDOS01} - which are also those covered in the present article -
are isotropic. \medskip

\citet{RSH12} show that if we split the parameters further,
with $\vecteur{\beta}$ remaining the lowest-ranking parameter,
$\sigma^2$ being of middle rank
and $\theta$ being the highest-ranking parameter,
the reference prior algorithm yields the same reference prior.
In other words, the reference prior for the ordering $\vecteur{\beta} \prec \sigma^2 \prec \theta$
is the same as the reference prior for the ordering $\vecteur{\beta} \prec (\sigma^2, \theta)$.
It is this prior we consider in the present article. \medskip

To express it conveniently, denote $\bs{Q}_\theta := \bs{I}_n - \bs{H} \left( \bs{H} \trans \co^{-1} \bs{H} \right)^{-1} \bs{H} \trans \co^{-1}$. \medskip

If $p=0$, $\bs{Q}_\theta := \bs{I}_n$. Also fix $\bs{W}$, an $n \times (n-p)$ matrix such that $\bs{W} \trans \bs{W} = \bs{I}_{n-p}$ and $\bs{H} \trans \bs{W}$ is the $p \times (n-p)$ null matrix. The columns of $\bs{W}$ form an orthonormal basis of the orthogonal complement of the subspace of $\R^n$ spanned by the columns of $\bs{H}$.
If $p=0$, fix $\bs{W}$ as an orthogonal matrix, for example $\bs{I}_n$. \medskip

If $p>0$, the matrix $\bs{W}$ can for instance be constructed by computing a Singular Value Decomposition (SVD) of $\bs{H}$. Let us write this decomposition $\bs{H}=\bs{U} \bs{S} \bs{V} \trans $. $\bs{U}$ and $\bs{V}$ are orthogonal matrices of size $n \times n$ and $p \times p$ respectively, and $\bs{S}$ is an $n \times p$ matrix whose only non-null entries are on the main diagonal. Therefore the last $n-p$ rows of $\bs{S}$ are filled with zeros. The matrix $\bs{W}$ can then be defined as the $n \times (n-p)$ matrix formed by the last $n-p$ columns of $\bs{U}$. \medskip

The next two propositions give formulas for the reference prior density and for the integrated likelihood.
The latter is obtained by integrating the likelihood against the reference prior on $\vecteur{\beta}$ conditionally on $\sigma^2$ and $\theta$ and against the reference prior on $\sigma^2$ conditionally on $\theta$.

\begin{prop} \label{Prop:rep_prior}
The reference prior with ordering 
$\bs{\beta} \prec (\sigma^2, \theta)$
or
$\bs{\beta} \prec \sigma^2 \prec \theta$ is $\pi(\bs{\beta}, \sigma^2, \theta) \propto \left(\sigma^2 \right)^{-1} \pi(\theta)$, where

\begin{equation} \label{Eq:Gibbs_ref_prior_theta_berger}
\pi(\theta) \propto 
\sqrt{\Tr \left[ \left\{ \left( \frac{ d }{d \theta }  \bs{ \Sigma }_{ \theta } \right)  \bs{ \Sigma }_{ \theta }^{-1} \bs{Q}_\theta \right\}^2 \right] 
-
\frac{1}{n-p} \left[ \Tr \left\{\left( \frac{ d }{d \theta }  \bs{ \Sigma }_{ \theta } \right)  \bs{ \Sigma }_{ \theta }^{-1} \bs{Q}_\theta \right\}\right]^2 }.
\end{equation}

Denoting $ \bs{W} \trans \co \bs{W}$ by $\bs{\Sigma}_\theta^{\bs{W}}$, $\pi(\theta)$ can also be written as:

\begin{equation} \label{Eq:Gibbs_ref_prior_theta}
\pi(\theta) \propto 
\sqrt{\Tr \left[ \left\{ \left( \frac{ d }{d \theta }  \bs{\Sigma}_\theta^{\bs{W}} \right)  \left( \bs{\Sigma}_\theta^{\bs{W}} \right)^{-1} \right\}^2 \right] 
-
\frac{1}{n-p} \left[ \Tr \left\{\left( \frac{ d }{d \theta }  \bs{\Sigma}_\theta^{\bs{W}} \right) \left( \bs{\Sigma}_\theta^{\bs{W}} \right)^{-1} \right\}\right]^2 }.
\end{equation}
\end{prop}

The fact that \eqref{Eq:Gibbs_ref_prior_theta_berger} is an expression
of the reference prior fitting
the ordering $\vecteur{\beta} \prec (\sigma^2, \theta)$
was first shown in \citet{BDOS01}.
\citet{RSH12} then showed that the reference prior
fitting the ordering $\vecteur{\beta} \prec \sigma^2 \prec \theta$ is the same.
The proof that \eqref{Eq:Gibbs_ref_prior_theta}
is an expression of the same reference prior can be found in Appendix \ref{App:rep_prior}. \medskip

Proposition \ref{Prop:rep_prior} shows that $\pi(\theta)$ is the only untractable
factor in the expression of the reference prior $\pi(\vecteur{\beta}, \sigma^2, \theta)$.
The parameters $\vecteur{\beta}$ and $\sigma^2$ can be marginalized out of 
$L(\vecteur{y} | \vecteur{\beta}, \sigma^2, \theta)$.

\begin{prop} \label{Prop:marginal_likelihood}
If $p \geqslant 1$, after marginalizing $\bs{\beta}$ and $\sigma^2$ out, we have
\begin{align}
\likelihood &:= \iint \allparamlikelihood / \sigma^2 d \bs{\beta} d \sigma^2 \nonumber \\
&= \left( \frac{2 \pi^{\frac{n-p}{2}}}{ \Gamma \left(\frac{n-p}{2}\right)} \right)^{-1} 
\left| \co^{-1} \right|^{\frac{1}{2}} | 
\left| \bs{H} \trans  \co^{-1} \bs{H} \right|^{- \frac{1}{2}} 
\left( \bs{y} \trans \co^{-1} \bs{Q}_\theta \bs{y} \right)^{-\frac{n-p}{2}}.
\end{align}

Alternatively, the integrated likelihood with $p \geqslant 1$ can also be written

\begin{equation} \label{Eq:integrated_likelihood}
\likelihood = \left( \frac{2 \pi^{\frac{n-p}{2}}}{ \Gamma \left(\frac{n-p}{2}\right)} \right)^{-1} 
\left| \bs{H} \trans \bs{H} \right|^{-\frac{1}{2}} 
\left| \bs{W} \trans  \co \bs{W} \right|^{- \frac{1}{2}} 
\left( \bs{y} \trans \bs{W} \left( \bs{W} \trans \co \bs{W} \right)^{-1} \bs{W} \trans \bs{y} \right)^{-\frac{n-p}{2}}.
\end{equation}

If $p=0$, the integrated likelihood is simply

\begin{equation}
\likelihood = \int \nulltrendallparamlikelihood / \sigma^2 d \sigma^2 = \left( \frac{2 \pi^{\frac{n}{2}}}{ \Gamma \left(\frac{n}{2}\right)} \right)^{-1} 
\left| \co^{-1} \right|^{\frac{1}{2}}
\left( \bs{y} \trans \co^{-1} \bs{y} \right)^{-\frac{n}{2}}.
\end{equation}

\end{prop}

The proof of this Proposition can be found in Appendix \ref{App:marginal_likelihood}. \medskip

Proving that the reference posterior is proper amounts to finding appropriate upper bounds
on the tail rates of $\likelihood \pi(\theta)$ as $\theta \to 0$ and as $\theta \to +\infty$.

\section{Scope of the original proof} \label{Sec:smoothness_correlation_kernel}

In \citet{BDOS01}, Lemmas 1 and 2 require that the correlation kernel and design set should be such that
$\bs{\Sigma}_{\theta} = \bs{1} \bs{1} \trans + g_0(\theta) \bs{D} + \bs{R}_0(\theta)$,
where $\bs{1}$ is the vector with $n$ entries all equal to 1,
$g_0(\theta)$ is a real-valued function such that $\lim_{\theta \to +\infty} g_0(\theta) = 0$,
$\bs{D}$ is a fixed nonsingular matrix and $\bs{R}_0$ is a function from $(0,+\infty)$
to the set of $n \times n$ real matrices $\mathcal{M}_n$ such that
$\lim_{\theta \to +\infty} \| \frac{1}{g_0(\theta)} \bs{R}_0(\theta) \| = 0$.
\citet{BDOS01} use this form to derive an asymptotic expansion
of $\co^{-1}$ which involves $\matrice{D}^{-1}$ (see Equations (B.4) and (B.5) of the paper:
they are part of the proof of their Lemma 1, which is then used to prove their Theorem 4
which states that the reference posterior is proper). \medskip

\begin{rmq}
  In fact, Lemma 2 of \citet{BDOS01} has additional hypotheses.
  The original statement specifies, among other things, that
  $\co = \vecteur{1} \vecteur{1} \trans + v(\theta) \matrice{D} 
  + w(\theta) \matrice{D}^\star + \matrice{R}(\theta)$, where $\matrice{D}$ (nonsingular) and 
  $\matrice{D}^\star$ are fixed matrices, $\matrice{R}(\theta)$ is a matrix
  that depends on $\theta$,
  and as $\theta \to +\infty$, 
  $w(\theta)/v(\theta)\to0$ and
  $\| \matrice{R}(\theta) \|/w(\theta) \to 0$.
  Additional assumptions on the derivatives of $v$ and $w$ are also made.
  Informally, the goal of all these assumptions is to make sure
  that when $\theta \to +\infty$,
  for the purposes of their proof,
  $(\vecteur{1} \vecteur{1} \trans + v(\theta) \matrice{D})^{-1}$
  is close enough to $\co^{-1}$ 
  and $v'(\theta) \matrice{D} + w'(\theta) \matrice{D}^\star$
  is close enough to $\frac{d}{d\theta} \co$.
  The framework of the present article does not make the assumption that
  $\matrice{D}$ is nonsingular, which means that an
  asymptotic expansion of $\co$ at higher orders is required
  and that neither $\matrice{D}$ nor $\matrice{D}^\star$ plays any particular role.
\end{rmq}

The following results make it clear that $\matrice{D}$ is singular 
in practically relevant situations.

\begin{prop} \label{Prop:small_order_asymptotic_decomposition}
  For Spherical, Power Exponential, Rational Quadratic and Matérn correlation kernels, the correlation matrix $\co$ can be expressed as
  \begin{equation}
    \co = \vecteur{1} \vecteur{1} \trans + g_0(\theta) \matrice{D}^{(q)} + \matrice{R}_0(\theta),
  \end{equation}
  where:
  \begin{itemize}
    \item $\matrice{D}^{(q)}$ is the matrix with $(i,i')$-th element $ \left\| \bs{x}^{(i)} - \bs{x}^{(i')} \right\|^{2q}$, with $q$ given in Table \ref{Tab:D_non_singularity};
    \item $g_0$ is a real-valued function given in Table \ref{Tab:D_non_singularity};
    \item $\bs{R}_0$ is a function from $(0,+\infty)$ to the set of $n \times n$ real matrices $\mathcal{M}_n$ such that $\lim_{\theta \to +\infty} \| \frac{1}{g_0(\theta)} \bs{R}_0(\theta) \| = 0$.
  \end{itemize}
\end{prop}

The proof of this Proposition can be found in Appendix \ref{App:small_order_asymptotic_decomposition}. \medskip

\begin{table}[!ht]
  \begin{center}
  \begin{tabular}{|c|c|c|}
  \hline 
  \textbf{Kernel} & $g_0(\theta)$ & $q$ \\ 
  \hline 
  Spherical & $ -(3/2) \theta^{-1} $ & $1/2$ \\
  Power Exponential ($\rho \in (0,2]$) & $- \theta^{-\rho}$ & $\rho/2$ \\
  
  Rational Quadratic & $-\nu \theta^{-2}$ & $1$ \\
  Matérn ($\nu<1$) & $ \Gamma(-\nu) \nu^\nu \Gamma(\nu)^{-1} \theta^{-2\nu} $ & $\nu$ \\
  Matérn ($\nu=1$) & $-2 \theta^{-2} \log(\theta)$ & $1$ \\
  Matérn ($\nu>1$) & $- \nu (\nu-1)^{-1} \theta^{-2} $ & $1$ \\\hline
  \end{tabular} 
  \end{center}
  \caption{Details of the second order expansion of 
  $\co$ for various correlation kernels.
  For Spherical, Power Exponential and Rational Quadratic kernels,
  deriving this from Table \ref{Tab:Kernel_families} is straigntforward.
  For Matérn kernels, they follow from lemmas in supplementary material:
  Lemma \ref{Lem:Matern_noninteger_nu_asymptotic_expansion}
  if $\nu$ is not an integer
  and Lemma \ref{Lem:Matern_integer_nu_asymptotic_expansion}
  if it is.}
  \label{Tab:D_non_singularity}
  \end{table}

Theorem 4 from \cite{Sch37} has the following corollary:

\begin{prop} \label{Cor:D_nonsingular}
For $q \in (0,1)$,
the matrix $\matrice{D}^{(q)}$ defined in Proposition \ref{Prop:small_order_asymptotic_decomposition}
is nonsingular.
\end{prop}

The picture is dramatically different when the correlation kernel $K$ is smooth enough to have $q=1$. This happens as soon as $K$ is twice continuously differentiable. 
Theorem 6 from \citet{Gow85} implies the following result:

\begin{prop} \label{Cor:D_singular}
For $q=1$, the matrix $\matrice{D}^{(q)} = \matrice{D}^{(1)}$ defined in Proposition \ref{Prop:small_order_asymptotic_decomposition}
has rank lower or equal to $r+2$.
\end{prop}

If $n \leqslant r+2$, then Proposition \ref{Cor:D_singular} becomes trivial.
For all practical purposes however, $n$ is much greater than $r$ and $\bs{D}^{(1)}$ is singular. \medskip

Thanks to Propositions \ref{Cor:D_nonsingular} and \ref{Cor:D_singular},
Proposition \ref{Prop:small_order_asymptotic_decomposition} yields this result:

\begin{prop} \label{Prop:claim_abstract}
  In the decomposition of $\co$ given in Proposition \ref{Prop:small_order_asymptotic_decomposition}:
  \begin{itemize}
    \item if the correlation kernel is Spherical, Power Exponential with $\rho<2$ or Matérn with $\nu<1$, then $q<1$ and $\matrice{D}^{(q)}$ is nonsingular;
    \item if the correlation kernel is Squared Exponential (i.e. Power Exponential with $\rho=2$), Rational Quadratic or Matérn with $\nu \geqslant 1$, then $q=1$;
    \item if $n > r + 2$ and the correlation kernel is Squared Exponential, Rational Quadratic or Matérn with $\nu \geqslant 1$, then $\matrice{D}^{(q)} = \matrice{D}^{(1)}$ is singular.
  \end{itemize}
\end{prop}

Proposition \ref{Prop:claim_abstract} justifies the claim in the abstract that the Squared Exponential kernel, Matérn kernels with smoothness $\nu \geqslant 1$ and Rational Quadratic kernels require a proof of the reference posterior's propriety. \medskip 

\section{Propriety of the reference posterior distribution} \label{Sec:propriety}

As shown by \citet{BDOS01}, the reference posterior distribution of $\bs{\beta}$ and $\sigma^2$ conditionally on $\theta$ is proper.

In this section, we prove that the joint reference posterior distribution is proper for Matérn kernels with smoothness $\nu \geqslant 1$, Rational Quadratic kernels and the Squared Exponential kernel.

\begin{prop} \label{Prop:ref_prior_behavior}
For Matérn kernels with smoothness $\nu \geqslant 1$, for Rational Quadratic kernels with parameter $\nu>0$ and for the Squared Exponential kernel, 
the ``marginal'' reference prior distribution $\pi(\theta)$ defined by Proposition \ref{Prop:rep_prior} has the following behavior.
\begin{enumerate}
\item  When $\theta \to 0$, 
\begin{equation}
\pi(\theta) =
\left\{
\begin{array}{lr} 
o(1)
& \text{for Matérn and Squared Exponential kernels;}
\\ 
O(\theta^{2\nu-1})
& \text{for Rational Quadratic kernels.}
\end{array}
\right.
\end{equation}
\item  When $\theta \to +\infty$,
\begin{equation}
\pi(\theta) = O(\theta^{-1}).
\end{equation}
\end{enumerate}
\end{prop}

\begin{rmq} \label{Rmq:truncated_prior}
  The second assertion of Proposition \ref{Prop:ref_prior_behavior} is not strong enough to make the reference prior proper
  because $\int_1^{+\infty} \theta^{-1} d \theta = +\infty$.
  The first assertion, however, implies that if we truncated the reference prior at some value $T>0$, only taking $(0,T)$ as its support, the resulting prior would be proper. The choice of $T$ would be very informative, though.
\end{rmq}

The proof of Proposition \ref{Prop:ref_prior_behavior} can be found in Appendix \ref{App:ref_prior_behavior}. \medskip

Figure \ref{Fig:log-prior} shows the tail rate of the reference prior 
for Matérn kernels with smoothness $\nu=1.5$ and $2.5$ 
for the following Kriging models:

\begin{itemize}
  \item Simple Kriging: the mean function is assumed to be null;
  \item Ordinary Kriging: the mean function is assumed to be an unknown constant: $\matrice{H}$ is $\vecteur{1}$, the vector of $\R^n$ filled with ones;
  \item Affine Kriging: the mean function is assumed to be an unknown affine function: $\matrice{H}$ is the $n \times (r+1)$ matrix whose first column is $\vecteur{1}$ and whose last $r$ columns contain the coordinates of the points in the design set.
\end{itemize}

\begin{figure}[ht]
  \begin{center}
    \subfloat[Line: $-\log(\theta)+2.2$]{
      \includegraphics[angle=0,width=0.3\textwidth, height=0.25\linewidth]{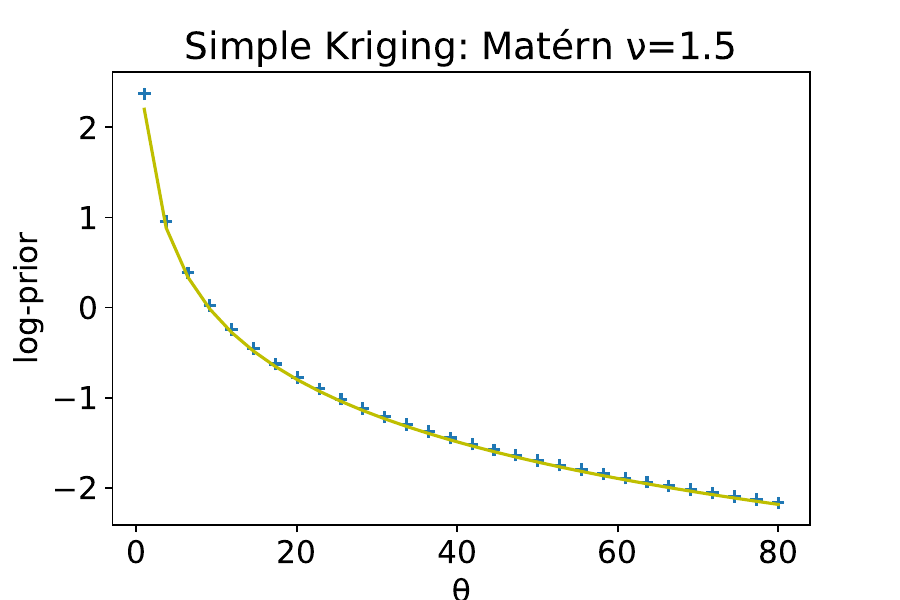}
      
                         }
      \subfloat[Line: $-\log(\theta)+1.1$]{
      \includegraphics[angle=0,width=0.3\textwidth, height=0.25\linewidth]{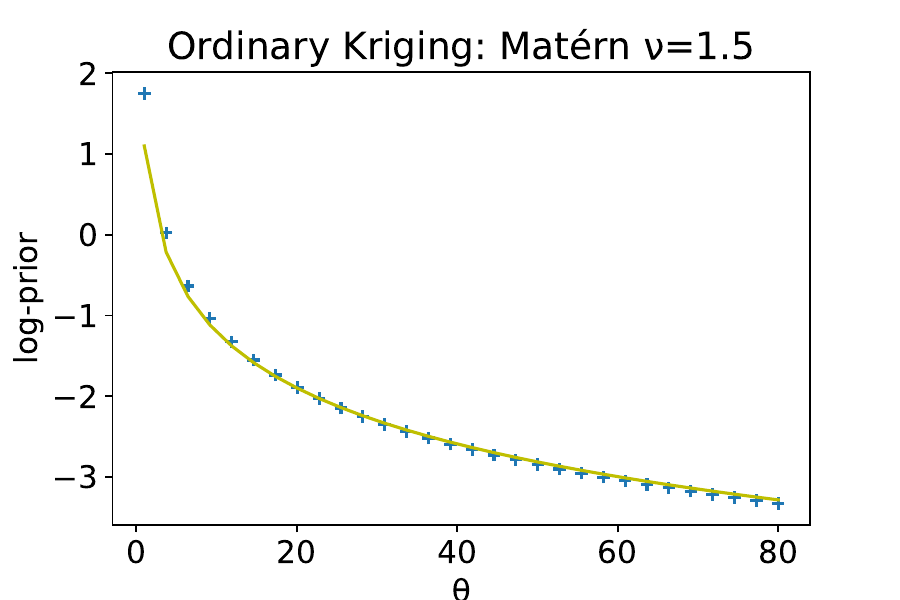}
      \label{Sub:log-prior-b}
                         }
      \subfloat[Line: $-2\log(\theta)+1.1$]{
      \includegraphics[angle=0,width=0.3\textwidth, height=0.25\linewidth]{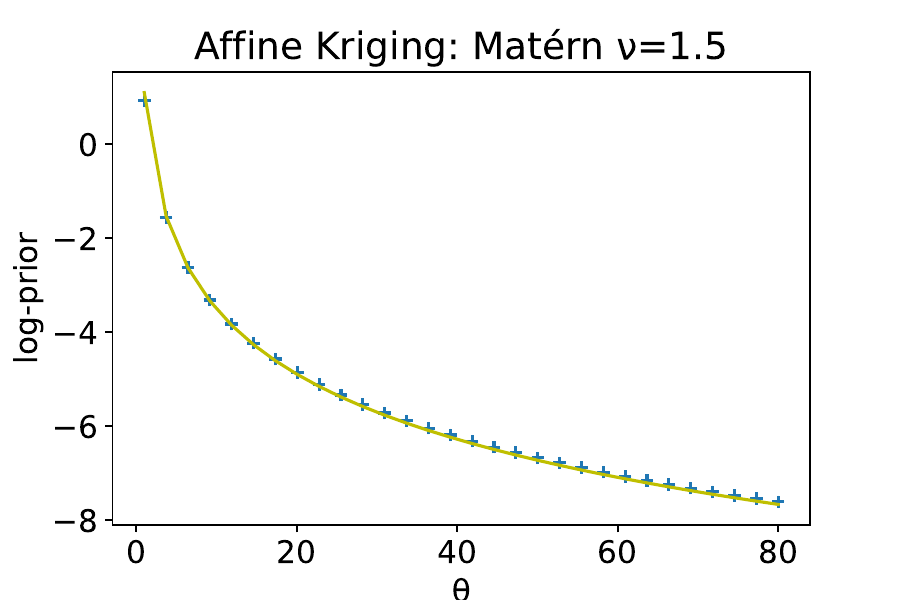}
      \label{Sub:log-prior}
                         }
    \\
    \subfloat[Line: $-\log(\theta)+2.8$]{
      \includegraphics[angle=0,width=0.3\textwidth, height=0.25\linewidth]{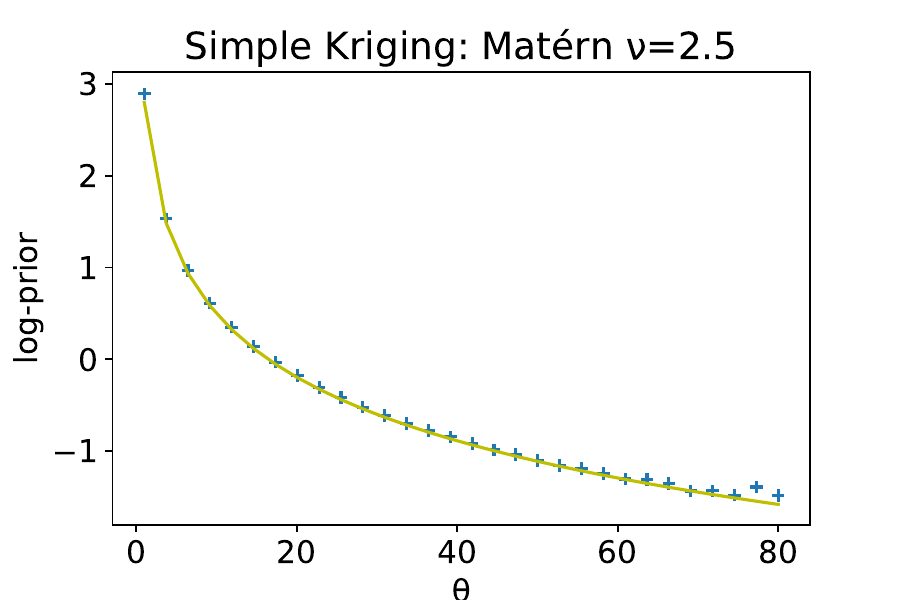}
      
                         }
    \subfloat[Line: $-\log(\theta)+2.2$]{
      \includegraphics[angle=0,width=0.3\textwidth, height=0.25\linewidth]{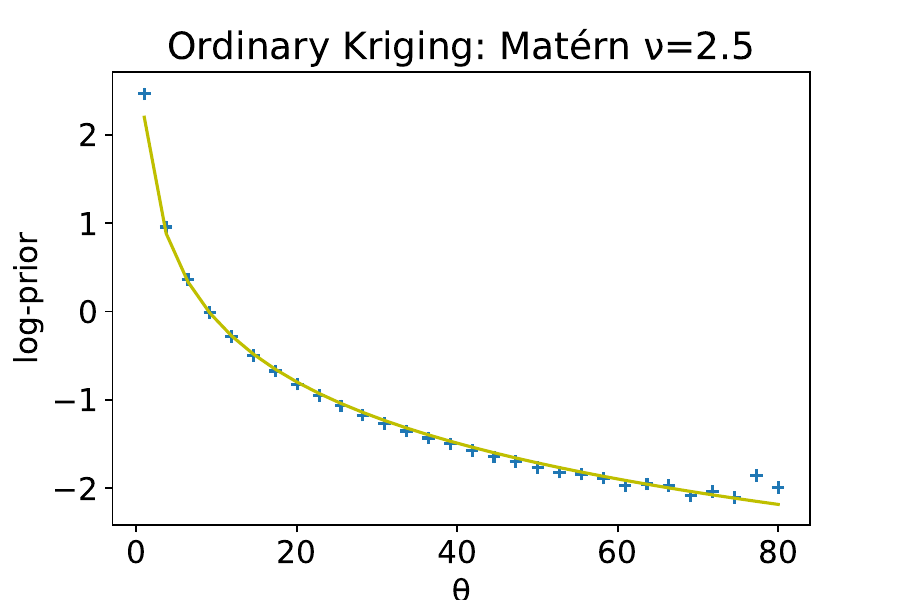}
      \label{Sub:log-prior-e}
                         }   
    \subfloat[Line: $-\log(\theta)+1.0$]{
      \includegraphics[angle=0,width=0.3\textwidth, height=0.25\linewidth]{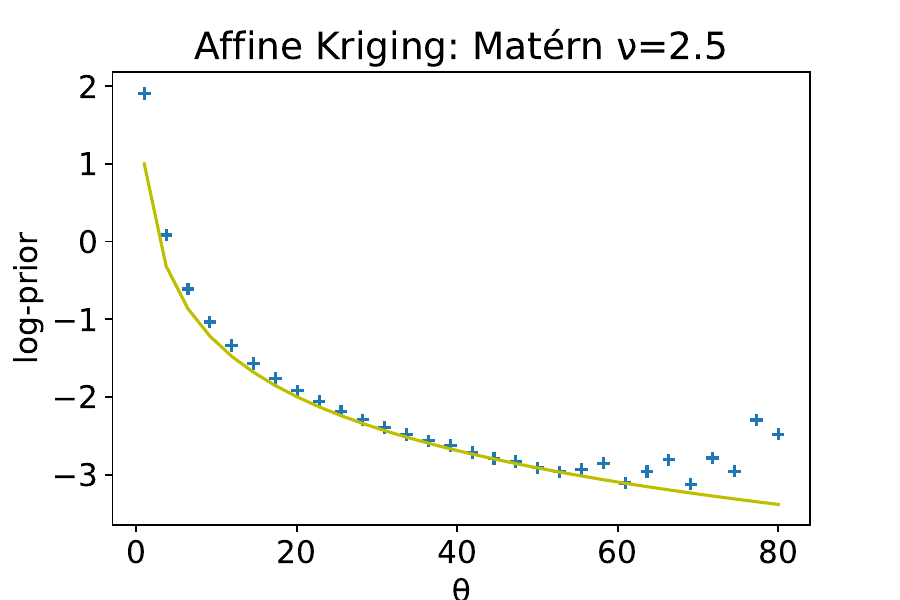}
      \label{Sub:log-prior-f}
                         }     
    \caption{Logarithm of the reference prior density for large values of $\theta$ for varying Kriging models and correlation kernels. 
    In each case, the design space is one-dimensional ($r=1$) and the design set is the 10-point regular grid on $[0,1]$.}
    \label{Fig:log-prior}
  \end{center}
\end{figure}

Figure \ref{Fig:log-prior} is consistent with the bound for tail rates 
given in Proposition \ref{Prop:ref_prior_behavior}. \medskip

For rough kernels (Spherical, Power Exponential with $\rho<2$, 
Matérn with $\nu<1$), \citet{BDOS01} prove that 
if $\vecteur{1}$ belongs to the vector space spanned 
by the columns of $\matrice{H}$ (as happens in Ordinary and 
Affine Kriging), 
then the reference prior is proper.

\medskip

The upper bound provided by Proposition \ref{Prop:ref_prior_behavior}, 
$O(\theta^{-1})$, can be tightened in some cases.
For example,
in the case of Affine Kriging with a Matérn kernel with smoothness $\nu=1.5$ (figure \ref{Sub:log-prior}), the prior is bounded by $O(\theta^{-2})$. See Proposition \ref{Prop:tail_rates_affine_kriging} in  Appendix \ref{App:tail_rates} for a proof. \medskip

The following lemma deals with the asymptotic behavior of the integrated likelihood
when $\theta \to +\infty$. \medskip

Let $v_{1}(\theta) \geqslant ... \geqslant v_{n-p}(\theta) > 0$ be the 
ordered eigenvalues of $\bs{W} \trans \co \bs{W}$.

\begin{lem} \label{Lem:upper_bound_yfactor_in_likelihood}
For Rational Quadratic and Squared Exponential kernels and for Matérn kernels with smoothness $\nu \geqslant 1$, there exists a hyperplane $\mathcal{H}_n$ of $\R^n$ such that for every $\bs{y} \in \R^n \setminus \mathcal{H}_n$, when $\theta \to +\infty$:

\begin{equation} \label{Eq:majoration_norm_inverse_correlation}
\left( \bs{y} \trans \bs{W} \left( \bs{W} \trans \co \bs{W} \right)^{-1} \bs{W} \trans \bs{y} \right)^{-1} = O(v_{n-p}(\theta)). 
\end{equation}

In particular, in the case of Simple Kriging, 
$\matrice{W}$ is the identity matrix and $p=0$:

\begin{equation}
  \left( \bs{y} \trans \co^{-1} \bs{y} \right)^{-1} = O(v_{n}(\theta)). 
\end{equation}

\end{lem}

The proof of this lemma can be found in Appendix \ref{App:upper_bound_yfactor_in_likelihood}. \medskip

For $\vecteur{y} \notin \mathcal{H}_n$,
Lemma \ref{Lem:upper_bound_yfactor_in_likelihood} shows that
the smallest eigenvalue $v_{n-p}(\theta)$ of $\bs{W} \trans \co \bs{W}$
is an asymptotic upper bound for
one of the factors of the squared integrated likelihood $\likelihood^2$ when $\theta \to +\infty$.
\medskip

We need to take the other factors into account
in order to obtain an asymptotic upper bound
of the squared integrated likelihood.
Combined with Proposition \ref{Prop:marginal_likelihood},
Lemma \ref{Lem:upper_bound_yfactor_in_likelihood} implies 
the following result.

\begin{prop} \label{Prop:likelihood_as_product}
If the observation vector $\bs{y}$ belongs to $\R^n \setminus \mathcal{H}_n$
($\mathcal{H}_n$ is defined in Lemma \ref{Lem:upper_bound_yfactor_in_likelihood}), then

\begin{equation} \label{Eq:likelihood_domination}
\likelihood^2 =  \prod_{i=1}^{n-p} \frac{O(v_{n-p}(\theta))}{v_{i}(\theta)} 
\quad \mathit{when} \; \theta \to +\infty.
\end{equation}

\end{prop}

Notice that all factors in the product \eqref{Eq:likelihood_domination} are 
$O(1)$ when $\theta \to +\infty$.
For $\vecteur{y} \notin \mathcal{H}_n$,
Proposition \ref{Prop:likelihood_as_product} thus implies
the existence of a constant asymptotic upper bound
for the likelihood $\likelihood$.
It does not provide any tighter explicit upper bound,
and this bound can only guarantee posterior propriety 
if the prior is itself proper (cf. Remark \ref{Rmq:truncated_prior}).
However, it provides an implicit upper bound
that will be used in proofs of posterior propriety. \medskip

For $\vecteur{y} \notin \mathcal{H}_n$,
when $\theta \to +\infty$, any upper bound of the ratio 
$\frac{v_{n-p}(\theta)}{v_{1}(\theta)}$ is an upper bound 
of the squared likelihood $\likelihood^2$. \medskip

Figure \ref{Fig:likelihood} in Appendix \ref{App:tail_rates}
gives an idea of the explicit tail rates of the marginal log-likelihood in several situations.

\begin{thm} \label{THM:REFERENCE_POSTERIOR_PROPER}
For Rational Quadratic kernels and Squared Exponential kernels and for any Matérn kernel with smoothness $\nu \geqslant 1$, there exists a hyperplane $\mathcal{H}_n$ of $\R^n$ such that if $\bs{y}$ belongs to $\R^n \setminus \mathcal{H}_n$, then the reference posterior distribution $\pi(\theta|\bs{y})$ is proper.
\end{thm}

\begin{rmq}
  The requirement that  $\bs{y}$ should belong to $\R^n \setminus \mathcal{H}_n$ is not substantially restrictive. If $\bs{y}$ is actually sampled from a nondegenerate Gaussian process, it almost surely does not belong to the hyperplane $\mathcal{H}_n$. Therefore, if $\bs{y}$ actually does belong to $\mathcal{H}_n$, it might be better explained by a degenerate Gaussian model. The most compelling example is that of a constant observation vector, for which the Kriging model would be grossly inappropriate.
  In case there is some doubt about whether a given
  $\vecteur{y}$ belongs to $\R^n \setminus \mathcal{H}_n$,
  Appendix \ref{App:Precise_formulation} provides explicit 
  sufficient conditions instead.
\end{rmq}

The proof of the Theorem can be found in Appendix \ref{App:REFERENCE_POSTERIOR_PROPER}.
\medskip

Depending on the matrix $\matrice{H}$ (or its absence in the case of Simple Kriging),
the bounds of the likelihood function may differ. 
In some cases, it is even possible to derive a tighter bound 
for the prior than the general one given in 
Proposition \ref{Prop:ref_prior_behavior}. 
Examples are given in Appendix \ref{App:tail_rates}: 
Affine Kriging with a Matérn kernel with smoothness $\nu \in [1,2)$ 
(Propositions \ref{Prop:tail_rates_affine_kriging} and 
\ref{Prop:tail_rates_affine_kriging_Matern_1}) 
and Ordinary Kriging with a Matérn kernel with smoothness $\nu=1$ 
(Proposition \ref{Prop:tail_rates_ordinary_kriging_Matern_1}).

\section{Conclusion}

In this work, we proved that for a large class of smooth kernels, 
the reference prior leads to a proper posterior distribution.
This class contains the Squared Exponential correlation kernel 
as well as the important Matérn family \citep{Ste99} 
with smoothness parameter $\nu \geqslant 1$. 
The seldom used Rational Quadratic kernels
are also included within this class. \medskip

\cite{BDOS01} proved this result for a class of rough correlation kernels. This class includes the complementary set of the Matérn family -- kernels with smoothness parameter $\nu < 1$ -- as well as all other Power Exponential kernels. Spherical kernels, which are mostly used in the field of geostatistics also belong to this class. \medskip

The results from \cite{BDOS01},
together with Theorem \ref{THM:REFERENCE_POSTERIOR_PROPER},
show how polyvalent the reference prior is, 
insofar as it is able to adapt to very different correlation kernels 
and always leads to a proper posterior.
The key to this flexibility is the way the tail rate of the reference prior
adapts to the tail rate of the integrated likelihood,
which depends on the correlation matrix $\co$
and the trend matrix $\matrice{H}$.
For rough correlation kernels, \citet{BDOS01}
were able to express likelihood tail rates as an explicit function of $\theta$.
More research is needed to do the same for smoother correlation kernels,
even though Appendix \ref{App:tail_rates} provides explicit tail rates in a few specific cases. \medskip

The reference prior's flexilibility means 
no \emph{ad-hoc} technique is required to derive useable inference. 
This makes the approach appealing from a Bayesian point of view 
when no explicit prior information is available. 
Even when explicit prior information is available, following \citet{DM07}, 
the reference prior can be used to derive maximum a posteriori (MAP) estimates or 
High Probability Density (HPD) sets that are invariant under reparametrization.
\medskip

In Section \ref{Sec:smoothness_correlation_kernel}, we recalled that
the original proof from \citet{BDOS01}
relies on the matrix $\matrice{D}^{(q)}$ from the Taylor expansion of the 
correlation matrix being nonsingular
(cf. Proposition \ref{Prop:small_order_asymptotic_decomposition}). 
Provided there are enough observation points, this implies that $q<1$. 
\citet{Ste99} shows that the correlation kernel cannot 
then be twice continously differentiable
at 0 (p. 28, section ``Principal irregular term''). \medskip

Since they make the same assumption, this restriction to rough correlation 
kernels also applies to all 
works generalizing the original result from \citet{BDOS01}.
\citet{GWB18} provide the largest generalization to date:
their setting admits anisotropic kernels defined as products
of one-dimensional kernels and a possible additional noise term (nugget effect).
It does not admit anisotropic geometric kernels, however (see Table \ref{Tab:anisotropic_kernels}
for a definition).
They prove that one of the reference priors leads to a proper posterior:
the prior derived from the reference prior algorithm where $\vecteur{\beta}$
is the lower-ranking parameter and $\sigma^2$,
$\vecteur{\theta}=(\theta_1,...,\theta_r)$ and possibly the parameter controlling the nugget effect
are collectively the higher-ranking group of parameters.
Like \citet{BDOS01}, they assume that the matrix $\matrice{D}^{(q)}$ from
Proposition \ref{Prop:small_order_asymptotic_decomposition}
is nonsingular regardless of $q$, 
so their proof applies to products of rough correlation kernels (Spherical,
Power Exponential with $\rho<2$ and Matérn with $\nu < 1$).
Unfortunately, the proof used in the present article to deal with
smoother kernels (Rational Quadratic, Squared Exponential and Matérn with $\nu \geqslant1$)
cannot easily be adjusted
to their setting. The corresponding reference prior is indeed much more complex as
it is proportional to the square root of the Fisher
information matrix of either $r+1$ parameters ($\sigma^2,\theta_1,...,\theta_r$) or
$r+2$ parameters if there is a nugget effect.

\begin{table}[!ht]
  \begin{center}
  \begin{tabular}{|c|c|c|}
  \hline 
  \textbf{Correlation lengths} & Product &  Anisotropic geometric\\ 
  \hline 
  $\vecteur{\theta} = (\theta_1,...,\theta_r)$
  &
  $\prod_{i=1}^r K_{\theta_i} (t_i)$
  &
  $K_1 \left(\left(\sum_{i=1}^r \frac{t_i^2}{\theta_i^2}\right)^{1/2} \right)$ \\
  \hline
  \end{tabular} 
  \end{center}
  \caption{Anisotropic correlation kernels.
  $K_\theta$ is a 1-dimensional kernel with correlation length $\theta$.}
  \label{Tab:anisotropic_kernels}
\end{table}

Whenever the reference posterior is known to be proper -- whether in the case
of isotropic correlation kernels as shown in the present article
or in the case of a product of rough correlation kernels as shown in \citet{GWB18} --
it is theoretically possible to propagate parameter uncertainty to the predictions
of the Gaussian process model.
This can be done through Markov Chain Monte-Carlo (MCMC) sampling
of the marginal reference posterior distribution on $\vecteur{\theta}$. 
The spread of the predictions obtained using the different values of $\vecteur{\theta}$
account for parameter uncertainty: not only the uncertainty on $\vecteur{\theta}$,
but also on $\vecteur{\beta}$ and $\sigma^2$ since the latter can be marginalized
out of the model (cf. Proposition \ref{Prop:marginal_likelihood}).
In practice though, \citet{GWB18} do not advocate this method
because of the computational cost of MCMC in this setting.
They use the maximum a posteriori (MAP) estimate for $\vecteur{\theta}$ instead.
This effectively means that the density of the reference posterior distribution
acts as a penalization factor on the likelihood.
An alternative proposal is to sample the $\theta_i$ ($1 \leqslant i  \leqslant r$)
from one-dimensional reference 
posterior distributions in order to make MCMC tractable \citep{Mur19}.

\section*{Acknowledgements}
The author would like to thank his PhD advisor Professor Josselin Garnier (École Polytechnique, Centre de Mathématiques Appliquées) for his guidance, Loic Le Gratiet (EDF R\&D, Chatou) and Anne Dutfoy (EDF R\&D, Saclay) for their advice and helpful suggestions.
He also thanks the editor, the associate editor and the referees for
their comments which substantially improved this article.
The author acknowledges the support of the French Agence Nationale de la Recherche (ANR), under grant ANR-13-MONU-0005 (project CHORUS).

\pagebreak

\appendix

\section{Proofs}

The proofs presented in this Appendix rely on auxiliary facts from Appendix \ref{App:auxiliary_facts}. Appendix \ref{App:auxiliary_facts} can be found in supplementary material.

\subsection{Proof of Proposition \ref{Prop:rep_prior}} \label{App:rep_prior}

\begin{proof}
  Both assertions follow from results proved by \citet{RSH12}. 
  The first follows from their Proposition 5 (a) (taken in the particular case with no nugget effect). The second is a consequence of Lemma \ref{Lem:switch_pov}, which restates their Lemma 6: it implies that

  \begin{equation} \label{Eq:conversion_HW}
    \matrice{W} \left( \matrice{W} \trans \co \matrice{W} \right)^{-1} \matrice{W} \trans
    =
    \co^{-1} \left( \matrice{I}_n - \matrice{H} \left( \matrice{H} \trans \co^{-1} \matrice{H} \right)^{-1} \matrice{H} \trans \co^{-1} \right).
  \end{equation}
\end{proof}

\subsection{Proof of Proposition \ref{Prop:marginal_likelihood}} \label{App:marginal_likelihood}

\begin{proof}
  The result for $p=0$ and the first result for $p \geqslant 1$ are from \citet{BDOS01}.
  \medskip
  
  From  \eqref{Eq:conversion_HW}, all that remains to be proved is the determinant equality $\left| \co \right| = \left| \bs{W} \trans  \co \bs{W} \right| \left| \bs{H} \trans  \bs{H} \right|
  \left| \bs{H} \trans  \co^{-1} \bs{H} \right|^{-1}  $. Choose an $n \times p$ matrix $\bs{P}$ with columns forming an orthonormal basis of the $p$-dimensional subspace of $\R^n$ spanned by the columns of $\bs{H}$. Let $(\bs{W} \bs{P})$ denote the matrix whose left $n \times (n-p)$ block is $\bs{W}$ and whose right $n \times p$ block is $\bs{P}$. $(\bs{W} \bs{P})$ is an $n \times n$ orthogonal matrix, so we have $\left| \co \right| = \left| (\bs{W} \bs{P}) \trans \co (\bs{W} \bs{P}) \right|$. Using Schur's complement (see for example \citet{Ser02} p. 139),
  
  \begin{equation} \label{Eq:produit_determinants}
  \left| \co \right| = \left| \bs{W} \trans  \co \bs{W} \right|
  \left| \bs{P} \trans  \co \left( \bs{I}_n - \bs{W} \left( \bs{W} \trans \co \bs{W} \right)^{-1} \bs{W} \trans \co \right) \bs{P} \right|.
  \end{equation}
  
  Equation \eqref{Eq:conversion_HW} is equivalent to:
  
  \begin{equation}
    \co \left( \matrice{I}_n - \matrice{W} \left( \matrice{W} \trans \co \matrice{W} \right)^{-1} \matrice{W} \trans \co \right)
    =
    \matrice{H} \left( \matrice{H} \trans \co^{-1} \matrice{H} \right)^{-1} \matrice{H} \trans.
  \end{equation}
  
  Plugging this in Equation \eqref{Eq:produit_determinants}, we obtain:

  \begin{align}
    \left| \co \right| &= \left| \bs{W} \trans  \co \bs{W} \right|
    \left| \bs{P} \trans  \bs{H} \right|^2
    \left| \bs{H} \trans  \co^{-1} \bs{H} \right| \nonumber \\
    &=
    \left| \bs{W} \trans  \co \bs{W} \right|
    \left| \bs{H} \trans \underbrace{\bs{P} \bs{P} \trans}_{\matrice{I}_p} \bs{H} \right|
    \left| \bs{H} \trans  \co^{-1} \bs{H} \right|.
    \end{align}

  \end{proof}

\subsection{Proof of Proposition \ref{Prop:small_order_asymptotic_decomposition}} \label{App:small_order_asymptotic_decomposition}

\begin{proof}
  Let us consider the $(i,i')$-th element of the matrix $\co$. Letting $K_\theta$ denote the correlation kernel, it is given by $K_\theta \left( \vecteur{x}^{(i)} - \vecteur{x}^{(i')} \right)$. \medskip

  1. With a spherical kernel,

  \begin{equation}
    K_\theta \left( \vecteur{x}^{(i)} - \vecteur{x}^{(i')} \right) := 1 - \frac{3}{2} \frac{\left\| \vecteur{x}^{(i)} - \vecteur{x}^{(i')} \right\|}{\theta} + \frac{1}{2} \frac{\left\| \vecteur{x}^{(i)} - \vecteur{x}^{(i')} \right\|^3}{\theta^3}.
  \end{equation}

  We can identify $g_0(\theta)$ as $-(3/2) \theta^{-1}$, $q$ as $1/2$ and $\matrice{R}_0(\theta)$ as the matrix whose $(i,j)$-th element is $\frac{1}{2} \frac{\left\| \vecteur{x}^{(i)} - \vecteur{x}^{(i')} \right\|^3}{\theta^3}$. \medskip 

  2. With a Power Exponential kernel, 

  \begin{equation}
    K_\theta \left( \vecteur{x}^{(i)} - \vecteur{x}^{(i')} \right) := \exp \left(- \frac{ \left\| \vecteur{x}^{(i)} - \vecteur{x}^{(i')} \right\|^\rho }{\theta^\rho} \right). 
  \end{equation}

    When $\theta \to +\infty$,

    \begin{equation}
      K_\theta \left( \vecteur{x}^{(i)} - \vecteur{x}^{(i')} \right) =  1 - \frac{ \left\| \vecteur{x}^{(i)} - \vecteur{x}^{(i')} \right\|^\rho}{\theta^\rho} + O \left(\theta^{-2\rho}\right). 
    \end{equation}

  We can identify $g_0(\theta)$ as $-\theta^{-\rho}$ and  $q$ as $\rho/2$. Then 
   $\left\|\matrice{R}_0(\theta) \right\| = O(\theta^{-2\rho})$ when $\theta \to \infty$. \medskip

   3. With a Rational Quadratic kernel,

   \begin{equation}
    K_\theta \left( \vecteur{x}^{(i)} - \vecteur{x}^{(i')} \right) := \left( 1 + \frac{ \left\| \vecteur{x}^{(i)} - \vecteur{x}^{(i')} \right\|^2}{\theta^2} \right)^{-\nu}.
   \end{equation}

   When $\theta \to +\infty$,

   \begin{equation}
     K_\theta \left( \vecteur{x}^{(i)} - \vecteur{x}^{(i')} \right) =  1 - \nu \frac{ \left\| \vecteur{x}^{(i)} - \vecteur{x}^{(i')} \right\|^2}{\theta^2} + O(\theta^{-4}).
   \end{equation}

   We can identify $g_0(\theta)$ as $- \nu \theta^{-2}$ and $q$ as $1$. Then 
   $\left\|\matrice{R}_0(\theta) \right\| = O(\theta^{-4})$ when $\theta \to \infty$. \medskip

   4. With a Matérn kernel, we only need to refer to the appropriate decomposition of $\co$ in Appendix \ref{Sec:asymptotic_corr} (Lemma \ref{Lem:Matern_noninteger_nu_asymptotic_expansion} if $\nu$ is no integer and Lemma \ref{Lem:Matern_integer_nu_asymptotic_expansion} if $\nu$ is an integer) to identify $g_0(\theta)$, $\matrice{D}^{(q)}$ and obtain the relevant properties for $\matrice{R}_0(\theta) = \co - \matrice{11} \trans -  g_0(\theta) \matrice{D}^{(q)}$. 

  \begin{itemize}
    \item If $\nu < 1$: 
    \begin{itemize}
     \item $g_0(\theta) 
     = \Gamma(-\nu) \nu^\nu / (\Gamma(\nu) \theta^{2\nu})$;
     \item $q=\nu$;
     \item 
     $\|\matrice{R}(\theta)\| = O(\theta^{-2}) = o(g_0(\theta))$ when $\theta \to \infty$.
    \end{itemize}
    \item If $\nu = 1$:
    \begin{itemize}
      \item $g_0(\theta) 
      = -2 \log(\theta) / \theta^{2}$;
      \item $q=1$;
      \item 
      $\|\matrice{R}(\theta)\| = O(\theta^{-2}) = o(g_0(\theta))$ when $\theta \to \infty$.
    \end{itemize}
    \item If $\nu > 1$: 
    \begin{itemize}
     \item $g_0(\theta) 
     = - \nu (\nu - 1)^{-1}/ \theta^{2}$;
      \item $q=1$;
      \item 
      $\|\matrice{R}(\theta)\| = O   (\theta^{-2 \min(2,\nu)}) = o(g_0(\theta))$ when $\theta \to \infty$.
    \end{itemize}  
  \end{itemize}

\end{proof}

\subsection{Proof of Proposition \ref{Prop:ref_prior_behavior}} \label{App:ref_prior_behavior}

\begin{proof}

  When $\theta \to 0$, $\co$ converges to $\bs{I}_n$, so its inverse does too. 
  Therefore, in order to prove the first assertion, it is enough to prove that $\left\| \de \right\| = o(1)$ for Matérn ($\nu \geqslant 1$) and Squared Exponential kernels and that $\left\| \de \right\| = O(\theta^{2\nu-1})$ for Rational Quadratic kernels. To do this, we prove that these bounds hold for every element of the matrix $\de$. \medskip

  Letting $K_\theta$ be one of the considered kernels, the $(i,i')$-th element of $\de$ is given by $\frac{d}{d \theta} K_\theta \left( \left\| \vecteur{x}^{(i)} - \vecteur{x}^{(i')} \right\| \right)$. \medskip
  
  If $K_\theta$ is Squared Exponential, $\lim_{\theta \to 0} \frac{d}{d \theta} K_\theta \left( \left\| \vecteur{x}^{(i)} - \vecteur{x}^{(i')} \right\| \right) = 0$. This also holds if $K_\theta$ is a Matérn kernel with smoothness $\nu \geqslant 1$ (see \citet{AS64} 9.6.28. and 9.7.2.). 
  \medskip

  If $K_\theta$ is a Rational Quadratic kernel with parameter $\nu>0$, then when $\theta \to 0$, $\frac{d}{d \theta}K_\theta \left( \left\| \vecteur{x}^{(i)} - \vecteur{x}^{(i')} \right\| \right) / \theta^{2\nu-1}$ admits a finite limit. \medskip

  Let us prove the second assertion. \medskip
  
  Lemma \ref{Lem:majoration_derivee_Matern} shows that for Matérn kernels with smoothness $\nu$, for all $\theta>0$ and all $\vecteur{\xi} \in \R^n$:
  
  \begin{equation} 
    0 \leqslant
    \vecteur{\xi} \trans \left( r \theta^{-1} \co - \de \right) \vecteur{\xi}
    \leqslant
    (2\nu+r) \theta^{-1} \vecteur{\xi} \trans \co \vecteur{\xi}.
  \end{equation}
  
  Because of this, Lemma \ref{Lem:ref_prior_majoration} yields an upper bound on the reference prior density:
  
  \begin{equation} 
    \begin{split}
    &\sqrt{\Tr \left[ \left\{ \left( \frac{ d }{d \theta }  \bs{ \Sigma }_{ \theta } \right)  \bs{ \Sigma }_{ \theta }^{-1} \bs{Q}_\theta \right\}^2 \right] 
    -
    \frac{1}{n-p} \left[ \Tr \left\{\left( \frac{ d }{d \theta }  \bs{ \Sigma }_{ \theta } \right)  \bs{ \Sigma }_{ \theta }^{-1} \bs{Q}_\theta \right\}\right]^2 } \\
    \leqslant&
    (n-p) (2\nu+r) \theta^{-1}.
    \end{split}
  \end{equation}
  
  For Squared Exponential and Rational Quadratic kernels, a similar proof is possible. Lemma \ref{Lem:majoration_derivee_RQ_SE} implies that there exists a positive constant $C$ such that for large enough $\theta$, for all $\vecteur{\xi} \in \R^n$,
  
  \begin{equation} 
    0 \leqslant
    \vecteur{\xi} \trans \left( \theta^{-1} \co - \de \right) \vecteur{\xi}
    \leqslant
    C \theta^{-1} \vecteur{\xi} \trans \co \vecteur{\xi}.
  \end{equation}
  
  Like in the Matérn case, Lemma \ref{Lem:ref_prior_majoration} shows that this implies an upper bound on the reference prior density:

  \begin{equation} 
    \begin{split}
    &\sqrt{\Tr \left[ \left\{ \left( \frac{ d }{d \theta }  \bs{ \Sigma }_{ \theta } \right)  \bs{ \Sigma }_{ \theta }^{-1} \bs{Q}_\theta \right\}^2 \right] 
    -
    \frac{1}{n-p} \left[ \Tr \left\{\left( \frac{ d }{d \theta }  \bs{ \Sigma }_{ \theta } \right)  \bs{ \Sigma }_{ \theta }^{-1} \bs{Q}_\theta \right\}\right]^2 } \\
    \leqslant&
    (n-p) C \theta^{-1}.
    \end{split}
  \end{equation}

\end{proof}

\subsection{Proof of Lemma \ref{Lem:upper_bound_yfactor_in_likelihood}} \label{App:upper_bound_yfactor_in_likelihood}

\begin{proof}
The kernel of a matrix $\matrice{M}$ is denoted by $\Ker(\matrice{M})$. \medskip

For Rational Quadratic and Squared Exponential kernels, Lemma \ref{Lem:RQ_SE_asymptotic_expansion} provides an asymptotic expansion of $\co$. When $\theta$ is large enough, 

\begin{equation} \label{Eq:AD_infinitely_differentiable_main_paper}
\co = \sum_{k=0}^\infty \frac{a_k}{\theta^{2k}} \bs{D}^{(k)}.
\end{equation}

In Equation \eqref{Eq:AD_infinitely_differentiable_main_paper}, for every $k$, $\bs{D}^{(k)}$ is the $n \times n$ matrix with $(i,i')$-th element $\| \bs{x}^{(i)} - \bs{x}^{(i')} \|^{2k}$ and $a_k$ is a non-null real number that depends on the kernel.
\medskip

Because $\co$ is nonsingular, the intersection 
$\cap_{k=0}^\infty \Ker \matrice{D}^{(k)}$ is the trivial vector space, i.e. the vector space containing only the null vector. This means there must exist (cf. Lemma \ref{Lem:first_trivial_intersection}) a nonnegative integer $k'$ 
such that the vector space $\cap_{k=0}^{k'} \Ker \left( \bs{W} \trans \bs{D}^{(k)} \bs{W} \right)$ is trivial and such that the vector space $\cap_{0 \leqslant k < k'} \Ker \left( \bs{W} \trans \bs{D}^{(k)} \bs{W} \right)$ is non-trivial (if $k'=0$, the intersection is done over an empty index set, so we take it to be $\R^{n-p}$ by convention). \medskip

This implies (cf. Lemma 
\ref{Lem:minoration_yinversey_Maclaurin}) 
that for any $\vecteur{y}' \in \R^{n-p}$ 
that does not belong to the vector subspace $\mathcal{A}_{k'-1}$ spanned by the columns of the 
matrices $\matrice{W} \trans \matrice{D}^{(k)} \matrice{W}$
($0 \leqslant k \leqslant k'-1$),
there exists $c_{\bs{y}'}>0$ such that for large enough $\theta$,

\begin{equation}
\left(\bs{y}'\right) \trans \left( \bs{W} \trans \co \bs{W} \right)^{-1} \bs{y}' \geqslant c_{\bs{y}'} \left\| \left( \bs{W} \trans \co \bs{W} \right)^{-1} \right\|.
\end{equation}

As a consequence, for every $\bs{y} \in \R^n$ such that 
$\bs{W} \trans \bs{y} \notin \mathcal{A}_{k'-1}$,
there exists $c_{\bs{y}}>0$ such that for large enough $\theta$,
 
\begin{equation}
\bs{y} \trans \bs{W} \left( \bs{W} \trans \co \bs{W} \right)^{-1} \bs{W} \trans \bs{y} \geqslant c_{\bs{y}} \left\| \left( \bs{W} \trans \co \bs{W} \right)^{-1} \right\|.
\end{equation}

Let $\matrice{W} \mathcal{A}_{k'-1}$ denote the vector subspace of $\R^n$
of all vectors $\vecteur{v} \in \R^n$ such that  $\bs{W} \trans \bs{v}$
\emph{does} belong to $\mathcal{A}_{k'-1}$.
Because the matrix $\bs{W} \trans$ has full row rank, 
$\matrice{W} \mathcal{A}_{k'-1}$ is included within a hyperspace
$\mathcal{H}_n$ of $\R^n$. 
Therefore, for every $\bs{y} \in \R^n \setminus \mathcal{H}_n$, 
there exists $c_{\bs{y}}>0$ such that for large $\theta$ 
the equation above holds. \medskip

For Matérn kernels with noninteger smoothness $\nu>0$ (resp. with integer smoothness $\nu>0$), Lemma \ref{Lem:Matern_noninteger_nu_asymptotic_expansion} (resp. Lemma \ref{Lem:Matern_integer_nu_asymptotic_expansion}) allows a similar argument. 

\begin{align} 
  \co &= \sum_{k=0}^{\floor{\nu}} \frac{a_k}{\theta^{2k}} \bs{D}^{(k)} +  \frac{a_\nu}{\theta^{2\nu}} \bs{D}^{(\nu)} +  \bs{R}_\nu(\theta) & \textrm{if } \nu \textrm{ is noninteger.} \label{Eq:AD_Matern_nu_noninteger_simplified}\\
  \co &= \sum_{k=0}^{\nu-1} \frac{a_k}{\theta^{2k}} \bs{D}^{(k)} 
  +  \tilde{a}_\nu \left(  \frac{\log(\theta)}{\theta^{2\nu}} \bs{D}^{(\nu)} +   \frac{1}{\theta^{2\nu}} \bs{\tilde{D}}^{(\nu)} \right) 
  + \bs{\tilde{R}}_\nu(\theta) 
  & \textrm{if } \nu \textrm{ is an integer.} \label{Eq:AD_Matern_nu_integer_simplified}
\end{align}

In these expressions the $a_k$, $a_\nu$ and $\tilde{a}_\nu$ are non-null real numbers, for every $k$, $\bs{D}^{(k)}$ is the $n \times n$ matrix with $(i,i')$-th element $\| \bs{x}^{(i)} - \bs{x}^{(i')} \|^{2k}$, $\bs{D}^{(\nu)}$ is the $n \times n$ matrix with $(i,i')$-th element $\| \bs{x}^{(i)} - \bs{x}^{(i')} \|^{2\nu}$, $\bs{\tilde{D}}^{(\nu)}$ is another non-null symmetric $n \times n$ matrix, and $\matrice{R}_\nu$ (resp.$\matrice{\tilde{R}}_\nu$) is a function  from $(0,+\infty)$ to the space of $n \times n$ matrices such that $\left\|\matrice{R}(\nu)\right\| = o(\theta^{2\nu})$ (resp. $\left\|\matrice{\tilde{R}}(\nu)\right\| = o(\theta^{2\nu})$) when $\theta \to +\infty$. \medskip

With Matérn kernels, when $\theta \to +\infty$, $\left\|\co^{-1}\right\| = O(\theta^{2\nu})$ (cf. Lemma \ref{Lem:Matern_decay}), 
so in the decomposition of $\co$ given by 
Equation \eqref{Eq:AD_Matern_nu_noninteger_simplified} 
(resp. Equation \eqref{Eq:AD_Matern_nu_integer_simplified}), 
the intersection 
$\cap_{k=0}^{\floor{\nu}} \Ker \left( \bs{W} \trans \bs{D}^{(k)} \bs{W} \right) \cap \Ker \left( \bs{W} \trans \bs{D}^{(\nu)} \bs{W} \right)$ (resp. the intersection $\cap_{k=0}^{\nu} \Ker \left( \bs{W} \trans \bs{D}^{(k)} \bs{W} \right) \cap \Ker \left( \bs{W} \trans \bs{\tilde{D}}^{(\nu)} \bs{W} \right)$) is necessarily the trivial vector space. \medskip

The rest of the proof is the same as in the case of Rational Quadratic and Squared Exponential kernels.

\end{proof}

\subsection{Proof of Theorem \ref{THM:REFERENCE_POSTERIOR_PROPER}} \label{App:REFERENCE_POSTERIOR_PROPER}

\begin{proof}

  The first assertion of Proposition \ref{Prop:ref_prior_behavior} implies the reference prior $\pi(\theta)$ is integrable in the neighborhood of 0. Furthermore, when $\theta \to 0$, $\co \to \bs{I}_n$ so the reference posterior $\pi(\theta | \bs{y}) \propto L(\vecteur{y}|\theta) \pi(\theta)$ is integrable in the neighborhood of 0 as well. \medskip
  
  All that remains to be proved is therefore that the reference posterior is integrable in the neighborhood of $+\infty$. In the following $\theta \to +\infty$, so we rely on the asymptotic expansion of $\co$ which is detailed in Appendix \ref{Sec:asymptotic_corr}. \medskip
  
  Let $\mathcal{H}_n$ be the hyperplane of $\R^n$
  defined by Lemma \ref{Lem:upper_bound_yfactor_in_likelihood}.
  Let us fix the observation  vector
  $\vecteur{y} \in \R^n \setminus \mathcal{H}_n$. \medskip
  
  The proof is somewhat trickier for Matérn kernels with integer smoothness, so we tackle this case at the end. Until further notice, assume the kernel is Rational Quadratic, Squared Exponential or Matérn with noninteger smoothness $\nu>1$.
  
  \subsubsection{Rational Quadratic, Squared Exponential and Matérn kernels with noninteger smoothness \texorpdfstring{$\nu>1$}{Lg}} \label{App:RQESMatern_noninteger_soothness}
  
  For Rational Quadratic and Squared Exponential (resp. Matérn with noninteger smoothness parameter $\nu>1$) kernels, Lemma \ref{Lem:RQ_SE_framework_proof_theorem} 
  (resp. Lemma \ref{Lem:Matern_noninteger_smoothness_framework_proof_theorem}) 
  shows how $\bs{W} \trans \co \bs{W}$ can be decomposed as
  
  \begin{equation} \label{Eq:correlation_decomposition}
  \bs{W} \trans \co \bs{W} = g(\theta) \left( \bs{W} \trans \nouveauD \bs{W} + g^{\star}(\theta) \bs{W} \trans \nouveauD^{\star} \bs{W} + \bs{R}_g(\theta) \right),
  \end{equation}
  
  where:
  
  \begin{itemize}
  \item $g$ is a positive differentiable function on $(0,+\infty)$;
  \item $g^{\star}(\theta) = \theta^{-2l}$ with $l \in (0,+\infty)$ (actually, if the kernel is Rational Quadratic or Squared Exponential, $l \in \Z_+$);
  \item $\bs{R}_g$ is a differentiable function 
  from $(0,+\infty)$ to $\mathcal{M}_n$ 
  such that $\| \bs{R}_g(\theta) \| = o(g^{\star}(\theta))$ 
  and $\| \frac{d}{d\theta} \bs{R}_g(\theta) \| = o(g^{\star \prime}(\theta))$;
  \item $\nouveauD$ and $\nouveauD^{\star}$ are both fixed symmetric matrices;
  \item $\bs{W} \trans \nouveauD \bs{W}$ is non-null.
  \end{itemize}
  
  \begin{rmq}
    Readers familiar with \citet{BDOS01} may recognize similarities
    with the assumptions in Lemma 2 of the paper. This is not coincidental. 
    The matrices $\nouveauD$ and $\nouveauD^\star$
    essentially play the roles of the matrices $\matrice{D}$ and
    $\matrice{D}^\star$ respectively.
    However, $\nouveauD$ (resp.  $\nouveauD^\star$) is not 
    necessarily equal to $\matrice{D}$ (resp. $\matrice{D}^\star$). 
    In fact, 
    in the Simple Kriging case where $\matrice{W}$ is the identity
    matrix, $\nouveauD = \matrice{1} \matrice{1} \trans$.
  \end{rmq}

  Let us differentiate $\bs{W} \trans \co \bs{W}$:
  
  \begin{equation} \label{Eq:correlation_derivative_decomposition}
  \frac{d}{d\theta} \bs{W} \trans \co \bs{W} = 
  \frac{g'(\theta)}{g(\theta)} \bs{W} \trans \co \bs{W} + g(\theta) \left( g^{\star \prime} (\theta) \bs{W} \trans \nouveauD^{\star} \bs{W}  +  \frac{d}{d\theta} \bs{R}_g(\theta) \right).
  \end{equation}
  
  This decomposition of the matrix $\frac{d}{d\theta} \bs{W} \trans \co \bs{W}$ 
  implies (cf. Lemma \ref{Lem:sd_eigenvalues}) that 
  it can be replaced in Equation (\ref{Eq:Gibbs_ref_prior_theta}) by 
  $g(\theta) \left( g^{\star \prime} (\theta) \bs{W} \trans \nouveauD^{\star} \bs{W}  +  \frac{d}{d\theta} \bs{R}_g(\theta) \right)$:
  
  \begin{equation}
    \begin{split}
    & \Tr \left[ \left\{ \left( 
    \frac{g'(\theta)}{g(\theta)} \matrice{W} \trans \co \matrice{W} 
    + g(\theta) 
    \left( g^{\star \prime} (\theta) \bs{W} \trans \nouveauD^{\star} \bs{W}  +  \frac{d}{d\theta} \bs{R}_g(\theta) \right)
     \right)  \left(\bs{W} \trans \bs{\Sigma}_\theta \bs{W} \right)^{-1} \right\}^2 \right] \\
    &-
    \frac{1}{n-p} \left[ \Tr \left\{\left( 
    \frac{g'(\theta)}{g(\theta)} \matrice{W} \trans \co \matrice{W} + g(\theta) \left( g^{\star \prime} (\theta) \bs{W} \trans \nouveauD^{\star} \bs{W}  +  \frac{d}{d\theta} \bs{R}_g(\theta) \right)  
    \right) \left(\bs{W} \trans \bs{\Sigma}_\theta \bs{W} \right)^{-1} \right\}\right]^2 \\
    = &
     \Tr \left[ \left\{ g(\theta) \left( g^{\star \prime} (\theta) \bs{W} \trans \nouveauD^{\star} \bs{W}  +  \frac{d}{d\theta} \bs{R}_g(\theta) \right)  \left( \bs{W} \trans \bs{\Sigma}_\theta \bs{W} \right)^{-1} \right\}^2 \right] \\
    & -
    \frac{1}{n-p} \left[ \Tr \left\{g(\theta) \left( g^{\star \prime} (\theta) \bs{W} \trans \nouveauD^{\star} \bs{W}  +  \frac{d}{d\theta} \bs{R}_g(\theta) \right) \left(\bs{W} \trans \bs{\Sigma}_\theta \bs{W} \right)^{-1} \right\}\right]^2 \\
  \end{split}
  \end{equation}

  So $\pi(\theta) \propto w(\theta)$, where 
  
  \begin{equation}
  \begin{split}
  w(\theta)^2:=& 
  \Tr \left[ \left\{ g(\theta) \left( g^{\star \prime} (\theta) \bs{W} \trans \nouveauD^{\star} \bs{W}  +  \frac{d}{d\theta} \bs{R}_g(\theta) \right)  \left( \bs{W} \trans \bs{\Sigma}_\theta \bs{W} \right)^{-1} \right\}^2 \right] \\
  & \quad -
  \frac{1}{n-p} \left[ \Tr \left\{g(\theta) \left( g^{\star \prime} (\theta) \bs{W} \trans \nouveauD^{\star} \bs{W}  +  \frac{d}{d\theta} \bs{R}_g(\theta) \right) \left(\bs{W} \trans \bs{\Sigma}_\theta \bs{W} \right)^{-1} \right\}\right]^2 .
  \end{split}
  \end{equation}
  
  We have $w(\theta) \leqslant \tilde{w}(\theta)$, where
  
  \begin{equation} \label{Eq:ersatz_prior}
  \tilde{w}(\theta):= \sqrt{
  \Tr \left[ \left\{ g(\theta) \left( g^{\star \prime} (\theta) \bs{W} \trans \nouveauD^{\star} \bs{W}  +  \frac{d}{d\theta} \bs{R}_g(\theta) \right)  \left( \bs{W} \trans \bs{\Sigma}_\theta \bs{W} \right)^{-1} \right\}^2 \right] }.
  \end{equation}
  
  If $\bs{W} \trans \nouveauD \bs{W}$ is nonsingular, then 
  $\tilde{w}(\theta) = O(g^{\star \prime} (\theta))$.
  This implies $\pi(\theta) = O(g^{\star \prime} (\theta)) = O(\theta^{-2l-1})$, so the reference prior is proper.
  The likelihood function is bounded 
  due to Proposition \ref{Prop:likelihood_as_product},
  so the reference posterior is proper.
  
  \begin{rmq}
    Recall the decomposition of $\co$ from either 
    Proposition \ref{Prop:small_order_asymptotic_decomposition}
    or Appendix \ref{App:upper_bound_yfactor_in_likelihood}.
    If the vector $\vecteur{1}$ is one of the columns of $\matrice{H}$,
    then we have 
    $\matrice{W} \trans \matrice{D}^{(0)} \matrice{W}
    = \matrice{W} \trans \matrice{11} \trans \matrice{W} 
    = \matrice{0}$. 
    \citet{BDOS01} assume that the matrix $\matrice{D}^{(1)}$
    is necessarily nonsingular, which implies that
    $\matrice{W} \trans \matrice{D}^{(1)} \matrice{W}$ is non-null
    (and thus equal to $\matrice{W} \trans \matrice{Z} \matrice{W}$)
    and even nonsingular, so the paragraph above is applicable.
    This is why they reach the conclusion
    that the reference prior is proper as soon as 
    $\vecteur{1}$ is one of the columns of $\matrice{H}$
    (denoted by $X$ in their article).
    Because the underlying assumption that $\matrice{D}^{(1)}$
    is nonsingular
    does not generally hold (cf. Proposition \ref{Prop:claim_abstract}),
    there is reason to doubt the conclusion.
    Indeed, Figures \ref{Sub:log-prior-b}, \ref{Sub:log-prior-e} and
    \ref{Sub:log-prior-f} do not seem to support
    the claim.
  \end{rmq}

  If $\bs{W} \trans \nouveauD \bs{W}$ is singular, 
  Proposition \ref{Prop:ref_prior_behavior} still ensures that 
  $\pi(\theta) = O(\theta^{-1})$. 
  Moreover, for any non-null vector $\vecteur{\xi} \in \R^{n-p}$
  that belongs to the kernel of $\matrice{W} \trans \nouveauD \matrice{W}$,
  $\vecteur{\xi} \trans \matrice{W} \trans \co \matrice{W} \vecteur{\xi}
  =
  O(g(\theta) g^\star(\theta))$.
  A fortiori,
  $v_{n-p}(\theta) = O(g(\theta) g^{\star}(\theta))$. 
  As the rank of $\bs{W} \trans \nouveauD \bs{W}$ is at least one, 
  $v_1(\theta)^{-1}=O(g(\theta)^{-1})$. 
  Gathering this, $v_{n-p}(\theta) / v_{1}(\theta) = O(g^{\star}(\theta))$.
  Proposition \ref{Prop:likelihood_as_product} implies that
  $\likelihood = O(g^{\star}(\theta))^{1/2}) = O(\theta^{-l})$. 
  The reference posterior is proportional to 
  $\likelihood \pi(\theta) = O(\theta^{-l-1})$ and is proper. \medskip
  
  \subsubsection{Matérn kernels with integer smoothness \texorpdfstring{$\nu$}{Lg}} \label{App:Matern_integer_soothness}
  
  We now address the case where the correlation kernel is Matérn with integer smoothness $\nu$. The proof strategy remains the same as for the other kernels, but the execution is a little trickier. 
  \medskip
  
  $\matrice{D}^{(\nu)}$ is the $n \times n$ matrix with $(i,i')$-th element $\| \bs{x}^{(i)} - \bs{x}^{(i')} \|^{2\nu} $. Let $\bs{\tilde{D}}^{(\nu)}$ denote the  $n \times n$ matrix with null diagonal and $(i,i')$-th element ($i\neq i'$) given by 
  $$\| \bs{x}^{(i)} - \bs{x}^{(i')} \|^{2\nu} \left\{-\log \left( \| \bs{x}^{(i)} - \bs{x}^{(i')} \| \right) - \frac{\log(\nu)}{2} - \gamma + \sum_{l=1}^\nu \frac{1}{2 l} \right\} ,$$ where $\gamma$ is Euler's constant.
  \medskip

  Both $\matrice{D}^{(\nu)}$ and $\bs{\tilde{D}}^{(\nu)}$ can appear in the decomposition of $\bs{W} \trans \co \bs{W}$ provided by Lemma \ref{Lem:Matern_integer_smoothness_framework_proof_theorem}:

  \begin{equation} 
    \bs{W} \trans \co \bs{W} = g(\theta) \left( \bs{W} \trans \nouveauD \bs{W} + g^{\star}(\theta) \bs{W} \trans \nouveauD^{\star} \bs{W} + \bs{R}_g(\theta) \right),
  \end{equation}
    
  where:
    
  \begin{itemize}
  \item $g$ is a positive differentiable function on $(1,+\infty)$;

  \item $\nouveauD$ and $\nouveauD^{\star}$ are both fixed symmetric matrices;
  \item $\bs{W} \trans \nouveauD \bs{W}$ is non-null;
  \item $g^{\star}(\theta) = \log(\theta)^{-1}$ if there exist non-null real numbers $\lambda, \lambda^\star$ such that $\nouveauD= \lambda \matrice{D}^{(\nu)}$ and $\nouveauD^\star = \lambda^\star \matrice{\tilde{D}}^{(\nu)}$;
  \item  $g^{\star}(\theta) = \theta^{-2l}$ or $g^{\star}(\theta) = \log(\theta)\theta^{-2l}$ with $l \in (0,+\infty)$ otherwise;
  \item $\bs{R}_g$ is a differentiable function from $(0,+\infty)$ to $\mathcal{M}_n$ such that $\| \bs{R}_g(\theta) \| = o(g^\star(\theta))$ and $\| \frac{d}{d\theta} \bs{R}_g(\theta) \| = o(g^{\star \prime}(\theta))$ when $\theta \to +\infty$.
  \end{itemize}
  
  First, assume either that for all $\lambda \neq 0$, $\nouveauD \neq \lambda \bs{D}^{(\nu)}$ \emph{or} that for all $\lambda^\star \neq 0$, $\nouveauD^{\star} \neq \lambda^\star \tilde{\bs{D}}^{(\nu)}$. 
  
  In Equation (\ref{Eq:correlation_decomposition}), according to Lemma \ref{Lem:Matern_integer_smoothness_framework_proof_theorem}, $g^{\star}(\theta)$ may be $\theta^{-2l} \log(\theta)$ 
  instead of $\theta^{-2l}$. If $g^{\star}(\theta) = \theta^{-2l}$ for some $l \in (0,+\infty)$, then the proof is the same as for Rational Quadratic, Squared Exponential and Matérn kernels with noninteger $\nu$. Assume therefore that $g^{\star}(\theta) = \theta^{-2l} \log(\theta)$ for some $l \in (0,+\infty)$. 
  Then its derivative is $g^{\star \prime}(\theta) = \theta^{-2l-1}(1-2l \log(\theta))$.

  If $\bs{W} \trans \nouveauD \bs{W}$ is nonsingular, 
  then the reference prior distribution is proper since 
  $\pi(\theta) = O(g^{\star \prime}(\theta)) = O(\theta^{-2l-1} \log(\theta))$. 
  Proposition \ref{Prop:likelihood_as_product} guarantees that 
  the likelihood function is bounded and therefore 
  that the reference posterior is proper.
  
  If $\bs{W} \trans \nouveauD \bs{W}$ is singular, 
  Proposition \ref{Prop:ref_prior_behavior} still ensures that 
  the reference prior is $O(\theta^{-1})$.  
  Given the rank of $\bs{W} \trans \nouveauD \bs{W}$ is at least one, 
  $v_{n-p}(\theta) / v_{1}(\theta) = O(g^{\star}(\theta))$. 
  Proposition \ref{Prop:likelihood_as_product}
  implies that
  $\likelihood = O(g^{\star}(\theta)^{1/2})= O(\theta^{-l} \log(\theta)^{1/2})$, so the reference posterior is proportional to $\likelihood \pi(\theta) = O(\theta^{-l-1} \log(\theta)^{1/2})$ and thus proper. \medskip
  
  Now, assume there exist non-null real numbers $\lambda, \lambda^\star$ such that $\nouveauD = \lambda \bs{D}^{(\nu)}$ \emph{and} $\nouveauD^{\star} = \lambda^{\star} \tilde{\bs{D}}^{(\nu)}$. 
  
  In that case, according to Lemma \ref{Lem:Matern_integer_smoothness_framework_proof_theorem}, $g^{\star}(\theta) = \log(\theta)^{-1}$. Its derivative is $g^{\star \prime}(\theta) =-\theta^{-1} \log(\theta)^{-2}$. 
  
  If $\bs{W} \trans \nouveauD \bs{W}$ is nonsingular, 
  the reference prior is proper since
  $\pi(\theta) = O(g^{\star \prime}(\theta)) = O(\theta^{-1} \log(\theta)^{-2})$.
  Proposition \ref{Prop:likelihood_as_product}
  implies that 
  the likelihood function is bounded and therefore 
  that the reference posterior is proper.
  
  If $\bs{W} \trans \nouveauD \bs{W}$ is singular, 
  it nevertheless turns out that for large enough $\theta$, 
  $\bs{W} \trans \nouveauD \bs{W} + g^\star(\theta) \bs{W} \trans \nouveauD^\star \bs{W}$ 
  is nonsingular. 
  This is due to Lemma \ref{Lem:Matern_decay}, 
  which asserts that $\left\|\co^{-1} \right\|= O(\theta^{2\nu})$. 
  The reference prior is then 
  $O(g^{\star \prime}(\theta) g^{\star}(\theta)^{-1}) = O(\theta^{-1} \log(\theta)^{-1})$. 
  Besides, as the rank of $\bs{W} \trans \nouveauD \bs{W}$ is 
  at least one, $1/v_1(\theta) = O(g(\theta)^{-1})$ and therefore 
  $v_{n-p}(\theta) / v_{1}(\theta) = O(g^{\star}(\theta))$. 
  Proposition \ref{Prop:likelihood_as_product}
  implies that 
  $\likelihood = O(g^{\star}(\theta)^{1/2}) =  O(\log(\theta)^{-1/2})$.
  The reference posterior is then proportional to 
  $\likelihood \pi(\theta) = O(\theta^{-1} \log(\theta)^{-3/2})$
  and is proper.
  
\end{proof}

\subsection{A more precise formulation of Lemma \ref{Lem:upper_bound_yfactor_in_likelihood},
Proposition \ref{Prop:likelihood_as_product} and Theorem \ref{THM:REFERENCE_POSTERIOR_PROPER}} \label{App:Precise_formulation}

The proof of Lemma \ref{Lem:upper_bound_yfactor_in_likelihood}
actually proves a slightly stronger result,
which we provide in this section.
This stronger result in turn leads to slightly stronger
versions of 
Proposition \ref{Prop:likelihood_as_product} and
Theorem \ref{THM:REFERENCE_POSTERIOR_PROPER}. \medskip

In order to be able to state this result,
we must use the notations of the proof of Lemma
\ref{Lem:upper_bound_yfactor_in_likelihood},
together with additional definitions:

\begin{defn}
  For any vector $\vecteur{y} \in \R^n$ and
  for any nonnegative real number $t$, 
  let $\mathbb{A}_t(\vecteur{y})$ be the following statement:
  \begin{quote}
    $\matrice{W} \trans \vecteur{y}$ does not belong to the vector
    subspace of $\R^{n-p}$ spanned by the columns of the matrices
    $\matrice{W} \trans \matrice{D}^{(k)} \matrice{W}$ with nonnegative
    interger $k$ strictly smaller than $t$ (resp. 
    $\matrice{W} \trans \vecteur{y} \neq \vecteur{0}$ if $t=0$). However, the vector subspace
    of $\R^{n-p}$ spanned by these columns and the columns of 
    $\matrice{W} \trans \matrice{D}^{(t)} \matrice{W}$ (resp. by the
    columns of $\matrice{W} \trans \matrice{D}^{(0)} \matrice{W}$ if $t=0$) is $\R^{n-p}$ itself.
  \end{quote}
\end{defn}
  
\begin{defn}
  For any vector $\vecteur{y} \in \R^n$ and
  for any positive integer $\nu$, 
  let $\mathbb{\tilde{A}}_\nu(\vecteur{y})$ be the following statement:
  \begin{quote}
    $\matrice{W} \trans \vecteur{y}$ does not belong to the vector
    subspace of $\R^{n-p}$ spanned by the columns of the matrices
    $\matrice{W} \trans \matrice{D}^{(k)} \matrice{W}$ with nonnegative
    interger $k$ smaller or equal to $\nu$. However, the vector subspace
    of $\R^{n-p}$ spanned by these columns and the columns of 
    $\matrice{W} \trans \matrice{\tilde{D}}^{(\nu)} \matrice{W}$ is $\R^{n-p}$ itself.
  \end{quote}
\end{defn}

\begin{rmq}
  In the case of Simple Kriging, in both definitions, $p=0$
  and the matrix $\matrice{W}$ is the identity $n \times n$ matrix.
\end{rmq}

The more precise version of Lemma \ref{Lem:upper_bound_yfactor_in_likelihood} is:

\begin{lem} \label{Lem:upper_bound_yfactor_in_likelihood_details}
  Depending on the correlation kernel, 
  the following condition on $\vecteur{y} \in \R^n$
  is sufficient 
  for Equation \eqref{Eq:majoration_norm_inverse_correlation}
  when $\theta \to +\infty$:
  \begin{itemize}
    \item \emph{Rational Quadratic and Squared Exponential kernels:}
    there exists a nonnegative integer $k'$ such that Assumption
    $\mathbb{A}_{k'}(\vecteur{y})$ holds;
    \item \emph{Matérn kernels with noninteger smoothness $\nu>1$:}
    either there exists a nonnegative integer $k'<\nu$ such that Assumption
    $\mathbb{A}_{k'}(\vecteur{y})$ holds or 
    Assumption $\mathbb{A}_\nu(\vecteur{y})$ holds;
    \item \emph{Matérn kernels with integer smoothness $\nu \geqslant 1$:}
    either there exists a nonnegative integer $k' \leqslant \nu$ such that 
    Assumption $\mathbb{A}_{k'}(\vecteur{y})$ holds 
    or Assumption $\mathbb{\tilde{A}}_\nu(\vecteur{y})$ holds.
  \end{itemize}
  Regardless of whether the kernel is Rational Quadratic,
  Squared Exponential,
  or Matérn (with integer or noninteger smoothness $\nu \geqslant 1$),
  the set of all $\vecteur{y} \in \R^n$ that do not satisfy this
  sufficient condition is a vector subspace of $\R^n$
  of dimension smaller or equal to $n-1$.
\end{lem}

\begin{proof}
  The proof of Lemma \ref{Lem:upper_bound_yfactor_in_likelihood}
  also proves this result.
\end{proof}

The condition on $\vecteur{y} \in \R^n$ 
stated in Lemma \ref{Lem:upper_bound_yfactor_in_likelihood_details}
replaces the condition from Lemma \ref{Lem:upper_bound_yfactor_in_likelihood}
about $\vecteur{y}$ not belonging to $\mathcal{H}_n$. \medskip

Proposition \ref{Prop:likelihood_as_product} can thus be replaced by the following proposition.

\begin{prop} \label{Prop:likelihood_as_product_details}
  The condition 
  on $\vecteur{y} \in \R^n$
  stated
  in Lemma \ref{Lem:upper_bound_yfactor_in_likelihood_details} 
  is sufficient
  for Equation \eqref{Eq:likelihood_domination} to hold when $\theta \to +\infty$.    
\end{prop}

This more precise version of 
Proposition \ref{Prop:likelihood_as_product}
leads to a more precise version of Theorem 
\ref{THM:REFERENCE_POSTERIOR_PROPER}:
  
\begin{thm} \label{THM:REFERENCE_POSTERIOR_PROPER_details}
  The condition 
  on $\vecteur{y} \in \R^n$
  stated
  in Lemma \ref{Lem:upper_bound_yfactor_in_likelihood_details} 
  is sufficient 
  for the reference posterior distribution $\pi(\theta|\vecteur{y})$
  to be proper. 
  The set of all $\vecteur{y} \in \R^n$ that do not satisfy this
  sufficient condition is a vector subspace of $\R^n$
  of dimension smaller or equal to $n-1$.
\end{thm}

\begin{proof}
  The proof of Theorem \ref{THM:REFERENCE_POSTERIOR_PROPER}
  can be used to prove this result, provided
  Proposition \ref{Prop:likelihood_as_product_details}
  is used instead of Proposition
  \ref{Prop:likelihood_as_product}.
\end{proof}

\section{Some tail rates of likelihood and prior} \label{App:tail_rates}

The purpose of this appendix is twofold. First, to list examples that show how the tail rate of the reference prior density varies to accomodate the various tail rates of the likelihood function while making sure the reference posterior is always proper. Second, to show that adressing the various cases considered in the proof of Theorem \ref{THM:REFERENCE_POSTERIOR_PROPER}, Appendix \ref{App:REFERENCE_POSTERIOR_PROPER}, is not merely necessary for the sake of mathematical rigor, but because these cases do occur in practice. The proofs of the results presented in this Appendix rely on auxiliary facts from Appendix \ref{App:auxiliary_facts}. Appendix \ref{App:auxiliary_facts} can be found in supplementary material.\medskip

Figure \ref{Fig:likelihood} gives a sample of the wide variety of tail rates for the likelihood function depending on the Kriging model and the smoothness of the correlation kernel.
Among the cases considered in Figure \ref{Fig:likelihood}, 
the most remarkable is \ref{sub:likelihood_O1}.  
Affine Kriging 
with a Matérn kernel with smoothness $\nu=3/2$
leads to a likelihood function that does not vanish
when $\theta \to +\infty$. \medskip

\begin{figure}[ht]
  \begin{center}
    \subfloat[Line: $-2 \log(\theta)+13.5$]{
      \includegraphics[angle=0,width=0.3\textwidth, height=0.25\linewidth]{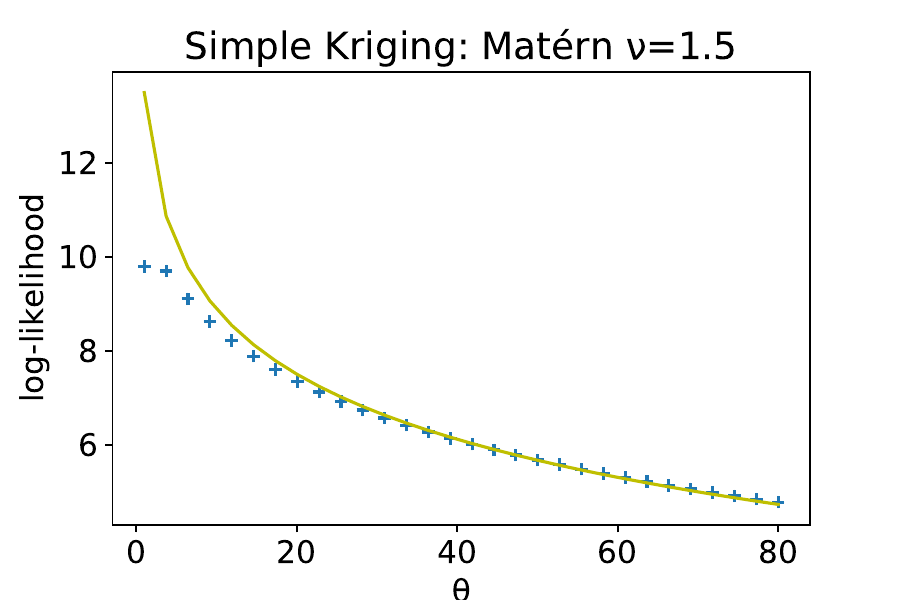}
    }
    \subfloat[Line: $-0.5\log(\theta)+14$]{
      \includegraphics[angle=0,width=0.33\textwidth, height=0.25\linewidth]{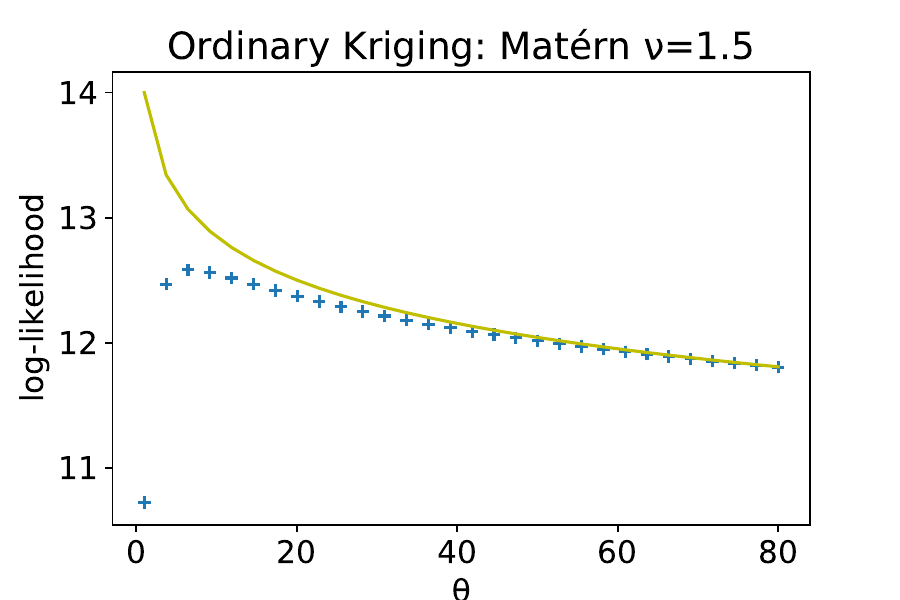}
    }
    \subfloat[]{
      \includegraphics[angle=0,width=0.33\textwidth, height=0.25\linewidth]{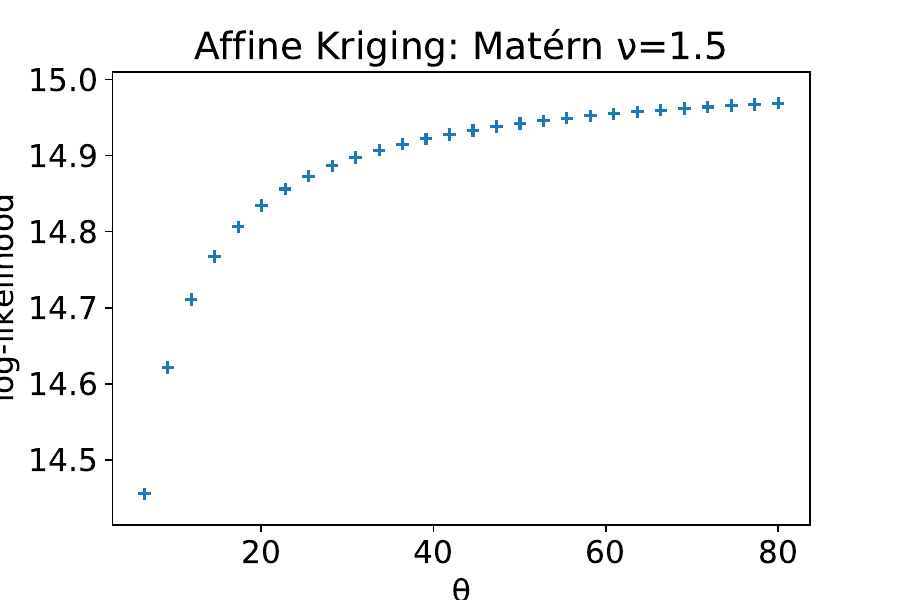}
      \label{sub:likelihood_O1}
    }
    \\
    \subfloat[Line: $-4.5\log(\theta)+25.5$]{
      \includegraphics[angle=0,width=0.33\textwidth, height=0.25\linewidth]{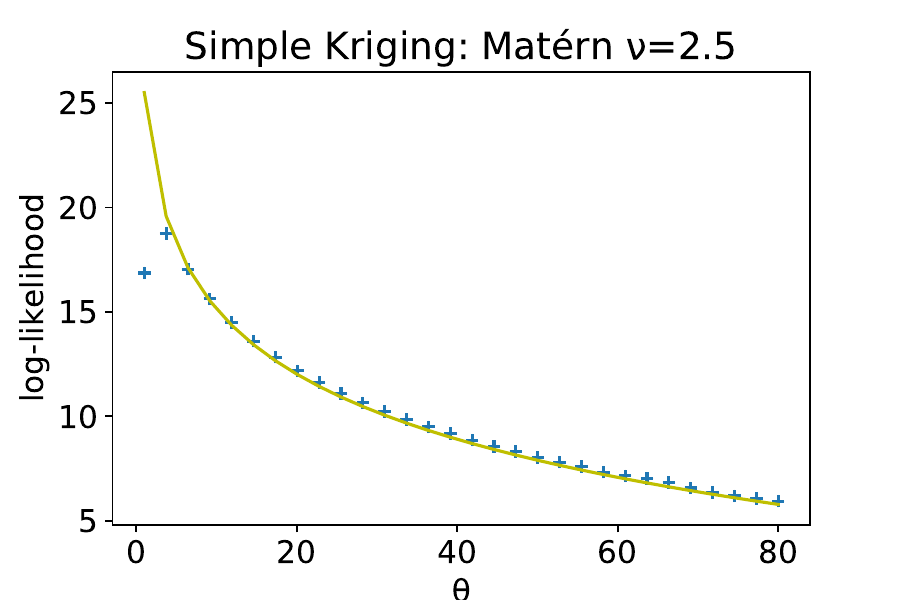}
    }
    \subfloat[Line: $-2\log(\theta)+25.7$]{
      \includegraphics[angle=0,width=0.33\textwidth, height=0.25\linewidth]{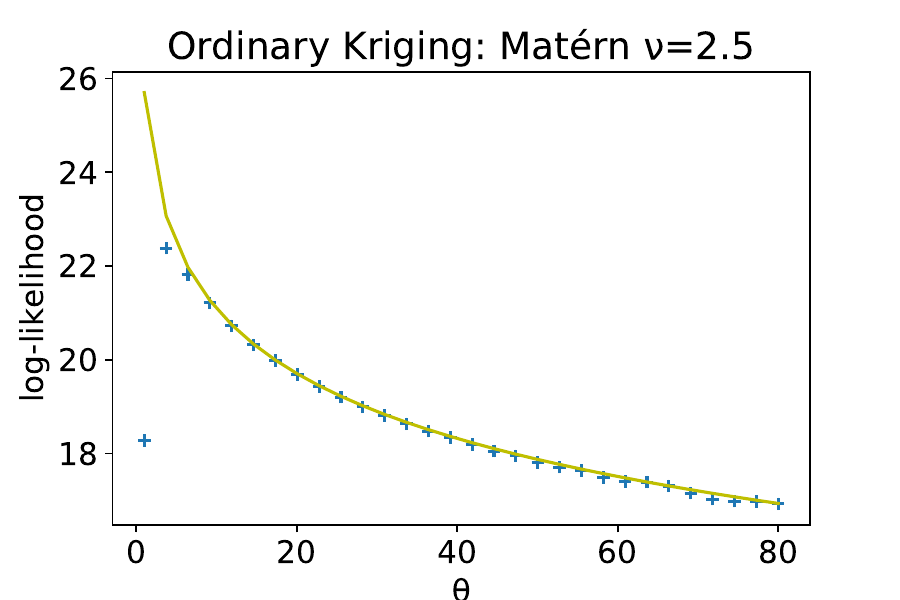}
    }   
    \subfloat[Line: $-0.5\log(\theta)+25.95$]{
      \includegraphics[angle=0,width=0.33\textwidth, height=0.25\linewidth]{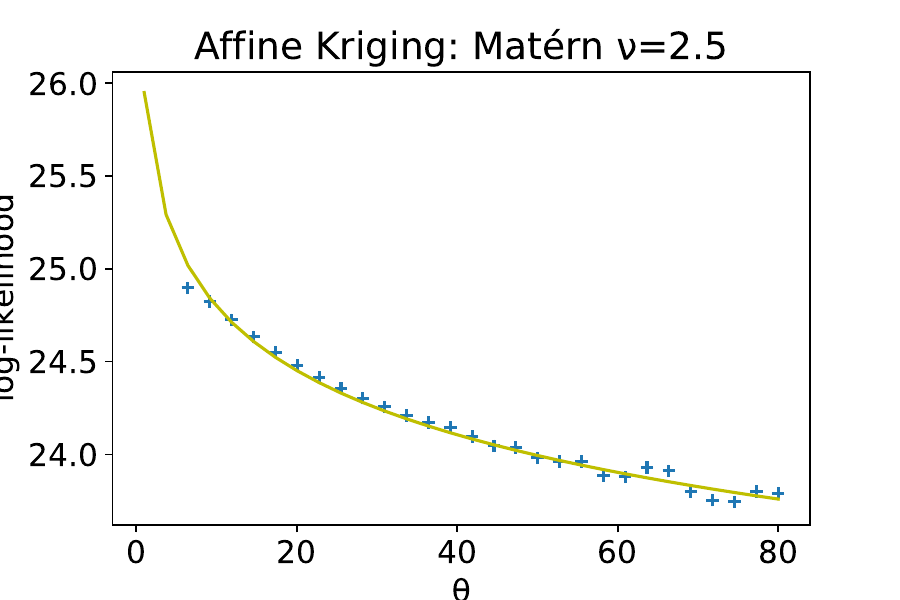}
    }     
    \caption{Logarithm of the likelihood function for large values of $\theta$ for varying Kriging models and correlation kernels.
    In each case, the design space is one-dimensional ($r=1$) and the design set is the 10-point regular grid on $[0,1]$. At every design point $x$, the observed value is $\sin(\pi x)$.}
    \label{Fig:likelihood}
  \end{center}
\end{figure}

This behavior of the likelihood is a good reason to investigate the behavior of the tail rate of the reference prior density more closely than we did in Proposition \ref{Prop:ref_prior_behavior}. \medskip

All propositions in this appendix are valid under the assumption that some $n$-point design set has been fixed and that $n>p$.

\begin{prop} \label{Prop:tail_rates_affine_kriging}
  In the case of Affine Kriging, with a Matérn kernel with smoothness $\nu \in (1,2)$, the reference prior $\pi(\theta)$ is $O(\theta^{-2(2-\nu)-1})$ when $\theta \to +\infty$. 
  It is a proper prior distribution.
  Furthermore, for every $\vecteur{y} \in \R^n$ 
  such that $\matrice{W} \trans \vecteur{y}$ is non-null,
  the likelihood function converges to a non-null constant when $\theta \to +\infty$.
\end{prop}

\begin{rmq}
  While this Proposition only provides an upper bound for the tail rate of the reference prior, the bound $O(\theta^{-2(2-\nu)-1})$ seems to be tight, judging by Figure \ref{Sub:log-prior} which was obtained with $\nu=3/2$.
\end{rmq}

\begin{proof}
  According to Lemma \ref{Lem:Matern_noninteger_nu_asymptotic_expansion}, $\co$ can be written as:

  \begin{equation}
    \co = \matrice{D}^{(0)} + \frac{a_1}{\theta^2} \matrice{D}^{(1)}
    + \frac{a_\nu}{\theta^{2\nu}} \matrice{D}^{(\nu)} + \frac{a_2}{\theta^4} \matrice{D}^{(2)} + \matrice{R}(\theta).
  \end{equation}
  
  In the expression above, 
  \begin{itemize}
  \item $a_1 = - \Gamma(\nu-1) \nu / \Gamma(\nu)$, 
  \item $a_2 = \Gamma(\nu-2) \nu^2 / \Gamma(\nu)$, 
  \item $a_\nu = \Gamma(-\nu) \nu^\nu / \Gamma(\nu)$, 
  \item $\matrice{D}^{(0)}$ is the $n \times n$ matrix filled with ones, 
  \item $\matrice{D}^{(1)}$ is the $n \times n$ matrix with $(i,i')$-th element $\left\| \vecteur{x}^{(i)} - \vecteur{x}^{(i')} \right\|^2$,  
  \item $\matrice{D}^{(2)}$ is the $n \times n$ matrix with $(i,i')$-th element $\left\| \vecteur{x}^{(i)} - \vecteur{x}^{(i')} \right\|^4$,  
  \item $\matrice{D}^{(\nu)}$ is the $n \times n$ matrix with $(i,i')$-th element $\left\| \vecteur{x}^{(i)} - \vecteur{x}^{(i')} \right\|^{2\nu}$,  
  \item $\matrice{R}$ is a differentiable function from $(0,+\infty)$ to the space of real $n \times n$ matrices that satisfies
  $\| \bs{R}(\theta) \| =  o(\theta^{-4})$ and $\| \frac{d}{d \theta} \bs{R}(\theta) \| =  o(\theta^{-5})$ when $\theta \to +\infty$.
  \end{itemize}

  Under Affine Kriging, $\matrice{H}$ is the $n \times (r+1)$ matrix 
  whose first column is $\vecteur{1}$ (the vector of $\R^n$ filled with ones) 
  and whose last $r$ columns contain the coordinates of the design set. \medskip

  Because $\matrice{W}$ is orthogonal to $\matrice{H}$, \citet{Sch37} (for example) implies that $\matrice{W} \trans \matrice{D}^{(0)} \matrice{W}$ and $\matrice{W} \trans \matrice{D}^{(1)} \matrice{W}$ are both null. We have therefore

  \begin{equation} \label{Eq:DL_Matern_1_2}
    \matrice{W} \trans \co \matrice{W} =
    \frac{a_\nu}{\theta^{2\nu}} \matrice{W} \trans \matrice{D}^{(\nu)} \matrice{W} + \frac{a_2}{\theta^4}  \matrice{W} \trans \matrice{D}^{(2)} \matrice{W} +  \matrice{W} \trans \matrice{R}(\theta) \matrice{W}.
  \end{equation}

  Using the notations from Appendix \ref{App:REFERENCE_POSTERIOR_PROPER}, we can identify $\nouveauD:=a_{\nu} \matrice{D}^{(\nu)}$ and $\nouveauD^\star:=a_{2}\matrice{D}^{(2)}$. Thus $g(\theta):=\theta^{-2\nu}$ and $g(\theta)^\star:=\theta^{-2l}$ with $l=2-\nu$. \medskip

  Besides, Lemma \ref{Lem:Matern_decay} asserts that $\left\| \co^{-1} \right\| = O(\theta^{2\nu})$ when $\theta \to +\infty$. Applying Lemma \ref{Lem:switch_pov}, 
  
  \begin{equation}
    \co^{-1} = \matrice{W} \left(\matrice{W} \trans \co \matrice{W} \right)^{-1} \matrice{W} \trans + \co^{-1} \matrice{H} \left(\matrice{H} \trans \co^{-1} \matrice{H}\right)^{-1} \matrice{H} \trans \co^{-1}.
  \end{equation}
  
  As both matrices in the right term are positive semi-definite,
  we have 
  
  \begin{equation}
  \left\| \matrice{W} \left(\matrice{W} \trans \co \matrice{W} \right)^{-1} \matrice{W} \trans \right\| = O(\theta^{2\nu}).
  \end{equation}
  
  And since $\matrice{W} \trans \matrice{W} = \matrice{I}_{n-p}$, this implies that $\left\| \left(\matrice{W} \trans \co \matrice{W} \right)^{-1} \right\| = O(\theta^{2\nu})$. \medskip

  Therefore 
  $\matrice{W} \trans \matrice{D}^{(\nu)} \matrice{W}$ is nonsingular. 
  The proof of Theorem \ref{THM:REFERENCE_POSTERIOR_PROPER} given in 
  Appendix \ref{App:REFERENCE_POSTERIOR_PROPER}
  yields the first assertion of this Proposition: 
  the paragraph of Appendix \ref{App:REFERENCE_POSTERIOR_PROPER} 
  below Equation \eqref{Eq:ersatz_prior} establishes in this situation 
  that $\pi(\theta) = O(\theta^{-2l-1}) = O(\theta^{-2(2-\nu)-1})$ 
  and that the reference prior $\pi(\theta)$ is proper. \medskip
 
  Moreover, given that
  $\matrice{W} \trans \matrice{D}^{(\nu)} \matrice{W}$
  is nonsingular, Equation \eqref{Eq:DL_Matern_1_2} implies that, 
  for every $\vecteur{y} \in \R^n$
  such that $\matrice{W} \trans \vecteur{y}$ is non-null, 

  \begin{align}
    & \lim_{\theta \to +\infty}
    \left| \matrice{W} \trans \co \matrice{W} \right|^{-\frac{1}{2}}
    \left( \vecteur{y} \trans \matrice{W} \right.
    \left(\matrice{W} \trans \co \matrice{W} \right)^{-1}
    \left. \matrice{W} \trans \vecteur{y} \right)^{-\frac{n-p}{2}} 
    \nonumber \\
    &=
    \left| \matrice{W} \trans \matrice{D}^{(\nu)} \matrice{W} \right|^{-\frac{1}{2}}
    \left( \vecteur{y} \trans \matrice{W} \right.
    \left(\matrice{W} \trans \matrice{D}^{(\nu)} \matrice{W} \right)^{-1}
    \left. \matrice{W} \trans \vecteur{y} \right)^{-\frac{n-p}{2}} .
  \end{align}

  This implies that the likelihood function converges to a non-null constant.
\end{proof}

As mentioned in Appendix \ref{App:REFERENCE_POSTERIOR_PROPER}, 
the case of Matérn kernels with integer smoothness is particularly tricky. 
Figure \ref{Fig:Matern_nu_1} focuses on the case where $\nu=1$. 
Once again, the most striking subfigure is the one showing the likelihood 
function under Affine Kriging: \ref{sub:likelihood_Matern_1_O1}.
The behavior of the likelihood function and reference prior in this case
is given in the following Proposition.

\begin{figure}[ht]
  \begin{center}
    \subfloat[$-\log(\theta)+1.8$]{
      \includegraphics[angle=0,width=0.33\textwidth, height=0.25\linewidth]{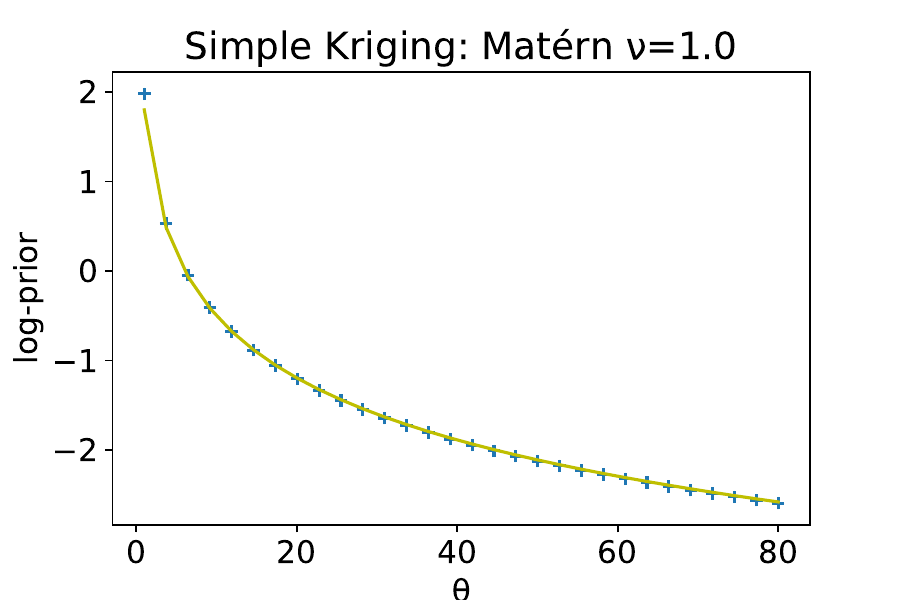}
    }
    \subfloat[$-\log(\theta)-\log(\log(\theta))+1$]{
      \includegraphics[angle=0,width=0.33\textwidth, height=0.25\linewidth]{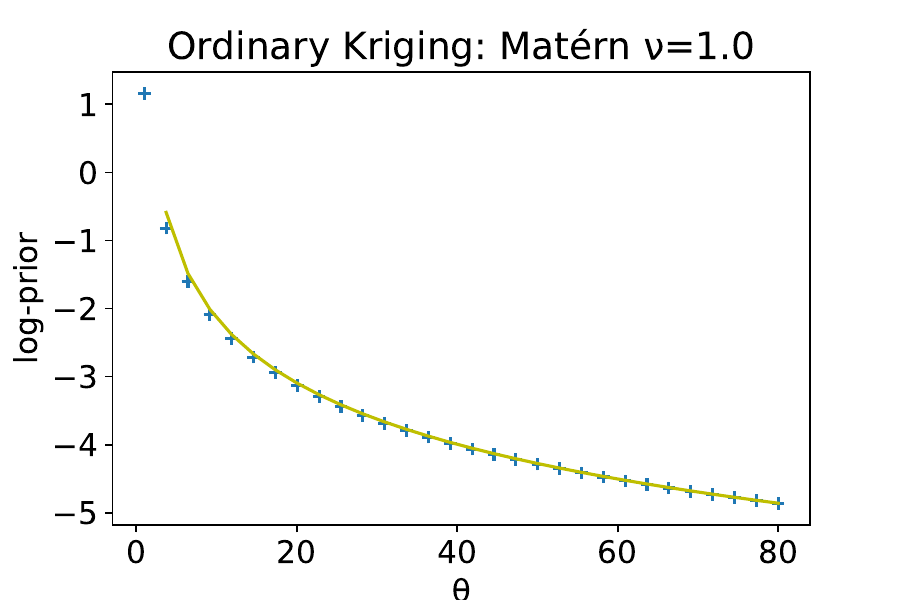}
      \label{Sub:log-prior_Ordinary_Kriging_Matern_1}
    }
    \subfloat[$-3\log(\theta)+\log(\log(\theta))+0.7$]{
      \includegraphics[angle=0,width=0.33\textwidth, height=0.25\linewidth]{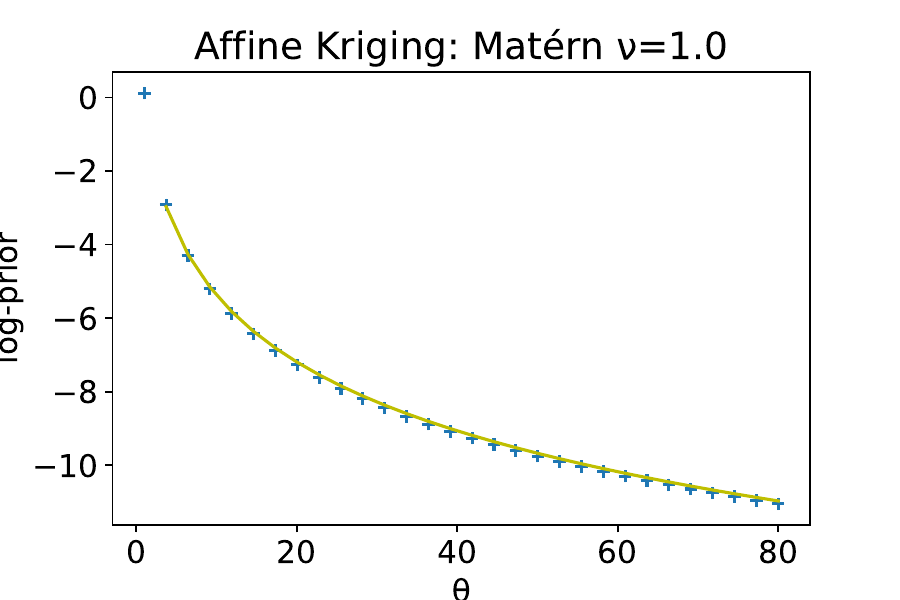}
      \label{Sub:log-prior_Affine_Kriging_Matern_1}
    }
    \\
    \subfloat[$-\log(\theta)-\frac{1}{2}\log(\log(\theta))+7.1$]{
      \includegraphics[angle=0,width=0.33\textwidth, height=0.25\linewidth]{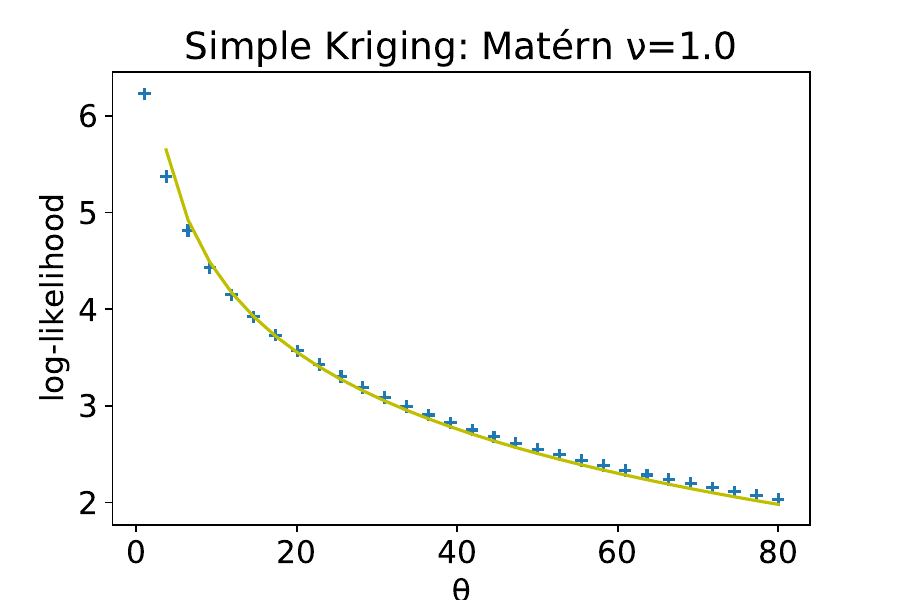}
    }
    \subfloat[$-\frac{1}{2}\log(\log(\theta))+7.74$]{
      \includegraphics[angle=0,width=0.33\textwidth, height=0.25\linewidth]{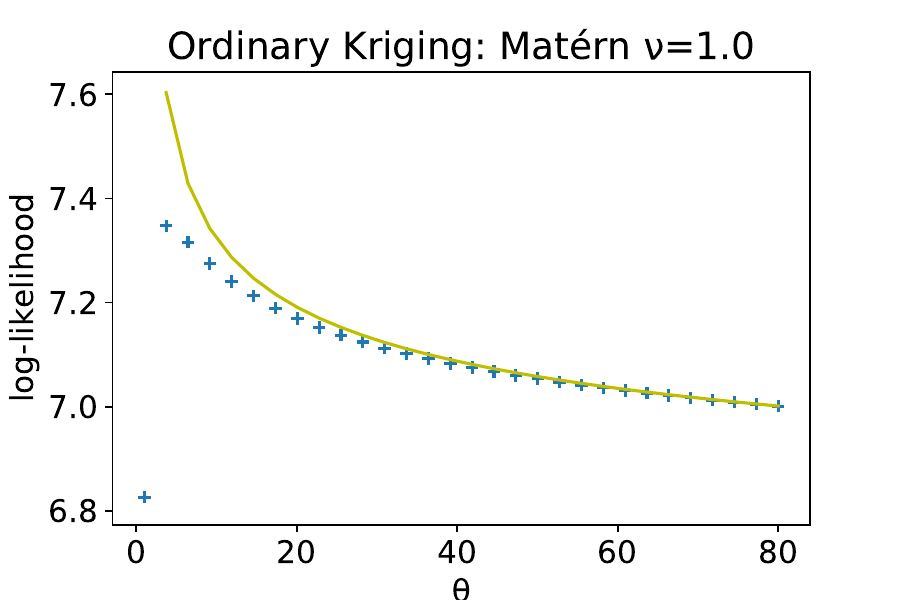}
      \label{Sub:likelihood_Ordinary_Kriging_Matern_1}
    }   
    \subfloat[]{
      \includegraphics[angle=0,width=0.33\textwidth, height=0.25\linewidth]{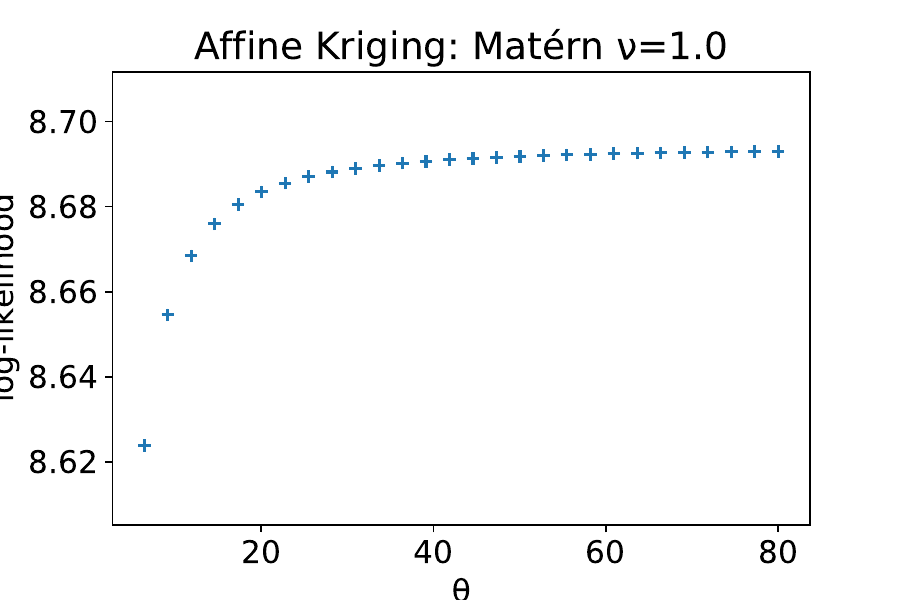}
      \label{sub:likelihood_Matern_1_O1}
    }     
    \caption{Logarithm of the prior (top) and likelihood function (bottom) 
    for large values of $\theta$ for varying Kriging models 
    using the Matérn kernel with smoothness $\nu=1$. 
    In each case, the design space is one-dimensional ($r=1$) 
    and the design set is the 10-point regular grid on $[0,1]$. 
    At every design point $x$, the observed value is $\sin(\pi x)$.}
    \label{Fig:Matern_nu_1}
  \end{center}
\end{figure}

\begin{prop} \label{Prop:tail_rates_affine_kriging_Matern_1}
  In the case of Affine Kriging, 
  with a Matérn kernel with smoothness $\nu=1$, 
  the reference prior $\pi(\theta)$ is $O(\theta^{-3} \log(\theta))$ 
  when $\theta \to +\infty$. 
  It is a proper prior distribution.
  Furthermore, for every $\vecteur{y} \in \R^n$ such that 
  $\matrice{W} \trans \vecteur{y}$ is non-null,
  the likelihood function converges to a non-null constant 
  when $\theta \to +\infty$.
\end{prop}

\begin{rmq}
  While this Proposition is only able to provide an upper bound for the tail rate of the reference prior, the bound $O(\theta^{-3} \log(\theta))$ seems to be tight, judging by Figure \ref{Sub:log-prior_Affine_Kriging_Matern_1}.
\end{rmq}

\begin{proof}
  
  We use Lemma \ref{Lem:Matern_integer_nu_asymptotic_expansion} to obtain:
  
  \begin{equation} \label{Eq:asymptotic_expansion_nu_1}
    \co = \matrice{D}^{(0)} -2 \left( \frac{\log(\theta)}{\theta^{2}} \matrice{D}^{(1)} + \frac{1}{\theta^2} \matrice{\tilde{D}}^{(1)} \right) 
    + \frac{\log(\theta)}{\theta^4} \matrice{D}^{(2)} + \matrice{R}(\theta).
  \end{equation}
  
  In the expression above, 
  \begin{itemize}
  \item $\matrice{D}^{(0)}$ is the $n \times n$ matrix filled with ones, 
  \item $\matrice{D}^{(1)}$ is the $n \times n$ matrix with $(i,i')$-th element $\left\| \vecteur{x}^{(i)} - \vecteur{x}^{(i')} \right\|^2$,  
  \item $\matrice{D}^{(2)}$ is the $n \times n$ matrix with $(i,i')$-th element $\left\| \vecteur{x}^{(i)} - \vecteur{x}^{(i')} \right\|^4$,  
  \item $\matrice{\tilde{D}}^{(1)}$ is the $n \times n$ matrix with null diagonal and $(i,i')$-th element ($i\neq i'$) given by 
  $$\| \bs{x}^{(i)} - \bs{x}^{(i')} \|^{2} \left\{-\log \left( \| \bs{x}^{(i)} - \bs{x}^{(i')} \| \right) - \gamma + \frac{1}{2} \right\} ,$$ where $\gamma$ is Euler's constant,
  \item $\matrice{R}$ is a differentiable function from $(0,+\infty)$ 
  to the space of real $n \times n$ matrices that satisfies
  $\| \bs{R}(\theta) \| =  O(\theta^{-4})$ and $\| \frac{d}{d \theta} \bs{R}(\theta) \| =  O(\theta^{-5})$ when $\theta \to +\infty$.
  \end{itemize}
  
  The rest is similar to the proof of Proposition \ref{Prop:tail_rates_affine_kriging}. Using the notations of Appendix \ref{App:REFERENCE_POSTERIOR_PROPER}, we have $\nouveauD:=-2\matrice{\tilde{D}}^{(1)}$, $\nouveauD^\star:=\matrice{D}^{(2)}$, $g(\theta) := \theta^{-2}$, $g^\star(\theta):= \theta^{-2l} \log(\theta)$ with $l=1$.
  The relevant part of Appendix \ref{App:REFERENCE_POSTERIOR_PROPER} 
  is the part concerning Matérn kernels with integer smoothness: 
  Appendix \ref{App:Matern_integer_soothness}. 
  As $\nu=1$, we are in the case where either for all $\lambda \neq 0$, 
  $\nouveauD \neq \lambda \matrice{D}^{(\nu)}$ or 
  for all $\lambda^\star \neq 0$, 
  $\nouveauD^\star \neq \lambda^\star \matrice{\tilde{D}}^{(\nu)}$. 
  Here, both checks hold. 
  Further, for the same reason as in the proof of Propostion 
  \ref{Prop:tail_rates_affine_kriging}, we are in the subcase where 
  $\matrice{W} \trans \nouveauD \matrice{W}$ is nonsingular. 
  In this subcase, Appendix \ref{App:Matern_integer_soothness} states 
  that the reference prior is 
  $O(\theta^{-2l-1} \log(\theta)) = O(\theta^{-3} \log(\theta))$ 
  and that it is proper. 
  Finally, $\matrice{W} \trans \nouveauD \matrice{W}$ being nonsingular
  also leads to the conclusion that, for every $\vecteur{y} \in \R^n$ 
  such that $\matrice{W} \trans \vecteur{y}$ is non-null,
  the likelihood function converges to a non-null constant 
  when $\theta$ goes to infinity.
\end{proof}

The most striking behavior of the reference prior density is given in Figure \ref{Sub:log-prior_Ordinary_Kriging_Matern_1} (Ordinary Kriging), 
where the tail rate seems to be $\theta^{-1} \log(\theta)^{-1}$ multiplied by some constant factor. It decreases a little faster than the upper bound given in Proposition \ref{Prop:ref_prior_behavior}, $O(\theta^{-1})$, but still not fast enough to make the reference prior proper. Proposition \ref{Prop:tail_rates_ordinary_kriging_Matern_1} below shows that $O(\theta^{-1} \log(\theta)^{-1})$ is indeed an upper bound for the reference prior.

\begin{prop} \label{Prop:tail_rates_ordinary_kriging_Matern_1}
  In the case of Ordinary Kriging, with a Matérn kernel with smoothness $\nu=1$, 
  if $n>r+3$, then:
  \begin{itemize}
    \item the reference prior $\pi(\theta)$ is 
    $O(\theta^{-1} \log(\theta)^{-1})$ when $\theta \to +\infty$;
    \item the matrix $\matrice{W} \trans \matrice{D}^{(1)} \matrice{W}$
  is singular;
    \item  for every vector $\vecteur{y} \in \R^n$ 
    such that $\matrice{W} \trans \vecteur{y}$
    does not belong to the vector subspace of $\R^{n-1}$ 
    spanned by the columns of 
    $\matrice{W} \trans \matrice{D}^{(1)} \matrice{W}$,
    the reference posterior distribution $\pi(\theta | \vecteur{y})$ 
    is $O(\theta^{-1} \log(\theta)^{-3/2})$ when $\theta \to +\infty$.
  \end{itemize}
\end{prop}

\begin{rmq}
  The upper bound given for the reference posterior tail rate just barely makes it proper. Yet the tail rate suggested by Figure \ref{Sub:likelihood_Ordinary_Kriging_Matern_1} for the likelihood function is consistent with it: $\log(\theta)^{-1/2}$ multiplied by some constant factor.
\end{rmq}

\begin{proof}
  The second assertion follows from Corollary \ref{Cor:D_singular}.
  Only the first and third assertions need to be proved. \medskip

  Let us use Lemma \ref{Lem:Matern_integer_nu_asymptotic_expansion} to obtain Equation \eqref{Eq:asymptotic_expansion_nu_1}. 
  \medskip

  With Ordinary Kriging, the matrix $\matrice{H}$ is the vector $\vecteur{1} \in \R^n$ whose entries are all equal to 1. Therefore $\matrice{W} \trans \matrice{D}^{(0)} \matrice{W}$ is the null $(n-1) \times (n-1)$ matrix and

  \begin{align}
    \matrice{W} \trans \co \matrice{W} = -2 \left( \frac{\log(\theta)}{\theta^{2}} \matrice{W}\ \trans \matrice{D}^{(1)} \matrice{W} + \frac{1}{\theta^2} \matrice{W} \trans \matrice{\tilde{D}}^{(1)} \matrice{W} \right)  \\
    + \frac{\log(\theta)}{\theta^4} \matrice{W} \trans \matrice{D}^{(2)} \matrice{W} + \matrice{W} \trans \matrice{R}(\theta) \matrice{W}. \nonumber  
  \end{align}
  
  Using the notations of Appendix \ref{App:REFERENCE_POSTERIOR_PROPER}, we have 
  $\nouveauD:=-2 \matrice{D}^{(1)}$ and $\nouveauD^\star:=-2\matrice{\tilde{D}}^{(1)}$, $g(\theta) = \theta^{-2} \log(\theta)$, $g^\star(\theta):= \log(\theta)^{-1}$. \medskip

  Corollary \ref{Cor:D_singular} states that the rank of $\nouveauD$ is lower or equal to $r+2$, so this is a fortiori true for the rank of $\matrice{W} \trans \nouveauD \matrice{W}$. Since $r+2<n-1$, this implies that $\matrice{W} \trans \nouveauD \matrice{W}$ is singular. \medskip

  The proof of Theorem \ref{THM:REFERENCE_POSTERIOR_PROPER}, 
  Appendix \ref{App:REFERENCE_POSTERIOR_PROPER}, 
  then yields both results of this Proposition. 
  The relevant part of Appendix \ref{App:REFERENCE_POSTERIOR_PROPER} is 
  the part concerning Matérn kernels with integer smoothness: 
  Appendix \ref{App:Matern_integer_soothness}. 
  As $\nu=1$, we are in the case where there exists $\lambda \neq 0$ 
  such that $\nouveauD = \lambda \matrice{D}^{(\nu)}$ ($\lambda=-2$) 
  and there exists $\lambda^\star \neq 0$ 
  such that  $\nouveauD^\star = \lambda^\star \matrice{\tilde{D}}^{(\nu)}$ 
  ($\lambda^\star = -2$). 
  Further, we are in the subcase where 
  $\matrice{W} \trans \nouveauD \matrice{W}$ is singular. 
  This subcase is dealt with in the last paragraph of 
  Appendix \ref{App:Matern_integer_soothness}:
  it yields the first assertion.
  \medskip

  Any vector $\vecteur{y} \in \R^n$ such that
  $\matrice{W} \trans \vecteur{y}$ does not belong to the subspace of $\R^{n-1}$
  spanned by the columns of $\matrice{W} \trans \matrice{D}^{(1)} \matrice{W}$
  verifies Assumption $\mathbb{\tilde{A}}_1(\vecteur{y})$
  from Appendix \ref{App:Precise_formulation}.
  Therefore Proposition \ref{Prop:likelihood_as_product_details}
  can be applied instead of 
  Proposition \ref{Prop:likelihood_as_product} and 
  the last paragraph of
  Appendix \ref{App:Matern_integer_soothness}
  yields the third assertion.
\end{proof}

\section{Auxiliary facts} \label{App:auxiliary_facts}

Throughout this appendix, the following notations are used. We write $[\![\cdot, \cdot]\!]$ to denote intervals of integers. For example, $[\![1,3]\!]$ is the set $\{1,2,3\}$.
We use the notation $\Ker$ to denote the kernel of a linear mapping 
or of a matrix.
The ``trivial vector space'' is the vector space containing only the null vector.

\subsection{Algebra}

\begin{lem} \label{Lem:switch_pov}
Let $a$ and $b$ be positive integers and let $\bs{\Sigma}$ be a nonsingular symmetric $(a+b) \times (a+b)$ matrix. Then, for any $(a+b) \times a$ matrix $\bs{A}$ with rank $a$ and any $(a+b) \times b$ matrix $\bs{B}$ with rank $b$ such that $\bs{A} \trans \bs{B}$ is the null $a \times b$ matrix,

\begin{equation}
\bs{B} \left( \bs{B} \trans \bs{\Sigma} \bs{B} \right)^{-1} \bs{B} \trans 
= \bs{\Sigma}^{-1} \left(\bs{I}_{a+b} - \bs{A} \left( \bs{A} \trans \bs{\Sigma}^{-1} \bs{A} \right)^{-1} \bs{A} \trans \bs{\Sigma}^{-1} \right).
\end{equation}
\end{lem}

Lemma \ref{Lem:switch_pov} is used in the proofs of Propositions \ref{Prop:rep_prior} and \ref{Prop:marginal_likelihood}.

\begin{proof}
  This is a simple reformulation of Lemma 6 from \citet{RSH12}.
\end{proof}

\begin{lem} \label{Lem:sd_eigenvalues}
Let $m$ be a positive integer, $\bs{\Sigma}$ be a nonsingular $m \times m$ matrix, and $\bs{A}$ and $\bs{B}$ be $m \times m$ matrices. If there exists a real number $t$ such that

\begin{equation}
\bs{A} = t \bs{\Sigma} + \bs{B},
\end{equation}

then

\begin{equation}
\Tr \left[ \left\{ \bs{A} \bs{\Sigma}^{-1} \right\}^2 \right] - \frac{1}{m} \left[ \Tr \left\{ \bs{A} \bs{\Sigma}^{-1} \right\} \right]^2
=
\Tr \left[ \left\{ \bs{B} \bs{\Sigma}^{-1} \right\}^2 \right] - \frac{1}{m} \left[ \Tr \left\{ \bs{B} \bs{\Sigma}^{-1} \right\} \right]^2.
\end{equation}

\end{lem}

Lemma \ref{Lem:sd_eigenvalues} is used in the proof of Theorem \ref{THM:REFERENCE_POSTERIOR_PROPER}.

\begin{proof}
The lemma follows from a direct calculation:

\begin{align}
\Tr \left[ \bs{A} \bs{\Sigma}^{-1} \right] 
&= \Tr \left[ \bs{B} \bs{\Sigma}^{-1} \right] + tm. \\
\Tr \left[ \left\{ \bs{A} \bs{\Sigma}^{-1} \right\}^2 \right] 
&= \Tr \left[ \left\{ \bs{B} \bs{\Sigma}^{-1} \right\}^2 \right] 
+ 2t \Tr \left[ \bs{B} \bs{\Sigma}^{-1} \right] + t^2 m.
\end{align}
\end{proof}

\begin{lem} \label{Lem:ref_prior_majoration}
Let $m>a$ be positive integers, $\bs{\Sigma}$ be an $m \times m$ symmetric positive definite matrix, $\bs{\Sigma}'$ be an $m \times m$ symmetric matrix and $\bs{A}$ be an $m \times a$ matrix with rank $a$. Denote $\bs{Q} := \bs{I}_m - \bs{A} \left( \bs{A} \trans \bs{\Sigma}^{-1} \bs{A} \right)^{-1} \bs{A} \trans \bs{\Sigma}^{-1}$. Then, if there exist $t_1 \in \R$ and $t_2 \in [0,+\infty)$ such that the matrix $\bs{F} := t_1 \bs{\Sigma} - \bs{\Sigma}'$ is positive semi-definite and satisfies $\forall \bs{\xi} \in \R^m$ $\bs{\xi} \trans \bs{F} \bs{\xi} \leqslant t_2 \bs{\xi} \trans \bs{\Sigma} \bs{\xi}$, then

\begin{equation}
\sqrt{ \Tr \left[ ( \bs{\Sigma}' \bs{\Sigma}^{-1} \bs{Q} )^2 \right] - \frac{1}{m-a} \left[ \Tr \left\{ \bs{\Sigma}' \bs{\Sigma}^{-1} \bs{Q} \right\} \right]^2 }
\leqslant (m-a) t_2.
\end{equation}

Moreover,

\begin{equation}
  \sqrt{ \Tr \left[ ( \bs{\Sigma}' \bs{\Sigma}^{-1} )^2 \right] - \frac{1}{m} \left[ \Tr \left\{ \bs{\Sigma}' \bs{\Sigma}^{-1} \right\} \right]^2 }
  \leqslant m t_2.
\end{equation}

\end{lem}

Lemma \ref{Lem:ref_prior_majoration} is used in the proof of Proposition \ref{Prop:ref_prior_behavior}.
\begin{proof}

Let $\bs{B}$ be an $m \times (m-a)$ matrix with rank $m-a$ such that $\bs{A} \trans \bs{B}$ is the null $a \times (m-a)$ matrix.  \medskip

We only prove the first assertion.
The proof of the second assertion is identical,
except that $a$ must be replaced by 0
and $\matrice{Q}$ and $\matrice{B}$
must both be replaced by $\matrice{I}_m$. \medskip

By applying Lemma \ref{Lem:switch_pov}, we obtain that $\bs{\Sigma}^{-1} \bs{Q} = \bs{B} \left( \bs{B} \trans \bs{\Sigma} \bs{B} \right)^{-1} \bs{B} \trans $. \medskip

Because of the properties of the trace, this implies

\begin{align}
\Tr \left[ \bs{\Sigma}' \bs{\Sigma}^{-1} \bs{Q} \right] &= \Tr \left[  \bs{B} \trans \bs{\Sigma}'  \bs{B} \left( \bs{B} \trans \bs{\Sigma} \bs{B} \right)^{-1} \right] \\
\Tr \left[ (\bs{\Sigma}' \bs{\Sigma}^{-1} \bs{Q} )^2 \right] &= \Tr \left[ \left\{ \bs{B} \trans \bs{\Sigma}'  \bs{B} \left( \bs{B} \trans \bs{\Sigma} \bs{B} \right)^{-1} \right\}^2 \right].
\end{align}

Similarly, we have

\begin{align}
\Tr \left[ \bs{F} \bs{\Sigma}^{-1} \bs{Q} \right] &= \Tr \left[  \bs{B} \trans \bs{F}  \bs{B} \left( \bs{B} \trans \bs{\Sigma} \bs{B} \right)^{-1} \right] \\
\Tr \left[ (\bs{F} \bs{\Sigma}^{-1} \bs{Q} )^2 \right] &= \Tr \left[ \left\{ \bs{B} \trans \bs{F}  \bs{B} \left( \bs{B} \trans \bs{\Sigma} \bs{B} \right)^{-1} \right\}^2 \right].
\end{align}

Because $\bs{B} \trans \bs{F} \bs{B} = t_1 \bs{B} \trans \bs{\Sigma} \bs{B} - \bs{B} \trans \bs{\Sigma}' \bs{B}$, Lemma \ref{Lem:sd_eigenvalues} implies
\begin{equation}
\begin{split}
& \Tr \left[ \left\{ \bs{B} \trans \bs{\Sigma}'  \bs{B} \left( \bs{B} \trans \bs{\Sigma} \bs{B} \right)^{-1} \right\}^2 \right] - \frac{1}{m-a} \left[\Tr \left\{  \bs{B} \trans \bs{\Sigma}'  \bs{B} \left( \bs{B} \trans \bs{\Sigma} \bs{B} \right)^{-1} \right\} \right]^2 \\
=& \Tr \left[ \left\{ \bs{B} \trans \bs{F}  \bs{B} \left( \bs{B} \trans \bs{\Sigma} \bs{B} \right)^{-1} \right\}^2 \right] - \frac{1}{m-a} \left[ \Tr \left\{  \bs{B} \trans \bs{F}  \bs{B} \left( \bs{B} \trans \bs{\Sigma} \bs{B} \right)^{-1} \right\} \right]^2.
\end{split}
\end{equation}

Combining the five equations above yields

\begin{align} \label{Eq:substitution}
&\Tr \left[ (\bs{\Sigma}' \bs{\Sigma}^{-1} \bs{Q} )^2 \right] - \frac{1}{m-a} \left[ \Tr \left\{ \bs{\Sigma}' \bs{\Sigma}^{-1} \bs{Q} \right\} \right]^2 \nonumber \\
=& \Tr \left[\left(\bs{F} \bs{\Sigma}^{-1} \bs{Q} \right)^2\right] - \frac{1}{m-a} \left[ \Tr \left\{ \bs{F} \bs{\Sigma}^{-1} \bs{Q} \right\} \right]^2.
\end{align}

An elementary computation shows that $\bs{\Sigma}^{-1} \bs{Q} =  \bs{Q} \trans \bs{\Sigma}^{-1} \bs{Q}$.
Consider the Cholesky decomposition $ \bs{\Sigma} =: \bs{L} \bs{L} \trans$. Then $\bs{\Sigma}^{-1} \bs{Q} = \bs{Q} \trans \bs{\Sigma}^{-1} \bs{Q} = \bs{Q} \trans \left(\bs{L}^{-1}\right) \trans \bs{L}^{-1} \bs{Q}$.

\begin{align} 
\Tr \left[\left(\bs{F} \bs{\Sigma}^{-1} \bs{Q} \right)^2\right]
&= \Tr \left[ \left( \bs{F} \bs{Q} \trans \left(\bs{L}^{-1}\right) \trans \bs{L}^{-1} \bs{Q} \right)^2 \right]
= \Tr \left[ \left( \bs{L}^{-1} \bs{Q} \bs{F} \bs{Q} \trans \left(\bs{L}^{-1}\right) \trans \right)^2 \right] \nonumber \\
&\leqslant \left[ \Tr \left\{ \bs{L}^{-1} \bs{Q} \bs{F} \bs{Q} \trans \left(\bs{L}^{-1}\right) \trans \right\} \right]^2
= \left[ \Tr \left\{ \bs{F} \bs{\Sigma}^{-1} \bs{Q} \right\} \right]^2. \label{Eq:majoration_trace}
\end{align} 

The inequality holds because $ \bs{L}^{-1} \bs{Q} \bs{F} \bs{Q} \trans \left(\bs{L}^{-1}\right) \trans $ is a symmetric positive semi-definite matrix. \medskip

Let $(\bs{\xi}_i)_{1 \leqslant i \leqslant m}$ be a basis of unit eigenvectors of $\bs{\Sigma}^{-1} \bs{Q}$ such that for every integer $ i \in [\![1,m]\!] \setminus [\![1,m-a]\!]$, $\bs{\xi}_i$ belongs to the kernel of $\bs{\Sigma}^{-1} \bs{Q}$. Indeed, because $\bs{\Sigma}^{-1} \bs{Q} = \bs{B} \left( \bs{B} \trans \bs{\Sigma} \bs{B} \right)^{-1} \bs{B} \trans $, this kernel has the same dimension as the kernel of $\bs{B} \trans$: $a$.

Denoting by $(s_i)_{1 \leqslant i \leqslant m}$ the family of the eigenvalues corresponding to the family of eigenvectors  $(\bs{\xi}_i)_{1 \leqslant i \leqslant m}$, we have for every integer $i \in [\![1,m-a]\!]$ $s_i\neq0$ and

\begin{align}
(\bs{\xi}_i) \trans \bs{\Sigma} \bs{\xi}_i &= s_i^{-2} \left\{ (\bs{\xi}_i) \trans \bs{Q} \trans \bs{\Sigma}^{-1} \right\} \bs{\Sigma} \left\{ \bs{\Sigma}^{-1} \bs{Q} \bs{\xi}_i \right\} \nonumber \\
&= s_i^{-2} (\bs{\xi}_i) \trans \bs{Q} \trans \bs{\Sigma}^{-1} \bs{Q} \bs{\xi}_i \nonumber \\
&= s_i^{-2} (\bs{\xi}_i) \trans \bs{\Sigma}^{-1} \bs{Q} \bs{\xi}_i \nonumber \\
&= s_i^{-1}.
\end{align}

This implies the third equality below:

\begin{align}
\Tr \left[ \bs{F} \bs{\Sigma}^{-1} \bs{Q} \right]
&= \sum_{i=1}^m \left(\bs{\xi}_i\right) \trans \bs{F} \bs{\Sigma}^{-1} \bs{Q} \bs{\xi}_i \nonumber \\
&= \sum_{i=1}^{m-a} s_i \left(\bs{\xi}_i\right) \trans  \bs{F} \bs{\xi}_i \nonumber \\
&= \sum_{i=1}^{m-a} \frac{\left(\bs{\xi}_i\right) \trans \bs{F} \bs{\xi}_i} {\left(\bs{\xi}_i\right) \trans \bs{\Sigma} \bs{\xi}_i} \nonumber \\
&\leqslant (m-a) t_2.
\end{align} 

Equations (\ref{Eq:substitution}) and (\ref{Eq:majoration_trace}) yield the result.

\end{proof}

\begin{lem} \label{Lem:first_trivial_intersection}
Let $(\bs{D}_k)_{k \in \N}$ be a sequence of matrices of the same size. If $\sum_{k \in \N} \bs{D}_k$ exists and its kernel is the trivial vector space, then there exists a nonnegative integer $N$ such that $\cap_{k=0}^N \Ker \bs{D}_k$ is the trivial vector space.
\end{lem}

Lemma \ref{Lem:first_trivial_intersection} is used in the proof of Lemma \ref{Lem:upper_bound_yfactor_in_likelihood} and thus contributes to the proof of Theorem \ref{THM:REFERENCE_POSTERIOR_PROPER}.

\begin{proof}
Assume the sum $\sum_{k \in \N} \bs{D}_k$ exists and its kernel is the trivial vector space. Consider the sequence $(d(n))_{n \in \N}$ where for every nonnegative integer $n$, $d(n)$ is the dimension of $\cap_{k=0}^n \Ker \bs{D}^{(k)}$. $(d(n))_{n \in \N}$ is a nonincreasing sequence of nonegative integers, so it is convergent. If its limit is strictly greater than 0, then for every nonnegative integer $n$, there exists a unit vector $\bs{v}_n$ that belongs to $\cap_{k=0}^n \Ker \bs{D}^{(k)}$. Because the unit sphere is compact, there exists an increasing mapping $\phi: \N \to \N$ such that the subsequence $(\bs{v}_{\phi(n)})_{n \in \N}$ converges to a limit $\bs{v}$ such that $\|\bs{v}\|=1$. Besides, for every pair of nonnegative integers $n \leqslant n'$, $\bs{v}_{\phi(n')} \in \cap_{k=0}^{\phi(n)} \Ker \bs{D}^{(k)}$. Given this set is closed, the limit $\bs{v}$ also belongs to $\cap_{k=0}^{\phi(n)} \Ker \bs{D}^{(k)}$. So for every nonnegative integer $k$, $\bs{v} \in \Ker \bs{D}^{(k)}$ and therefore $\bs{v} \in \cap_{k=0}^\infty \Ker \bs{D}^{(k)}$. So $\bs{v}$ can only be the null vector, which is absurd since $\|\bs{v}\|=1$.
We deduce from this contradiction that the limit of the sequence of integers $(d(n))_{n \in \N}$ is 0. Therefore there exists a nonnegative integer $N$ such that $d(N)=0$.
\end{proof}

\subsection{Maclaurin series} \label{App:Maclaurin_series}
The lemmas in this subsection deal with the following setting. \medskip

Let $m$ be a positive integer and let $\bs{M}$ be a continuous mapping from $\R$ to $\mathcal{M}_m$, the set of $m \times m$ matrices. Assume $\bs{M}$ admits the following Maclaurin series:
\begin{equation} \label{Eq:Maclaurin}
\bs{M}(t) = \sum_{k=0}^N a_k(t)  \bs{A}_k + \bs{B}(t).
\end{equation}
In the expression above, $N$ is a nonnegative integer and for every $k \in [\![0,N]\!]$:
\begin{enumerate}
\item $a_k$ is a continuous mapping $(0,+\infty) \to \R$ such that 
for all $t \in (0,+\infty)$, $a_k(t) \neq 0$;
\item for every nonnegative integer $l<k$, $a_k(t) = o(|a_l(t)|)$ when $t \to 0$;
\item $\bs{A}_k$ is a non-null symmetric $m \times m$ matrix.
\end{enumerate}
$\bs{B}$ is a continuous mapping $(0,+\infty) \rightarrow \mathcal{M}_m$ such that for every $t \in \R$, $\bs{B}(t)$ is a symmetric matrix and when $t \to 0$, $\| \bs{B}(t) \| = o(|a_N(t)|)$.

\begin{lem} \label{Lem:majoration_inverse_Maclaurin}
Consider \eqref{Eq:Maclaurin}. If $\cap_{k=0}^N \Ker \bs{A}_k$ is 
the trivial vector space and if there exists $T>0$ such that 
for all $t \in (0,T)$ $\bs{M}(t)$ is nonsingular, then when $t \to 0$, $\left\| \bs{M}(t)^{-1} \right\| = O\left(|a_N(t)|^{-1}\right)$.
\end{lem}

Lemma \ref{Lem:majoration_inverse_Maclaurin} and some elements of its proof below are used in the proof of Lemma \ref{Lem:minoration_yinversey_Maclaurin}, which itself is used in the proof of Lemma \ref{Lem:upper_bound_yfactor_in_likelihood} and thus contributes to the proof of Theorem \ref{THM:REFERENCE_POSTERIOR_PROPER}.
Lemma \ref{Lem:majoration_inverse_Maclaurin} is also used in the proof 
of Lemma \ref{Lem:developpement_asymptotique_defini_positif}, 
which is then used in the proof of 
Lemma \ref{Lem:majoration_derivee_RQ_SE} and thus contributes to 
the proof of Proposition \ref{Prop:ref_prior_behavior}.

\begin{proof}
Assume that $\cap_{k=0}^N \Ker \bs{A}_k$ is the trivial vector space and that there exists $T>0$ such that for all $t \in (0,T)$, $\bs{M}(t)$ is a nonsingular matrix.\medskip

If $N=0$, then $\bs{A}_0$ is nonsingular and the conclusion is trivial. \medskip

If $N \geqslant 1$, we may assume without loss of generality that $\cap_{k=0}^{N-1}  \Ker \bs{A}_k$ is a nontrivial vector space, otherwise we could replace $N$ by $N-1$ and $\bs{B}(t)$  by $\left\{ a_N(t) \bs{A}_N + \bs{B}(t) \right\}$ for all $t \in \R$. \medskip 

Let $d_N$ be the dimension of 
the orthogonal complement of
$\cap_{k=0}^{N-1}  \Ker \bs{A}_k$. Let $\bs{W}_N$ be an $m \times (m-d_N)$ matrix whose columns form an orthonormal basis of $\cap_{k=0}^{N-1}  \Ker \bs{A}_k$, and let $\bs{P}_N$ be an $m \times d_N$ matrix whose columns form an orthonormal basis of its orthogonal complement. 
Then $(\bs{P}_N, \bs{W}_N)$ is an orthogonal matrix: 
it is the $m \times m$ matrix whose left $m \times d_N$ block is 
$\matrice{P}_N$ and whose right $m \times (m-d_N)$ block is 
$\matrice{W}_N$. For all $t \in \R$, let us replace $\bs{M}(t)$ by 
$\tilde{\bs{M}}(t) := (\bs{P}_N, \bs{W}_N) \trans \bs{M}(t) (\bs{P}_N, \bs{W}_N)$. Because $(\bs{P}_N, \bs{W}_N)$ is an orthogonal matrix, the Frobenius norm of $\bs{M}(t)^{-1}$ is unchanged. Naturally, for all $k \in [\![0,N]\!]$, $\bs{A}_k$ is replaced by $\tilde{\bs{A}_k} := (\bs{P}_N, \bs{W}_N) \trans \bs{A}_k (\bs{P}_N, \bs{W}_N)$ and for every $t \in \R$, $\bs{B}(t)$ is replaced by $\tilde{\bs{B}}(t) := (\bs{P}_N, \bs{W}_N) \trans \bs{B}(t) (\bs{P}_N, \bs{W}_N)$. \medskip

Now, for every $k \in [\![1,N]\!]$, $\tilde{\bs{A}}_k$ 
can be decomposed into blocks -- a $d_N \times d_N$ block $\bs{A}_k'$, an $(m-d_N) \times (m-d_N)$ block $\bs{A}_k''$ and a $d_N \times (m-d_N)$ block $\bs{A}_k'''$:
\begin{equation}
\tilde{\bs{A}}_k = \begin{pmatrix} \bs{A}_k' & \bs{A}_k''' \\ 
(\bs{A}_k''') \trans & \bs{A}_k'' \end{pmatrix}.
\end{equation}

For all $t \in (0,+\infty)$, $\bs{B}(t)$ can be decomposed in a similar manner (here the $'$ notation is used to distinguish the blocks, not to express some derivative with respect to $t$):
\begin{equation}
\tilde{\bs{B}}(t) = \begin{pmatrix} \bs{B}(t)' & \bs{B}(t)''' \\ 
(\bs{B}(t)''') \trans & \bs{B}(t)'' \end{pmatrix}.
\end{equation}

Now, for any symmetric nonsingular matrix

\begin{equation} \label{Eq:blocks}
\bs{C} = \begin{pmatrix} \bs{C}' & \bs{C}''' \\ (\bs{C}''') \trans & 
\bs{C}'' \end{pmatrix},
\end{equation}

denoting by $\bs{S} := \left\{\bs{C}' - \bs{C}''' \left(\bs{C}''\right)^{-1} \left(\bs{C}'''\right) \trans \right\}$ the Schur complement of $\bs{C}''$, the inverse of $\bs{C}$ is 

\begin{equation} \label{Eq:inverse_Schur_complement}
\bs{C}^{-1} = \begin{pmatrix} \bs{I} & \bs{0} \\ -  \left(\bs{C}''\right)^{-1} \left(\bs{C}'''\right) \trans & \bs{I} \end{pmatrix}
\begin{pmatrix} \bs{S}^{-1} & \bs{0} \\ \bs{0} & \left(\bs{C}''\right)^{-1} \end{pmatrix}
\begin{pmatrix} \bs{I} & - \bs{C}''' \left(\bs{C}''\right)^{-1} \\ \bs{0} & \bs{I} \end{pmatrix}.
\end{equation}

See the section about Block Factorization in \citet{Ser02} (p. 138-139) for explanations about how Equation \eqref{Eq:inverse_Schur_complement} is obtained. \medskip

For every $k \in [\![0,N-1]\!]$, $\bs{A}_k''$ and $\bs{A}_k'''$ are null (note however that $\bs{A}_N''$ is nonsingular, otherwise $\cap_{k=0}^N \Ker \matrice{A}_k$ would be nontrivial). For all $t \in (0,T)$, $\tilde{\bs{M}}(t)$ is nonsingular. Its lower $(m-d_N) \times (m-d_N)$ block is $\left\{a_N(t) \bs{A}_N'' + \bs{B}(t)''\right\}$ and its Schur complement $\bs{S}_N(t)$ is

\begin{equation} \label{Eq:Schur_complement}
\begin{split}
\bs{S}_N(t) &:= \left\{ \sum_{k=0}^{N} a_k(t) \bs{A}_k' + \bs{B}(t)' \right\} - \\
&  \left\{ a_N(t) \bs{A}_N''' + \bs{B}(t)'''\right\} \left\{a_N(t) \bs{A}_N'' + \bs{B}(t)'' \right\}^{-1} \left\{a_N(t) \bs{A}_N''' + \bs{B}(t)''' \right\} \trans.
\end{split}
\end{equation}

Because we are dealing with the finite dimensional vector space of matrices of size $m \times m$, all norms are equivalent: for two norms $\| \cdot \|_1$ and $\| \cdot \|_2$, there exist positive constants $\kappa \leqslant \kappa'$ such that for any matrix $\matrice{X}$ of size $m \times m$, 

$$
\kappa \| \matrice{X} \|_1
\leqslant
\| \matrice{X} \|_2
\leqslant
\kappa' \| \matrice{X} \|_1.
$$

In particular, the Frobenius norm is equivalent to the algebra norm $$\bs{A} \mapsto \sup \left\{\sqrt{ \bs{\xi} \trans \bs{A} \trans \bs{A} \bs{\xi} / \bs{\xi} \trans \bs{\xi} } \; : \; \bs{\xi} \in \R^m \setminus \{\bs{0}_m\} \right\}.$$

So there exists a constant $C_m \in (0,+\infty)$ such that for every $t \in (0,T)$,

\begin{align} \label{Eq:bound_inverse_matrix}
\left\|\bs{M}(t)^{-1} \right\| \leqslant C_m \left( \|\bs{I}_m\| + \left\| \left\{ a_N(t) \bs{A}_N''' + \bs{B}(t)'''\right\} \left\{a_N(t) \bs{A}_N'' + \bs{B}(t)'' \right\}^{-1} \right\| \right)^2 \nonumber \\
\left( \left\|\bs{S}_N(t)^{-1} \right\| + \left\| \left\{a_N(t) \bs{A}_N'' + \bs{B}(t)'' \right\}^{-1} \right\| \right).
\end{align}

$\bs{A}_N''$ is nonsingular, otherwise $\cap_{k=0}^N \Ker \bs{A}_k$ would be nontrivial. This means that the norm of the matrix $ \left\{ a_N(t) \bs{A}_N''' + \bs{B}(t)'''\right\} \left\{a_N(t) \bs{A}_N'' + \bs{B}(t)'' \right\}^{-1}$ 
is bounded when $t \to 0$. 
Because of Equation \eqref{Eq:bound_inverse_matrix}, this implies that there exists $T_N>0$ and $\lambda_N>0$ such that for all $t \in (0,T_N)$, 

\begin{equation} \label{Eq:decrement}
\lambda_N \left\|\bs{M}(t)^{-1} \right\| \leqslant |a_N(t)|^{-1} + \left\| \bs{S}_N(t)^{-1} \right\|.
\end{equation}

Our goal is to use Equation (\ref{Eq:decrement}) recursively, by having $\bs{S}_N(t)$ play the part of $\bs{M}(t)$. To achieve this, a new expression of $\bs{S}_N(t)$ is required.

\begin{equation} \label{Eq:Maclaurin_lower_rank}
\bs{S}_N(t) = \sum_{k=0}^{N-1} a_k(t) \bs{A}_k' + \bs{B}_N(t),
\end{equation}

where

\begin{equation} \label{Eq:error_term}
\begin{split}
\bs{B}_N(t)&:= a_N(t) \bs{A}_N' + \bs{B}(t)'  - \\
& \left\{ a_N(t) \bs{A}_N''' + \bs{B}(t)'''\right\} \left\{a_N(t) \bs{A}_N'' + \bs{B}(t)'' \right\}^{-1} \left\{a_N(t) \bs{A}_N''' + \bs{B}(t)''' \right\} \trans.
\end{split}
\end{equation}

It turns out that when $t \to 0$, the norm of $\bs{B}_N(t)$ is 
$O(|a_N(t)|)$. This is due to the fact mentioned above that 
$\left\| \left\{ a_N(t) \bs{A}_N''' + \bs{B}(t)'''\right\} 
\left\{a_N(t) \bs{A}_N'' + \bs{B}(t)'' \right\}^{-1} \right\|$
is bounded when $t \to 0$. \medskip

 Furthermore, $\cap_{k=0}^{N-1} \Ker \bs{A}_k'$ is the trivial vector space.
 Indeed, let $\bs{v}_1 \in \cap_{k=0}^{N-1} \Ker \bs{A}_k'$.
 Then for any vector $\bs{v}_2 \in \R^{m-d_N}$,
 $(\bs{v}_1,\bs{v}_2) \trans \in \cap_{k=0}^{N-1} \Ker \tilde{\bs{A}}_k$.
 Independently from this, for any vector $\bs{v}_3 \in \R^{d_N}$, $(\bs{v}_3,\bs{0}_{m-d_N}) \trans$ belongs to the orthogonal complement of $ \cap_{k=0}^{N-1} \Ker \tilde{\bs{A}}_k$. So $(\bs{v}_1,\bs{0}_{m-d_N}) \trans$ belongs both to $\cap_{k=0}^{N-1} \Ker \tilde{\bs{A}}_k$ and its orthogonal complement: it is the null vector. Therefore $\bs{v}_1=\bs{0}_{d_N}$. \medskip
 
The two paragraphs above show that Equation (\ref{Eq:Maclaurin_lower_rank}) is formally similar to Equation (\ref{Eq:Maclaurin}): the role of $\bs{M}(t)$ is held by $\bs{S}_N(t)$, the role of $N$ by $N-1$, the role of the $\bs{A}_k$s by the $\bs{A}_k'$s and the role of $\bs{B}(t)$ by $\bs{B}_N(t)$. \medskip

Therefore an equation similar to \eqref{Eq:decrement} can be derived: there exist $T_{N-1}>0$ and $\lambda_{N-1}>0$ such that for all $t \in (0,T_{N-1})$, 

\begin{equation}
\lambda_{N-1} \left\|\bs{S}_N(t)^{-1} \right\| \leqslant |a_{N-1}(t)|^{-1} + \left\| \bs{S}_{N-1}(t)^{-1} \right\|.
\end{equation}

Here, $\bs{S}_{N-1}(t)$ is defined with respect to $\bs{S}_N(t)$ the same way $\bs{S}_N(t)$ was defined with respect to $\bs{M}(t)$. \medskip

Recursive application of this reasoning until 0 is reached yields the result.
\end{proof}

\begin{lem} \label{Lem:minoration_yinversey_Maclaurin}
Consider \eqref{Eq:Maclaurin}. 
If $\cap_{k=0}^N \Ker \bs{A}_k$ is the trivial vector space, 
if the vector space $\cap_{k=0}^{N-1} \Ker \bs{A}_k$ is nontrivial, 
and if there exists $T>0$ such that for all $t \in (0,T)$, 
$\bs{M}(t)$ is positive definite, 
then for any vector $\vecteur{v} \in \R^m$ that does not belong to the
vector space spanned by the columns of the matrices $\matrice{A}_k$
($1 \leqslant k \leqslant N-1$),

\begin{equation}
\liminf_{t \to 0} \bs{v} \bs{M}(t)^{-1} \bs{v} / \left\| \bs{M}(t)^{-1} \right\| > 0.
\end{equation} 
\end{lem}

Lemma \ref{Lem:minoration_yinversey_Maclaurin} is used in the proof of Lemma \ref{Lem:upper_bound_yfactor_in_likelihood} and thus contributes to the proof of Theorem \ref{THM:REFERENCE_POSTERIOR_PROPER}.

\begin{proof}
This result is trivial if $N=0$. If $N \geqslant 1$, it follows from the proof of Lemma \ref{Lem:majoration_inverse_Maclaurin}. Indeed, the requirements of this lemma are stronger than those of Lemma \ref{Lem:majoration_inverse_Maclaurin}, so all intermediate results of its proof are valid. 
Consider the right-hand side of Equation (\ref{Eq:inverse_Schur_complement}) while assuming $\bs{C}$ is positive definite. The matrices on the left-hand side and on the right-hand side are the transpose of one another, so the middle matrix is necessarily positive definite. In particular, both $\bs{S}^{-1}$ and $\left(\bs{C}''\right)^{-1}$ are positive definite. Any vector $\bs{v} \in \R^m$ can be decomposed as $\bs{v} = ((\bs{v}')\trans,(\bs{v}'')\trans) \trans$ with $\bs{v}' \in \R^{d_N}$ and $\bs{v}'' \in \R^{m-d_N}$. This decomposition yields a lower bound: $\bs{v} \trans \bs{C}^{-1} \bs{v} \geqslant \left(\bs{v}''\right) \trans \left(\bs{C}''\right)^{-1} \bs{v}''$. 
Here, $\bs{C}$ is $\bs{M}(t)$, $\bs{S}$ is $\bs{S}_N(t)$ and $\bs{C}''$ is $a_N(t) \bs{A}_N'' + \bs{B}''(t)$. Let us recall that $\bs{A}_N''$ is nonsingular and $\| \bs{B}''(t) \| = o(|a_N(t)|)$ when $t \to 0$.
So as long as  $\bs{v}$ is not orthogonal to $\cap_{k=0}^{N-1} \Ker \bs{A}_k$, $\bs{v}''$ is non-zero and there exists $\tilde{\lambda}_N(\bs{v})>0$ such that when $t$ is small enough, $\bs{v} \trans \bs{M}(t)^{-1} \bs{v} \geqslant \tilde{\lambda}_N(\bs{v}) |a_N(t)|^{-1}$. 
Then Lemma \ref{Lem:majoration_inverse_Maclaurin} yields the result.

\end{proof}

In order to be able to state the next lemmas, consider the vector spaces $V_0, \dots, V_N$ recursively defined as follows:

\begin{equation}
  V_N := \cap_{k=0}^{N-1} \Ker \matrice{A}_k.
\end{equation}

For every positive integer $K$ smaller or equal to $N-1$:

\begin{equation}
  V_K := \left( \cap_{k=0}^{K-1} \Ker \matrice{A}_k \right) \bigcap \left( \cap_{k=K+1}^{N} V_k^\perp \right).  
\end{equation}

Finally, define
\begin{equation}
  V_0 := \cap_{k=1}^{N} V_k^\perp.
\end{equation}

By construction, we have

\begin{equation}
  \R^m = V_0 \overset{\perp}{\oplus} V_1 \overset{\perp}{\oplus} \dots \overset{\perp}{\oplus} V_{N}.
\end{equation}

\begin{lem} \label{Lem:developpement_asymptotique_defini_positif}
  Consider \eqref{Eq:Maclaurin}. 
  Assume that $ \cap_{k=0}^{N} \Ker \matrice{A}_k$ is the trivial vector space.
  If there exists $T>0$ such that for all positive $t<T$, 
  $\matrice{M}(t)$ is positive definite, then 
  for every nonnegative integer $k \leqslant N$, 
  for all sufficiently small $t>0$, we have for every vector $\vecteur{\xi}_k \in V_k \setminus \{ \vecteur{0} \}$ $a_k(t) \vecteur{\xi}_k \trans \matrice{A}_k \vecteur{\xi}_k > 0$.
\end{lem}

Lemma \ref{Lem:developpement_asymptotique_defini_positif} gives a sufficient condition for one of the assumptions of Lemma \ref{Lem:encadrement_developpement_asymptotique} below. Lemma \ref{Lem:encadrement_developpement_asymptotique} is used in the proof of Lemma \ref{Lem:majoration_derivee_RQ_SE} and thus contributes to the proof of  Proposition \ref{Prop:ref_prior_behavior}.

\begin{rmq}
  Note that for any nonnegative integer $k \leqslant N$, 
  the assertion ``for all sufficiently small $t>0$, we have for every vector $\vecteur{\xi}_k \in V_k \setminus \{ \vecteur{0} \}$ $a_k(t) \vecteur{\xi}_k \trans \matrice{A}_k \vecteur{\xi}_k > 0$'' implies that one of the following assertions is true:

  \begin{itemize}
    \item For all $\vecteur{\xi}_k \in V_k \setminus \{\vecteur{0}\}$, $  \vecteur{\xi}_k \trans \matrice{A}_k \vecteur{\xi}_k >0$.
    \item For all $\vecteur{\xi}_k \in V_k \setminus \{\vecteur{0}\}$, $  \vecteur{\xi}_k \trans \matrice{A}_k \vecteur{\xi}_k <0$.
  \end{itemize}
\end{rmq}

\begin{proof}
Assume that $ \cap_{k=0}^{N} \Ker \matrice{A}_k$ is the 
trivial vector space and that there exists $T>0$ such that 
for all positive $t<T$, $\matrice{M}(t)$ is positive definite.
\medskip

If the conclusion of 
Lemma \ref{Lem:developpement_asymptotique_defini_positif} 
does not hold,
then there exist a nonnegative integer $K \leqslant N$, a sequence $(t_l)_{l \in \N}$ of non-null real numbers converging to 0 and a sequence $(\vecteur{\xi}_{K,l})_{l \in \N}$ of vectors of $V_K$ satisfying $\vecteur{\xi}_{K,l} \trans \vecteur{\xi}_{K,l} = 1$ such that 
$a_K(t_l) \vecteur{\xi}_{K,l} \trans \matrice{A}_K \vecteur{\xi}_{K,l} \leqslant 0$ for all $l \in \N$.
\medskip

Let $\matrice{W}_K$ be a matrix with $m$ rows and with number of columns equal to the dimension of $V_0 + \dots + V_K$ such that for any vector $\matrice{\xi} \in V_0 + \dots + V_K$, $\matrice{W}_K \matrice{W}_K \trans \vecteur{\xi} = \vecteur{\xi}$. 
For any sufficiently great $l \in \N$, $\matrice{M}(t_l)$ is positive definite. 
Given that the kernel of $\matrice{W}_K$ is the trivial vector space, $\matrice{W}_K \trans \matrice{M}(t_l) \matrice{W}_K$ is positive definite as well. 
Moreover, $\cap_{k=0}^K \Ker \matrice{W}_K \trans \matrice{A}_k \matrice{W}_K$ is the trivial vector space. 
This means that Lemma \ref{Lem:majoration_inverse_Maclaurin} is applicable and $\left\| \left( \matrice{W}_K \matrice{M}(t_l) \matrice{W}_K \right)^{-1} \right\| = O(|a_K(t_l)|^{-1})$ when $l \to +\infty$.
\medskip

However, for all $l \in \N$,
\begin{equation}
  a_K(t_l) \vecteur{\xi}_{K,l} \trans \matrice{A}_K \vecteur{\xi}_{K,l} = a_K(t_l) \vecteur{\xi}_{K,l} \trans \matrice{W}_K \left( \matrice{W}_K \trans \matrice{A}_K \matrice{W}_K \right) \matrice{W}_K \trans \vecteur{\xi}_{K,l} \leqslant 0. 
\end{equation}
Since for every nonnegative integer $k < K$ and every nonnegative integer $l$, $ \vecteur{\xi}_{K,l} \trans \matrice{W}_K \left( \matrice{W}_K \trans \matrice{A}_k \matrice{W}_K \right) \matrice{W}_K \trans \vecteur{\xi}_{K,l} = 0$,
this implies that when $l \to +\infty$,
\begin{equation}
  a_K(t_l) \vecteur{\xi}_{K,l} \trans 
  \matrice{W}_K
  \left(
    \matrice{W}_K \trans \matrice{M}(t_l)
    \matrice{W}_K
  \right)
  \matrice{W}_K \trans \vecteur{\xi}_{K,l} = o(|a_{K}(t_l)|).
\end{equation}

This contradicts the earlier result that when $l \to +\infty$,

\begin{equation}
  \left\| \left( \matrice{W}_K \matrice{M}(t_l) \matrice{W}_K \right)^{-1} \right\| = O(|a_K(t_l)|^{-1}).
\end{equation}

Therefore the conclusion of 
Lemma \ref{Lem:developpement_asymptotique_defini_positif} must hold.
\end{proof}

\begin{lem} \label{Lem:encadrement_developpement_asymptotique}
  Consider \eqref{Eq:Maclaurin}. Assume that $ \cap_{k=0}^{N} \Ker \matrice{A}_k$ is 
  the trivial vector space. 
  Further assume that for any nonnegative integer $k \leqslant N$, 
  for all sufficiently small $t>0$ and for all $\tilde{\vecteur{\xi}}_k \in V_k \setminus \{\vecteur{0}\}$,
  $a_k(t) \tilde{\vecteur{\xi}}_k \trans \matrice{A}_k \tilde{\vecteur{\xi}}_k > 0$.
  Then, for any $\epsilon>0$, there exists a real number $T>0$ such that for all $t \in (0,T)$
  and
  for any vector $\vecteur{\xi} \in \R^m$, letting $\vecteur{\xi} = \vecteur{\xi}_0 + \dots + \vecteur{\xi}_{N}$ be its unique decomposition
  according to the subspaces $V_0, \dots, V_{N}$,
  \begin{equation}
    (1-\epsilon) \sum_{k=0}^{N} a_k(t) \vecteur{\xi}_k \trans \matrice{A}_k \vecteur{\xi}_k
    \leqslant
    \vecteur{\xi} \trans \matrice{M}(t) \vecteur{\xi}
    \leqslant
    (1+\epsilon) \sum_{k=0}^{N} a_k(t) \vecteur{\xi}_k \trans \matrice{A}_k \vecteur{\xi}_k.
  \end{equation}
\end{lem} 

Lemma \ref{Lem:encadrement_developpement_asymptotique} is used in the proof of Lemma \ref{Lem:majoration_derivee_RQ_SE} and thus contributes to the proof of  Proposition \ref{Prop:ref_prior_behavior}.

\begin{proof}
For any nonnegative integers $k$, $m$ and $m'$ that are smaller or equal to $N$, if either $m>k$ or $m'>k$, then

\begin{equation}
  \vecteur{\xi}_m \trans \matrice{A}_k \vecteur{\xi}_{m'} = 0.
\end{equation}

This means that 
\begin{align}
  \vecteur{\xi} \trans \matrice{M}(t) \vecteur{\xi}
  &=
  \underbrace{\sum_{k=0}^N a_k(t) \vecteur{\xi}_k \trans \matrice{A}_k \vecteur{\xi}_k}_{f_1(t)}
  + \nonumber \\
  & \quad
  \underbrace{\sum_{k=0}^N \left( \sum_{\substack{0 \leqslant m,m' \leqslant k \leqslant N \\ m \neq k \, \mathrm{or} \, m' \neq k}} a_k(t) \vecteur{\xi}_m \trans \matrice{A}_k \vecteur{\xi}_{m'} \right)}_{f_2(t)}
  +
  \underbrace{\sum_{0 \leqslant m,m' \leqslant N} \vecteur{\xi}_m \trans \matrice{B}(t) \vecteur{\xi}_{m'}}_{f_3(t)}
  .
\end{align}

Let us examine every term in $f_2(t)$ and $f_3(t)$. 

To do this, define 

\begin{equation}
  d_k^{min} := \min \{ |\vecteur{\xi}_k \trans \matrice{A}_k \vecteur{\xi}_k| : \vecteur{\xi}_k \in V_k, \vecteur{\xi}_k \trans \vecteur{\xi}_k = 1\} \nonumber .
\end{equation}

And then

\begin{equation}
  d^{min} := \min \{ d_k^{min} : k \in \N, 0 \leqslant k \leqslant N \} \nonumber . 
\end{equation}

Due to the assumption that for any integer $k \in [\![0,N]\!]$, 
for all sufficiently small $t>0$, 
$a_k(t) \vecteur{\xi}_k \trans \matrice{A}_k \vecteur{\xi}_k > 0$, 

\begin{equation}
  d^{min} > 0.
\end{equation}

For good measure, also define

\begin{equation}
  d^{max} := \max \{ \| \matrice{A}_k \| : k \in \N, 0 \leqslant k \leqslant N \}. \nonumber
\end{equation}

Choose a real number $\epsilon > 0$. \medskip

For any nonnegative intergers $m, m' \leqslant N$, 

\begin{align}
  | \vecteur{\xi}_m \trans \matrice{B}(t) \vecteur{\xi}_{m'} |
  \leqslant
  \|\matrice{B}(t)\| \| \vecteur{\xi}_m \| \| \vecteur{\xi}_{m'} \|
  \leqslant
  \|\matrice{B}(t)\| (\vecteur{\xi}_m \trans \vecteur{\xi}_m + \vecteur{\xi}_{m'} \trans \vecteur{\xi}_{m'}).
\end{align}

Because $\|\matrice{B}(t)\| = o(|a_N(t)|)$, there exists $T_B>0$ such that for all $t \in (0,T_B)$ and for any nonnegative integer $K \leqslant N$:

\begin{equation}
  \|\matrice{B}(t)\| \leqslant |a_K(t)|  d^{min} \epsilon .
\end{equation}

It follows that for all $t \in (0,T_B)$:

\begin{align}
  \|\matrice{B}(t)\| (\vecteur{\xi}_m \trans \vecteur{\xi}_m + \vecteur{\xi}_{m'} \trans \vecteur{\xi}_{m'})
  \leqslant
  \epsilon \left( a_m(t) \vecteur{\xi}_m \trans \matrice{A}_m \vecteur{\xi}_m  + a_{m'}(t) \vecteur{\xi}_{m'} \trans \matrice{A}_{m'} \vecteur{\xi}_{m'} \right).
\end{align}

Therefore, for all $t \in (0,T_B)$:

\begin{equation}
  | \vecteur{\xi}_m \trans \matrice{B}(t) \vecteur{\xi}_{m'} |
  \leqslant
  \epsilon \left( a_m(t) \vecteur{\xi}_m \trans \matrice{A}_m \vecteur{\xi}_m  + a_{m'}(t) \vecteur{\xi}_{m'} \trans \matrice{A}_{m'} \vecteur{\xi}_{m'} \right).
\end{equation}

And then, for all  $t \in (0,T_B)$:

\begin{equation} \label{Eq:controle_matrice_B}
  |f_3(t)| \leqslant 
  \epsilon \sum_{0 \leqslant m,m' \leqslant N} a_m(t) \vecteur{\xi}_m \trans \matrice{A}_m \vecteur{\xi}_m  + a_{m'}(t) \vecteur{\xi}_{m'} \trans \matrice{A}_{m'} \vecteur{\xi}_{m'}
  =
  \epsilon 2 (N+1) f_1(t).
\end{equation}

For any nonegative integers $m,m',k \leqslant N$ such that both $m \leqslant k$ and $m' \leqslant k$:

\begin{align}
  | \vecteur{\xi}_m \trans \matrice{A}_k \vecteur{\xi}_{m'} |
  \leqslant
  \|\matrice{A}_k\| \| \vecteur{\xi}_m \| \| \vecteur{\xi}_{m'} \|
  \leqslant
  d^{max} (\vecteur{\xi}_m \trans \vecteur{\xi}_m + \vecteur{\xi}_{m'} \trans \vecteur{\xi}_{m'}).
\end{align}

Let us first consider the case where both $m<k$ and $m'<k$. 
Then, since $|a_k(t)| = o(|a_m(t)|)$ and  $|a_k(t)| = o(|a_{m'}(t)|)$, 
there exists $T_{m,m',k}>0$ such that for all $t \in (0, T_{m,m',k})$,

\begin{align}
  d^{max} |a_k(t)| \leqslant d^{min} \min(|a_m(t)|, |a_{m'}(t)|) \epsilon.  
\end{align}

So, in the case where both $m<k$ and $m'<k$, we have for all $t \in (0, T_{m,m',k})$:

\begin{align}
  | \vecteur{\xi}_m \trans \matrice{A}_k \vecteur{\xi}_{m'} |
  \leqslant
  \epsilon \left( a_m(t) \vecteur{\xi}_m \trans \matrice{A}_m \vecteur{\xi}_m  + a_{m'}(t) \vecteur{\xi}_{m'} \trans \matrice{A}_{m'} \vecteur{\xi}_{m'} \right). 
\end{align}

Let us now consider the case where $m<k$ and $m'=k$ (the case where $m=k$ and $m'<k$ being equivalent since $\matrice{A}_k$ is symmetric).

\begin{align}
  | \vecteur{\xi}_m \trans \matrice{A}_k \vecteur{\xi}_{k} |
  \leqslant
  \|\matrice{A}_k\| \| \vecteur{\xi}_m \| \| \vecteur{\xi}_{k} \|
  \leqslant
  d^{max} \| \vecteur{\xi}_m \| \| \vecteur{\xi}_{k} \|.
\end{align}

If $\| \vecteur{\xi}_m \| \leqslant  \epsilon (d^{min} / d^{max})  \| \vecteur{\xi}_{k} \|$, then:

\begin{align} \label{Eq:xim_petit}
  | \vecteur{\xi}_m \trans \matrice{A}_k \vecteur{\xi}_{k} |
  \leqslant
  d^{min} \vecteur{\xi}_k \trans \vecteur{\xi}_k \epsilon 
  \leqslant
  \vecteur{\xi}_k \trans \matrice{A}_k \vecteur{\xi}_k \epsilon.
\end{align}

If $\| \vecteur{\xi}_m \| >  \epsilon (d^{min} / d^{max})  \| \vecteur{\xi}_{k} \|$, then:

\begin{align}
  | \vecteur{\xi}_m \trans \matrice{A}_k \vecteur{\xi}_{k} |
  <
  \frac{d^{max}}{\epsilon (d^{min} / d^{max})} \vecteur{\xi}_m \trans \vecteur{\xi}_m.
\end{align}

Since $|a_k(t)| = o(|a_m(t)|)$, there exists $T_{m,k}>0$ such that for all $t \in (0, T_{m,k})$,

\begin{align}
  \frac{d^{max}}{\epsilon (d^{min} / d^{max})} |a_k(t)| \leqslant d^{min} |a_m(t)| \epsilon.  
\end{align}

Therefore, if $\| \vecteur{\xi}_m \| >  \epsilon (d^{min} / d^{max})  \| \vecteur{\xi}_{k} \|$, provided $t \in (0, T_{m,k})$,

\begin{equation} \label{Eq:xim_grand}
  |a_k(t)| | \vecteur{\xi}_m \trans \matrice{A}_k \vecteur{\xi}_{k} |
  <
  a_m(t) d^{min} \vecteur{\xi}_m \trans \vecteur{\xi}_m \epsilon
  \leqslant
  a_m(t) \vecteur{\xi}_m \trans \matrice{A}_m \vecteur{\xi}_m \epsilon.
\end{equation}

Putting together \eqref{Eq:xim_petit} and \eqref{Eq:xim_grand}, we obtain that for all $\vecteur{\xi} \in \R^m$, for all $t \in (0, T_{m,k})$,

\begin{align}
  | \vecteur{\xi}_m \trans \matrice{A}_k \vecteur{\xi}_{k} |
  \leqslant
  \epsilon \left( a_m(t) \vecteur{\xi}_m \trans \matrice{A}_m \vecteur{\xi}_m  + a_{k}(t) \vecteur{\xi}_{k} \trans \matrice{A}_{k} \vecteur{\xi}_{k} \right). 
\end{align}

Define 

\begin{align}
  T_A := \min( &\min\{ T_{m,m',k} | m,m',k \in \N ;  0 \leqslant m,m' < k \leqslant N \}, \nonumber \\
  & \min\{ T_{m,k} | m, k \in \N ; 0 \leqslant m < k \leqslant N\}). \nonumber
\end{align}

For all $t \in (0,T_A)$, 

\begin{align}
  |f_2(t)| & \leqslant
  \epsilon 
  \sum_{k=0}^N \left(
    \sum_{\substack{0 \leqslant m,m' \leqslant k \leqslant N \\ m \neq k \, \mathrm{or} \, m' \neq k}} 
    \left( a_m(t) \vecteur{\xi}_m \trans \matrice{A}_m \vecteur{\xi}_m  + a_{m'}(t) \vecteur{\xi}_{m'} \trans \matrice{A}_{m'} \vecteur{\xi}_{m'} \right)
  \right) \nonumber \\
  & \leqslant
  \epsilon (N+1) \sum_{0 \leqslant m,m' \leqslant N}
  \left( a_m(t) \vecteur{\xi}_m \trans \matrice{A}_m \vecteur{\xi}_m  + a_{m'}(t) \vecteur{\xi}_{m'} \trans \matrice{A}_{m'} \vecteur{\xi}_{m'} \right) \nonumber \\
  & \leqslant
  \epsilon 2 (N+1)^2 f_1(t). \label{Eq:controle_matrice_A}
\end{align}

Defining $T:=\min(T_A, T_B)$, Equations \eqref{Eq:controle_matrice_B} and \eqref{Eq:controle_matrice_A} yield that for all $t \in (0,T)$, 

\begin{equation}
  |f_2(t)| + |f_3(t)| \leqslant \epsilon 2 (N+2)^2 f_1(t).
\end{equation}

As $\epsilon$ can be taken arbitrarily small, after redefining 
$\epsilon := \epsilon / (2(N+2)^2)$, this yields the result.
\end{proof}

\subsection{Bound for \texorpdfstring{$\left\|\co^{-1}\right\|$}{Lg} with Matérn kernels}

In order to derive a bound for $\left\|\co^{-1}\right\|$ 
in the case of Matérn kernels (Lemma \ref{Lem:Matern_decay}), 
the spectral representation of the kernels is convenient.
\medskip

To use it, we need this preliminary Bochner-type result:

\begin{lem} \label{Lem:Bochner+}
Let $\mu$ be a positive measure on $\R^r$ with finite non-null total mass that is absolutely continuous with respect to the Lebesgue measure.
Then the mapping $K:\R^r \rightarrow \R$ defined by 

\begin{equation}
K(\bs{x}) = \int_{\R^r} e^{i \langle \bs{\omega} | \bs{x} \rangle } d \mu(\bs{\omega})
\end{equation}

is positive definite. Moreover, for any $\bs{\xi} \in \R^n \setminus \{\bs{0}_n\}$,

\begin{equation}
\sum_{k,l \in [\![1,n]\!]} \xi_k \xi_l K(\bs{x}^{(k)} - \bs{x}^{(l)}) > 0.
\end{equation}
  \end{lem}
  
\begin{proof}
  The first part results from Bochner's theorem. Let us show the second.
  
  \begin{align}
  \sum_{k,l \in [\![1,n]\!]} \xi_k \xi_l K(\bs{x}^{(k)} - \bs{x}^{(l)}) &= \sum_{k,l \in [\![1,n]\!]} \xi_k \xi_l \int_{\R^r} e^{i \langle \bs{\omega} | \bs{x}^{(k)} - \bs{x}^{(l)} \rangle } d \mu(\bs{\omega})  \nonumber \\
  &= \int_{\R^r} \left| \sum_{k=1}^n \xi_k e^{i \langle \bs{\omega} | \bs{x}^{(k)} \rangle } \right|^2 d \mu(\bs{\omega}).
  \end{align}
  
  Given $\bs{x}^{(1)},...,\bs{x}^{(n)}$ are all distinct, 
  for almost all unitary vectors $\bs{u}$ in the sense of the Lebesgue measure on the unit sphere $S^{r-1}$, the real numbers $\langle \bs{u} | \bs{x}^{(1)} \rangle,...,\langle \bs{u} | \bs{x}^{(n)} \rangle$ are distinct. Indeed, if two of these numbers, say $\langle \bs{u} | \bs{x}^{(1)} \rangle$ and $\langle \bs{u} | \bs{x}^{(2)} \rangle$, were equal, then $\bs{u}$ would be orthogonal to $\bs{x}^{(1)}-\bs{x}^{(2)}$. But the set of all vectors of $\R^r$ orthogonal to $\bs{x}^{(1)}-\bs{x}^{(2)}$ is a hyperplane. So there exists a finite number of hyperplanes of $\R^r$ such that, if $\bs{u}$ does not belong to any of them, the real numbers $\langle \bs{u} | \bs{x}^{(1)} \rangle,...,\langle \bs{u} | \bs{x}^{(n)} \rangle$ are distinct. \medskip
  
  Now, notice that the mapping $\mathbb{C} \rightarrow \mathbb{C}; z \mapsto \sum_{k=1}^n \xi_k e^{i z \langle \bs{u} | \bs{x}^{(k)} \rangle }$ is holomorphic. So either it is the null function or all its zeros are isolated. Given it clearly is not the null function, its zeros are isolated. So the set of all zeros that belong to $\R$ is countable and therefore of null Lebesgue measure. \medskip
    
  Let $f: \R^r \rightarrow \{0,1\}$ be the measurable mapping such that $f(\bs{\omega})=1$ if $\sum_{k=1}^n \xi_k e^{i \langle \bs{\omega} | \bs{x}^{(k)} \rangle } = 0$ and $f(\bs{\omega})=0$ if not.
  
  \begin{equation}
  \int_{\R^r} f(\bs{\omega}) d \bs{\omega} = \frac{2 \pi^{\frac{r}{2}}}{\Gamma \left( \frac{r}{2}\right)}\int_{S^{r-1}} \int_{(0,+\infty)} f(t \bs{u}) t^{r-1} dt d\bs{u} = \frac{2 \pi^{\frac{r}{2}}}{\Gamma \left( \frac{r}{2}\right)}\int_{S^{r-1}} 0 \, d \bs{u} = 0.
  \end{equation}
  
  Therefore the mapping $\R^r \rightarrow \mathbb{C}$; $\bs{\omega} \mapsto \sum_{k=1}^n \xi_k e^{i \langle \bs{\omega} | \bs{x}^{(k)} \rangle }$ takes null values on a Borel set that is negligible with respect to the Lebesgue measure. This set is therefore also negligible with respect to $\mu$, which yields the conclusion.
\end{proof}

\begin{lem} \label{Lem:Matern_spectre}
  For a Matérn kernel with smoothness $\nu$, for all $\theta \in (0,+\infty)$ and for any vector $\vecteur{\xi} \in \R^n$,

  \begin{equation} \label{Eq:spectral_decomposition}
    \forall \bs{\xi} \in \R^n, \;
    \bs{\xi} \trans \co \bs{\xi}
    = M_r \theta^r I_\theta (\bs{\xi}),
  \end{equation}
  
  where

  \begin{align}
    M_r &= \frac{\Gamma ( \nu + \frac{r}{2} ) (2 \sqrt{\nu} )^{2 \nu } }{ \pi^{\frac{r}{2} } \Gamma(\nu) } ; \\
    I_\theta(\bs{\xi}) &= \int_{\R^r} 
    \left( 4 \nu + \theta^2 \| \bs{s} \|^2 \right)^{-\frac{r}{2} - \nu}
    \left| \sum_{j=1}^n \xi_j e^{i \langle \left. \bs{s} \right| 
    \bs{x}^{(j)} \rangle } \right|^2 d \bs{s}.
  \end{align}
\end{lem}

\begin{proof}
  Let us set up a few notations. First, let $K$ denote the Matérn kernel with parameter $\nu$  and $\widehat{K}_{r}$ its $r$-dimensional Fourier transform:

  \begin{equation}
  \widehat{K}_{r} (\bs{\omega}) = (2 \pi)^{-r} \int_{\R^r} K (\|\bs{x}\|) e^{- i \left\langle \bs{\omega} \left| \bs{x} \right. \right\rangle } d \bs{x}
  \quad \mathrm{and} \quad
  K(\|\bs{x}\|) = \int_{\R^r} \widehat{K}_{r} ( \bs{\omega} ) e^{ i \left\langle \bs{\omega} \left| \bs{x} \right. \right\rangle } d \bs{\omega}.
  \end{equation}

  $\widehat{K}_{r} (\bs{\omega})$ has a straightforward expression \citep{Ras06}:

  \begin{equation} \label{Eq:Matern_explicit_Fourier_transform}
    \widehat{K}_{r} (\bs{\omega}) = M_r (4\nu + \|\bs{\omega}\|^2)^{-\frac{r}{2}-\nu}.
  \end{equation}
  
  For all $\theta \in (0,+\infty)$, using the correlation kernel $K_\theta(\cdot) = K(\cdot/\theta)$, the correlation matrix $\co$ is such that:
  
  \begin{equation} 
  \forall \bs{\xi} \in \R^n, \;
  \bs{\xi} \trans \co \bs{\xi}
  = \sum_{j,k=1}^n \xi_j \xi_k  K \left(\frac{ \left\| \bs{x}^{(j)} - \bs{x}^{(k)} \right\| }{\theta} \right) 
  = \int_{\R^r} \widehat{K_{r}} (\bs{\omega}) 
  \left| \sum_{j=1}^n \xi_j e^{i \left\langle \bs{\omega}
  \left| \frac{\bs{x}^{(j)}}{\theta} \right. \right\rangle }  \right|^2 d\bs{\omega} .
  \end{equation}

  Plugging \eqref{Eq:Matern_explicit_Fourier_transform} into this equation yields the result. 
\end{proof}

We are now able to prove a fact about Matérn kernels that plays a crucial role in the proof of Proposition \ref{Prop:ref_prior_behavior} and is also used in the proof of Theorem \ref{THM:REFERENCE_POSTERIOR_PROPER}.

\begin{lem} \label{Lem:Matern_decay}
For Matérn kernels, when $\theta \to +\infty$ $\left\| \co^{-1} \right\| = O(\theta^{2\nu})$.
\end{lem}
\begin{proof}
 
We use the notations from the proof of Lemma \ref{Lem:Matern_spectre}. \medskip

For all $\theta \geqslant 2 \sqrt{\nu}$ and all $\vecteur{s} \in \R^r$ such that $\|\vecteur{s}\| \geqslant 1$,

\begin{equation}
  4\nu + \theta^2 \|\vecteur{s}\|^2 
  \leqslant 
  \theta^2(1+\|\vecteur{s}\|^2)
  \leqslant
  \theta^2 2 \|\vecteur{s}\|^2.
\end{equation}

This yields the following lower bound on the quantity $I_\theta(\vecteur{\xi})$, which was defined in Lemma \ref{Lem:Matern_spectre}. \medskip

When $\theta \geqslant 2 \sqrt{\nu}$, for any $\bs{\xi} \in \R^n$, 

\begin{align}
I_\theta(\bs{\xi}) 
&\geqslant \int_{\|\bs{s}\| \geqslant 1} 
\left( 4 \nu + \theta^2 \| \bs{s} \|^2 \right)^{-\frac{r}{2} - \nu} 
\left| \sum_{j=1}^n \xi_j e^{i \langle \left. \bs{s} \right| \bs{x}^{(j)} \rangle } \right|^2 d \bs{s} \nonumber \\
&\geqslant 2^{-\frac{r}{2} - \nu} \theta^{-r -2\nu}
\int_{\|\bs{s}\| \geqslant 1}  \| \bs{s} \|^{-r - 2\nu} 
\left| \sum_{j=1}^n \xi_j e^{i \langle \left. \bs{s} \right| \bs{x}^{(j)} \rangle } \right|^2 d \bs{s}.
\end{align}

Define the mapping $K^{aux}: \R^r \rightarrow \R$ by

\begin{equation}
K^{aux}(\bs{x})= \int_{\|\bs{s}\| \geqslant 1} e^{i \langle \bs{s} | \bs{x} \rangle } \| \bs{s} \|^{-r - 2\nu} \bs{1}_{\|\bs{s}\| \geqslant 1}  d\bs{s}.
\end{equation}

By Lemma \ref{Lem:Bochner+},
the $n \times n$ matrix $\bs{M}$ with $(i,i')$-th element $K^{aux} \left(( \bs{x}^{(i)} - \bs{x}^{(i')} ) \right)$ is positive definite.
For any $\bs{\xi} \in \R^n$,
\begin{equation}
  \int_{\|\bs{s}\| \geqslant 1}  \| \bs{s} \|^{-r - 2\nu} \left| \sum_{j=1}^n \xi_j e^{i \langle \left. \bs{s} \right| \bs{x}^{(i')} \rangle } \right|^2 d \bs{s} = \bs{\xi} \trans \bs{M} \bs{\xi}.
\end{equation}
Let $M$ denote the smallest eigenvalue of $\bs{M}$. For any $\bs{\xi} \in \R^n$, when $\theta \geqslant 2 \sqrt{\nu}$, $I_\theta(\bs{\xi}) \geqslant 2^{-\frac{r}{2} - \nu} M \|\bs{\xi}\|^2  \theta^{-r - 2\nu}$.
Lemma \ref{Lem:Matern_spectre} implies the result.
\end{proof}

The last lemma in this section concerns the derivative of the correlation matrix $\co$ with respect to $\theta$ for Matérn kernels.
It is used in the proof of Proposition \ref{Prop:ref_prior_behavior}.

\begin{lem} \label{Lem:derivative_spectral_decomposition}
  Using the notations from Lemma \ref{Lem:Matern_spectre}, for a Matérn kernel with smoothness $\nu$, for all $\theta \in (0,+\infty)$ and for any $\bs{\xi} \in \R^n$:

  \begin{equation} \label{Eq:derivative_spectral_decomposition}
    \begin{split}
    \bs{\xi} \trans \left( \frac{d}{d\theta} \co \right) \bs{\xi} 
    &= M_r r \theta^{r-1} I_\theta (\bs{\xi}) + M_r \theta^r \frac{d}{d\theta} I_\theta (\bs{\xi}).
    \end{split}
  \end{equation}
\end{lem}

\begin{proof}
This is a corollary of Lemma \ref{Lem:Matern_spectre}.
\end{proof}

\subsection{Asymptotic expansion of the correlation matrix \texorpdfstring{$\co$}{Lg} } \label{Sec:asymptotic_corr}

Results presented in this appendix are essential to the proof of Theorem \ref{THM:REFERENCE_POSTERIOR_PROPER}. 

The first subsection deals with Squared Exponential and Rational Quadratic kernels, the second with Matérn kernels with noninteger smoothness $\nu$ and the third with Matérn kernels with integer smoothness $\nu$. \medskip

\subsubsection{Rational Quadratic and Squared Exponential kernels} \label{Sec:asymptotic_corr_RQ_SE}

\begin{lem} \label{Lem:RQ_SE_asymptotic_expansion}
When $\theta$ is large enough, if a Rational Quadratic kernel or a Squared Exponential kernel is used, 

\begin{equation} \label{Eq:AD_infinitely_differentiable}
\co = \sum_{k=0}^\infty \frac{a_k}{\theta^{2k}} \bs{D}^{(k)}.
\end{equation}

In the expression above, for every $k \in \N$, $\bs{D}^{(k)}$ is the $n \times n$ matrix with $(i,i')$-th element $\| \bs{x}^{(i)} - \bs{x}^{(i')} \|^{2k}$ and $a_k$ is a non-null real number. To be precise, $a_k = (-1)^k \left( \prod_{l=0}^{\textcolor{blue}{k-1}} (\nu+l) \right) / k!$ for Rational Quadratic kernels and $a_k = (-1)^k / k!$ for the Squared Exponential kernel. \medskip
\end{lem}

Lemma \ref{Lem:RQ_SE_asymptotic_expansion} is used to prove Lemmas \ref{Lem:upper_bound_yfactor_in_likelihood} and \ref{Lem:RQ_SE_framework_proof_theorem} and thus indirectly contributes to the proof of Theorem \ref{THM:REFERENCE_POSTERIOR_PROPER}. It also plays a role in Lemma \ref{Lem:majoration_derivee_RQ_SE}, which in turn contributes to the proof of Proposition \ref{Prop:ref_prior_behavior}.

\begin{proof}
  For all $\nu>0$, the series expansion of the mapping $x \mapsto (1 + x)^{-\nu}$ at $x=0$ has radius of convergence 1. Moreover, the series expansion of the exponential function has infinite radius of convergence. The former fact implies the result for Rational Quadratic kernels, the latter for the Squared Exponential kernel.
\end{proof}

\begin{lem} \label{Lem:RQ_SE_framework_proof_theorem}
  For Rational Quadratic and Squared Exponential kernels, $\bs{W} \trans \co \bs{W}$ can be decomposed as
  
  \begin{equation}
  \bs{W} \trans \co \bs{W} = g(\theta) \left( \bs{W} \trans \nouveauD \bs{W} + g^{\star}(\theta) \bs{W} \trans \nouveauD^{\star} \bs{W} + \bs{R}_g(\theta) \right),
  \end{equation}
  
  where 
  
  \begin{itemize}
  \item $g$ is a positive differentiable function on $(0,+\infty)$;
  \item  $g^{\star}(\theta) = \theta^{-2l}$ with $l \in \Z_+$;
  \item $\bs{R}_g$ is a differentiable mapping from $(0,+\infty)$ to $\mathcal{M}_n$ such that $\| \bs{R}_g(\theta) \| = o(\theta^{-2l})$ and $\| \frac{d}{d\theta} \bs{R}_g(\theta) \| = o(\theta^{-2l-1})$;
  \item $\nouveauD$ and $\nouveauD^{\star}$ are both fixed symmetric matrices;
  \item $\bs{W} \trans \nouveauD \bs{W}$ is non-null.
  \end{itemize}
\end{lem}

Lemma \ref{Lem:RQ_SE_framework_proof_theorem} is used in the proof of Theorem \ref{THM:REFERENCE_POSTERIOR_PROPER}.

\begin{proof}
  We use the notations of Lemma \ref{Lem:RQ_SE_asymptotic_expansion}. This lemma implies that

  \begin{equation}
    \bs{W} \trans \co \bs{W}
    = \sum_{k=0}^\infty \frac{a_k}{\theta^{2k}} \bs{W} \trans \bs{D}^{(k)} \bs{W}.
  \end{equation}
    
  $\co$ is positive definite and the kernel of $\bs{W}$ is trivial so $\bs{W} \trans \co \bs{W}$ is positive definite.
  Let $k_1$ be the smallest nonnegative integer such that $\bs{W} \trans \bs{D}^{(k_1)} \bs{W}$ is non-null.
  Define $\nouveauD:=a_{k_1} \bs{D}^{(k_1)}$.
  \medskip
  
  If $\bs{W} \trans \bs{D}^{(k_1)} \bs{W}$ is nonsingular, then define $k_2:=k_1+1$ and $\nouveauD^{\star}:=a_{k_2} \bs{D}^{(k_2)}$.
  \medskip
  
  If $\bs{W} \trans \bs{D}^{(k_1)} \bs{W}$ is singular, then 
  there must exist an integer $k>k_1$ such that 
  $\bs{W} \trans \bs{D}^{(k)} \bs{W}$ is non-null. 
  Otherwise
  $\bs{W} \trans \co \bs{W} = a_{k_1} \theta^{-2k_1} \bs{W} \trans \bs{D}^{(k_1)} \matrice{W}$, which is absurd since $\bs{W} \trans \co \bs{W}$ is nonsingular and $\bs{W} \trans \bs{D}^{(k_1)} \bs{W}$ is singular.
  Let $k_2$ be the smallest of these integers and define $\nouveauD^{\star}:=a_{k_2} \bs{D}^{(k_2)}$.
  Now, define the mappings $g(\theta)=\theta^{-2k_1}$ and $g^{\star}(\theta)=\theta^{-2l}$ with $l=k_2-k_1$.
  Finally, define
  
  \begin{equation}
  \bs{R}_g(\theta) = g(\theta)^{-1} \sum_{k=k_2+1}^\infty \frac{a_k}{\theta^{2k}} \bs{W} \trans \bs{D}^{(k)} \bs{W}
  =
  \sum_{k=l+1}^\infty \frac{a_{k_1+k}}{\theta^{2k}} \bs{W} \trans \bs{D}^{(k_1+k)} \bs{W}.
  \end{equation}
  
  It turns out that $\|\bs{R}_g(\theta)\|=o(\theta^{-2l})$ and $\|\frac{d}{d\theta} \bs{R}_g(\theta)\| = o(\theta^{-2l-1})$.  
    
\end{proof}

\subsubsection{Matérn kernels with noninteger smoothness \texorpdfstring{$\nu$}{Lg}} \label{Sec:asymptotic_corr_Matern}

\begin{lem} \label{Lem:Matern_noninteger_nu_asymptotic_expansion}
If a Matérn kernel with noninteger smoothness $\nu>0$ 
(whether greater or smaller than 1) is used, 
we can write $\co$ as

\begin{equation} \label{Eq:AD_Matern}
\co = \sum_{k=0}^{\floor{\nu}} \frac{a_k}{\theta^{2k}} \bs{D}^{(k)} +  \frac{a_\nu}{\theta^{2\nu}} \bs{D}^{(\nu)} + \frac{a_{\floor{\nu}+1}}{\theta^{2(\floor{\nu}+1)}} \bs{D}^{(\floor{\nu}+1)} + \bs{R}(\theta).
\end{equation}

Here are the notations used:

\begin{itemize}
\item 
For every $k \in \N$, 
$\bs{D}^{(k)}$ is the $n \times n$ matrix with $(i,i')$-th element $\| \bs{x}^{(i)} - \bs{x}^{(i')} \|^{2k}$. 

\item 
For every $k \in \N$, 
$a_k = (-1)^k \Gamma(\nu-k) \nu^k / \left( k! \Gamma(\nu) \right)$.
\item $\bs{D}^{(\nu)}$ is the $n \times n$ matrix with $(i,i')$-th element $\| \bs{x}^{(i)} - \bs{x}^{(i')} \|^{2\nu}$.
\item $a_\nu = \Gamma(-\nu) \nu^\nu / \Gamma(\nu)$.
\item  $\bs{R}$ is a differentiable mapping from $(0,+\infty)$ to the space of real $n \times n$ matrices $\mathcal{M}_n$ that satisfies
$\| \bs{R}(\theta) \| =  o(\theta^{-2(\floor{\nu}+1)})$ and $\| \frac{d}{d \theta} \bs{R}(\theta) \| =  o(\theta^{-2(\floor{\nu}+1)-1})$ when $\theta \to +\infty$.
\end{itemize}
\end{lem}

Lemma \ref{Lem:Matern_noninteger_nu_asymptotic_expansion} serves to prove Lemmas \ref{Lem:upper_bound_yfactor_in_likelihood} and \ref{Lem:Matern_noninteger_smoothness_framework_proof_theorem} and thus indirectly contributes to the proof of Theorem \ref{THM:REFERENCE_POSTERIOR_PROPER}.
It is also used in the proof of Proposition \ref{Prop:small_order_asymptotic_decomposition}.

\begin{proof}
  According to \citet{AS64} (Equations 9.6.2 and 9.6.10), the modified Bessel function of second kind $\mathcal{K}_\nu$ can be written:

  \begin{align}
    & \mathcal{K}_\nu(z) \nonumber \\
    := &\frac{\pi}{2 \sin(\nu \pi)} 
    \left[
      \left(\frac{z}{2}\right)^{-\nu} 
      \sum_{k=0}^\infty \frac{\left(\frac{z^2}{4}\right)^k}{k! \Gamma(-\nu+k+1)}
    -
      \left(\frac{z}{2}\right)^{\nu} 
      \sum_{k=0}^\infty \frac{\left(\frac{z^2}{4}\right)^k}{k! \Gamma(\nu+k+1)}
      \right]  \nonumber \\
    = &\frac{\pi (z/2)^{-\nu}}{2 \sin(\nu \pi)} 
    \left[      
      \sum_{k=0}^\infty \frac{\left(\frac{z^2}{4}\right)^k}{k! \Gamma(-\nu+k+1)}
    -
      \left(\frac{z}{2}\right)^{2\nu} 
      \sum_{k=0}^\infty \frac{\left(\frac{z^2}{4}\right)^k}{k! \Gamma(\nu+k+1)}
    \right].
  \end{align}
  
  The Matérn kernel with noninteger smoothness $\nu$ applied to $z$ is given by: 
  
  \begin{align}
  &\Gamma(\nu)^{-1} 2^{1-\nu} (2 \sqrt{\nu} z)^\nu \mathcal{K}_\nu(2 \sqrt{\nu} z) \nonumber \\ 
  :=&\frac{\pi}{\Gamma(\nu) \sin(\nu \pi)}
  \left[
    \sum_{k=0}^\infty \frac{\nu^k}{k! \Gamma(-\nu+k+1)} z^{2k}
  -
    \nu^\nu \sum_{k=0}^\infty \frac{\nu^k}{k! \Gamma(\nu+k+1)} z^{2(k+\nu)} 
  \right].
  \end{align}

  Now, for any nonnegative integer $k$,

  \begin{equation}
    \Gamma(\nu-k) \Gamma(-\nu+k+1)
    =
    \frac{\pi}{\sin((\nu-k)\pi)}
    =
    \frac{(-1)^k \pi}{\sin(\nu \pi)}.
  \end{equation}

  Therefore
  \begin{align}
  &\Gamma(\nu)^{-1} 2^{1-\nu} (2 \sqrt{\nu} z)^\nu \mathcal{K}_\nu(2 \sqrt{\nu} z) \nonumber \\ 
  :=&
    \sum_{k=0}^\infty \frac{(-1)^k \Gamma(\nu-k) \nu^k}{\Gamma(\nu) k!} z^{2k}
  - \frac{\pi}{\Gamma(\nu) \sin(\nu \pi)}
    \nu^\nu \sum_{k=0}^\infty \frac{\nu^k}{k! \Gamma(\nu+k+1)} z^{2(k+\nu)} 
   \nonumber 
  \end{align}

  Finally,

  \begin{equation}
    \Gamma(-\nu) \Gamma(\nu+1)
    =
    \frac{\pi}{\sin(-\nu\pi)}
    =
    \frac{- \pi}{\sin(\nu \pi)},
  \end{equation} 

  so we get 

  \begin{align}
    &\Gamma(\nu)^{-1} 2^{1-\nu} (2 \sqrt{\nu} z)^\nu \mathcal{K}_\nu(2 \sqrt{\nu} z)  \label{Eq:decomp_matern_nu_non_entier}\\ 
    :=&
      \sum_{k=0}^{\floor{\nu}} \frac{(-1)^k \Gamma(\nu-k) \nu^k}{\Gamma(\nu) k!} z^{2k}
    + \frac{\Gamma(-\nu) \nu^\nu}{\Gamma(\nu)} z^{2\nu} + \nonumber \\    
    & \sum_{k=\floor{\nu}+1}^\infty \frac{(-1)^k \Gamma(\nu-k) \nu^k}{\Gamma(\nu) k!} z^{2k} 
    - \frac{\pi}{\Gamma(\nu) \sin(\nu \pi)}
      \nu^\nu \sum_{k=1}^\infty \frac{\nu^k}{k! \Gamma(\nu+k+1)} z^{2(k+\nu)}. \nonumber
    \end{align}

The result follows after remembering that the $(i,i')$-th element of $\co$ is given in formula \eqref{Eq:decomp_matern_nu_non_entier} with $z = \left\| \vecteur{x}^{(i)} - \vecteur{x}^{(i')} \right\| / \theta$.
\end{proof}

\begin{lem} \label{Lem:Matern_noninteger_smoothness_framework_proof_theorem}
  For Matérn kernels with noninteger smoothness $\nu$, $\bs{W} \trans \co \bs{W}$ can be decomposed as
  
  \begin{equation} 
  \bs{W} \trans \co \bs{W} = g(\theta) \left( \bs{W} \trans \nouveauD \bs{W} + g^{\star}(\theta) \bs{W} \trans \nouveauD^{\star} \bs{W} + \bs{R}_g(\theta) \right),
  \end{equation}
  
  where 
  
  \begin{itemize}
  \item $g$ is a positive differentiable function on $(0,+\infty)$;
  \item  $g^{\star}(\theta) = \theta^{-2l}$ with $l \in (0,+\infty)$;
  \item $\bs{R}_g$ is a differentiable mapping from $(0,+\infty)$ to $\mathcal{M}_n$ such that $\| \bs{R}_g(\theta) \| = o(\theta^{-2l})$ and $\| \frac{d}{d\theta} \bs{R}_g(\theta) \| = o(\theta^{-2l-1})$ when $\theta \to +\infty$;
  \item $\nouveauD$ and $\nouveauD^{\star}$ are both fixed symmetric matrices;
  \item $\bs{W} \trans \nouveauD \bs{W}$ is non-null.
  \end{itemize}
\end{lem}

Lemma \ref{Lem:Matern_noninteger_smoothness_framework_proof_theorem} is used in the proof of Theorem \ref{THM:REFERENCE_POSTERIOR_PROPER}.

\begin{proof}
  We use the notations of Lemma \ref{Lem:Matern_noninteger_nu_asymptotic_expansion}. This lemma implies that 

  \begin{align}
    \bs{W} \trans \co \bs{W}
    = \sum_{k=0}^{\floor{\nu}} \frac{a_k}{\theta^{2k}} \bs{W} \trans \bs{D}^{(k)} \bs{W} +
     \frac{a_\nu}{\theta^{2\nu}} \bs{W} \trans \bs{D}^{(\nu)} \bs{W} + \nonumber \\
     \frac{a_{\floor{\nu}+1}}{\theta^{2(\floor{\nu}+1)}} \bs{W} \trans \bs{D}^{(\floor{\nu}+1)} \bs{W} +\bs{W} \trans \bs{R}(\theta) \bs{W}. 
  \end{align}
    
  Lemma \ref{Lem:Matern_decay} implies that when $\theta$ is large enough, $\co - \bs{R}(\theta) - \frac{a_{\floor{\nu}+1}}{\theta^{2(\floor{\nu}+1)}} \bs{D}^{(\floor{\nu}+1)}$ is positive definite.
  \medskip
  
  Since the kernel of $\bs{W}$ is trivial, when $\theta$ is large enough, this implies in turn that $\bs{W} \trans \co \bs{W} - \bs{W} \trans \bs{R}(\theta) \bs{W} - \frac{a_{\floor{\nu}+1}}{\theta^{2(\floor{\nu}+1)}} \bs{W} \trans \bs{D}^{(\floor{\nu}+1)} \bs{W}$ is positive definite. If it exists, let $k_1$ be the smallest nonnegative integer smaller than $\nu$ such that $\bs{W} \trans \bs{D}^{(k_1)} \bs{W}$ is non-null and define $\nouveauD:=a_{k_1} \bs{D}^{(k_1)}$ and $g(\theta):=\theta^{-2k_1}$. If not, then define $\nouveauD:=a_{\nu} \bs{D}^{(\nu)}$ and $g(\theta):=\theta^{-2\nu}$. In any case, $\bs{W} \trans \nouveauD \bs{W}$ is non-null. \medskip 
  
  If $k_1$ exists and $\bs{W} \trans \bs{D}^{(k_1)} \bs{W}$ is nonsingular, then define $k_2:=k_1+1$ if $k_1 < \floor{\nu}$ and $k_2=\nu$ if $k_1=\floor{\nu}$. Then define $\nouveauD^{\star}:=a_{k_2} \bs{D}^{(k_2)}$ and $g^{\star}(\theta):=g(\theta)^{-1} \theta^{-2k_2} = \theta^{-2l}$ where $l=k_2-\nu$.  \medskip
  
  If $k_1$ exists and  $\bs{W} \trans \bs{D}^{(k_1)} \bs{W}$ is singular, then there must exist $k \in [\![k_1+1,\floor{\nu}]\!] \cup \{\nu\}$ such that $\bs{W} \trans \bs{D}^{(k)} \bs{W}$ is non-null. Let $k_2$ be the smallest number among all such $k$. Define $\nouveauD^{\star}:=a_{k_2} \bs{D}^{(k_2)}$ and $g^{\star}(\theta):=g(\theta)^{-1} \theta^{-2k_2} = \theta^{-2l}$ where $l=k_2-\nu$. \medskip
  
  If $k_1$ does not exist, then $\bs{W} \trans \bs{D}^{(\nu)} \bs{W}$ is necessarily nonsingular. Define $\nouveauD^{\star} := a_{\floor{\nu}+1} \matrice{D}^{(\floor{\nu}+1)}$ and $g^{\star}(\theta) = g(\theta)^{-1} \theta^{-2(\floor{\nu}+1)}
  = \theta^{-2l}$ where $l= \floor{\nu}+1-\nu $. \medskip
  
  Finally, define 
  
  \begin{equation}
  \bs{R}_g(\theta) := g(\theta)^{-1} g^{\star}(\theta)^{-1} \left( \bs{W} \trans \co \bs{W} - g(\theta) \bs{W} \trans \nouveauD \bs{W} -  g(\theta) g^{\star}(\theta) \bs{W} \trans \nouveauD^{\star} \bs{W} \right).
  \end{equation}

  In all situations, $\| \bs{R}_g(\theta) \| = o(\theta^{-2l})$ and $\| \frac{d}{d\theta} \bs{R}_g(\theta) \| = o(\theta^{-2l-1})$.
  
\end{proof}

\subsubsection{Matérn kernels with integer smoothness \texorpdfstring{$\nu$}{Lg}} \label{Sec:asymptotic_corr_Matern_integer_smoothness}

\begin{lem} \label{Lem:Matern_integer_nu_asymptotic_expansion}
  If a Matérn kernel with integer smoothness $\nu$ is used, we can write $\co$ as  

  \begin{equation} \label{Eq:AD_Matern_integer_nu}
    \co = \sum_{k=0}^{\nu-1} \frac{a_k}{\theta^{2k}} \bs{D}^{(k)} 
    +  \tilde{a}_\nu \left(  \frac{\log(\theta)}{\theta^{2\nu}} \bs{D}^{(\nu)} +   \frac{1}{\theta^{2\nu}} \bs{\tilde{D}}^{(\nu)} \right) 
    + \tilde{a}_{\nu+1} \frac{\log(\theta)}{\theta^{2(\nu+1)}} \bs{D}^{(\nu+1)}
    + \bs{R}(\theta) .
  \end{equation}
  
  \begin{itemize}
  \item $a_k$ and $\bs{D}^{(k)}$ $(k \in \N)$
  have the same definitions as in Lemma \ref{Lem:Matern_noninteger_nu_asymptotic_expansion}.
  \item $\tilde{a}_\nu := (-1)^\nu 2 \nu^\nu / ((\nu-1)! \nu!)$.
  \item $\tilde{a}_{\nu+1} := (-1)^\nu 2 \nu^{\nu+1} / ((\nu-1)! (\nu+1)!)$.
  \item 
  $\bs{\tilde{D}}^{(\nu)}$ is the  $n \times n$ matrix with null diagonal and $(i,i')$-th element ($i\neq i'$) given by 
  $$\| \bs{x}^{(i)} - \bs{x}^{(i')} \|^{2\nu} \left\{-\log \left( \| \bs{x}^{(i)} - \bs{x}^{(i')} \| \right) - \frac{\log(\nu)}{2} - \gamma + \sum_{l=1}^\nu \frac{1}{2 l} \right\} ,$$ where $\gamma$ is Euler's constant. 
  \item $\bs{R}$ is a differentiable mapping from $(0,+\infty)$ to the space of real $n \times n$ matrices $\mathcal{M}_n$ that satisfies
  $\| \bs{R}(\theta) \| =  O(\theta^{-2(\nu+1)})$ and $\| \frac{d}{d\theta} \bs{R}(\theta) \| =  O(\theta^{-2(\nu+1)-1})$ when $\theta \to +\infty$.
  \end{itemize}
\end{lem}

Lemma \ref{Lem:Matern_integer_nu_asymptotic_expansion} serves to prove Lemmas \ref{Lem:upper_bound_yfactor_in_likelihood} and \ref{Lem:Matern_integer_smoothness_framework_proof_theorem} and thus indirectly contributes to the proof of Theorem \ref{THM:REFERENCE_POSTERIOR_PROPER}.
It is also used in the proof of Proposition \ref{Prop:small_order_asymptotic_decomposition}.

\begin{proof}
 Let us combine Equations 9.6.10, 9.6.11 and 6.3.2 from \citet{AS64}. 
 Letting $\gamma$ be Euler's constant, we obtain:

 \begin{align}
   & \mathcal{K}_\nu(z) \nonumber \\
   := &
   \frac{1}{2} \left(\frac{z}{2}\right)^{-\nu} 
   \sum_{k=0}^{\nu-1} \frac{(\nu-k-1)!}{k!} \left(-\frac{z^2}{4}\right)^k \nonumber \\
   & + 
   (-1)^{\nu+1} \log \left( \frac{z}{2} \right) \left(\frac{z}{2}\right)^{\nu} \sum_{k=0}^\infty \frac{1}{k! (\nu+k)!} \left(\frac{z^2}{4}\right)^k \nonumber \\
   & +
   (-1)^\nu \frac{1}{2} \left( \frac{z}{2} \right)^\nu \sum_{k=0}^\infty \left[ -2\gamma + \sum_{l=1}^k l^{-1} + \sum_{l=1}^{\nu + k} l^{-1} \right] \frac{1}{k! (\nu+k)!} \left(\frac{z^2}{4}\right)^k.
 \end{align}

 Let us now compute the value of the Matérn kernel with integer smoothness parameter $\nu$ at $z$:

 \begin{align}
  &\left((\nu-1)!\right)^{-1} 2^{1-\nu} (2 \sqrt{\nu} z)^\nu \mathcal{K}_\nu(2 \sqrt{\nu} z)  \label{Eq:decomp_matern_nu_entier} \\  
  :=&
    \sum_{k=0}^{\nu-1} \frac{(-1)^k (\nu-k-1)! \nu^k}{(\nu-1)! k!} z^{2k} \nonumber \\
  & 
  + \sum_{k=\nu}^\infty \frac{(-1)^{\nu} \nu^k}{(\nu-1)! (k-\nu)! k!} 
  \left[- 2\log (z) - \log(\nu) -2\gamma + \sum_{l=1}^{k - \nu} l^{-1} + \sum_{l=1}^{k} l^{-1} \right] 
  z^{2k} . \nonumber
 \end{align}

The $(i,i')$-th element of the matrix $\co$ is given by Equation \eqref{Eq:decomp_matern_nu_entier} with $z= \left\| \vecteur{x}^{(i)} - \vecteur{x}^{(i')} \right\| / \theta$:

\begin{align}
  &\left((\nu-1)!\right)^{-1} 2^{1-\nu} \left(2 \sqrt{\nu} \frac{\left\| \vecteur{x}^{(i)} - \vecteur{x}^{(i')} \right\| }{\theta} \right)^\nu \mathcal{K}_\nu \left(2 \sqrt{\nu} \frac{\left\| \vecteur{x}^{(i)} - \vecteur{x}^{(i')} \right\| }{\theta} \right) \\  
  :=&
    \sum_{k=0}^{\nu-1} \frac{(-1)^k (\nu-k-1)! \nu^k}{(\nu-1)! k!} \frac{\left\| \vecteur{x}^{(i)} - \vecteur{x}^{(i')} \right\|^{2k} }{\theta^{2k}} \nonumber \\
  & 
  + \sum_{k=\nu}^\infty \frac{(-1)^{\nu} 2 \nu^k}{(\nu-1)! (k-\nu)! k!} 
  \left[ \log(\theta) - \log \left(\left\| \vecteur{x}^{(i)} - \vecteur{x}^{(i')} \right\|\right) - \frac{\log(\nu)}{2} \right. \nonumber \\
  &
  \qquad \qquad \left. -  \gamma + \sum_{l=1}^{k - \nu} \frac{1}{2 l} - \sum_{l=1}^{k} \frac{1}{2 l} \right] 
  \frac{\left\| \vecteur{x}^{(i)} - \vecteur{x}^{(i')} \right\|^{2k} }{\theta^{2k}} . \nonumber
  \end{align}

  The result follows.
\end{proof}

\begin{lem} \label{Lem:Matern_integer_smoothness_framework_proof_theorem}
  For Matérn kernels with integer smoothness $\nu$ and for $\theta \in (1,+\infty)$, $\bs{W} \trans \co \bs{W}$ can be decomposed as
  
  \begin{equation} 
  \bs{W} \trans \co \bs{W} = g(\theta) \left( \bs{W} \trans \nouveauD \bs{W} + g^{\star}(\theta) \bs{W} \trans \nouveauD^{\star} \bs{W} + \bs{R}_g(\theta) \right),
  \end{equation}
  
  where 
  
  \begin{itemize}
  \item $g$ is a positive differentiable function on $(1,+\infty)$;
  \item $\nouveauD$ and $\nouveauD^{\star}$ are both fixed symmetric matrices;
  \item $\bs{W} \trans \nouveauD \bs{W}$ is non-null;
  \item $g^{\star}(\theta) = \log(\theta)^{-1}$ if there exist non-null real numbers $\lambda, \lambda^\star$ such that $\nouveauD= \lambda \matrice{D}^{(\nu)}$ and $\nouveauD^\star = \lambda^\star \matrice{\tilde{D}}^{(\nu)}$ ($\matrice{D}^{(\nu)}$ and $\matrice{\tilde{D}}^{(\nu)}$ are defined in Lemma \ref{Lem:Matern_integer_nu_asymptotic_expansion});
  \item  $g^{\star}(\theta) = \theta^{-2l}$ or $g^{\star}(\theta) = \log(\theta)\theta^{-2l}$ with $l \in (0,+\infty)$ otherwise;
  \item $\bs{R}_g$ is a differentiable mapping from $(0,+\infty)$ to $\mathcal{M}_n$ such that $\| \bs{R}_g(\theta) \| = o(g^\star(\theta))$ and $\| \frac{d}{d\theta} \bs{R}_g(\theta) \| = o(g^{\star \prime}(\theta))$ when $\theta \to +\infty$.
  \end{itemize}
\end{lem}

Lemma \ref{Lem:Matern_integer_smoothness_framework_proof_theorem} is used in the proof of Theorem \ref{THM:REFERENCE_POSTERIOR_PROPER}.

\begin{proof}
We use the notations of Lemma \ref{Lem:Matern_integer_nu_asymptotic_expansion}. This lemma implies that 

\begin{equation}
\begin{split}
 \bs{W} \trans \co \bs{W}
= \sum_{k=0}^{\nu-1} \frac{a_k}{\theta^{2k}} \bs{W} \trans \bs{D}^{(k)} \bs{W} 
+  \frac{\log(\theta)}{\theta^{2\nu}} \tilde{a}_\nu \bs{W} \trans \bs{D}^{(\nu)} \bs{W} 
+ \frac{\tilde{a}_\nu}{\theta^{2\nu}} \bs{W} \trans \bs{\tilde{D}}^{(\nu)} \bs{W} + \nonumber \\
\frac{\log(\theta)}{\theta^{2(\nu+1)}} \tilde{a}_{\nu+1} \bs{W} \trans \bs{D}^{(\nu+1)} \bs{W}
+ \bs{W} \trans \bs{R}(\theta) \bs{W}. 
\end{split}
\end{equation}

Lemma \ref{Lem:Matern_decay} implies that when $\theta$ is large enough, 
$\co - \bs{R}(\theta)$ is positive definite. 
Since the kernel of $\bs{W}$ is trivial, 
this implies in turn that when $\theta$ is large enough, 
$$\bs{W} \trans \co \bs{W} 
- \bs{W} \trans \bs{R}(\theta) \bs{W}
- \frac{\log(\theta)}{\theta^{2(\nu+1)}} \tilde{a}_{\nu+1}
\matrice{W} \trans \matrice{D}^{(\nu+1)} \matrice{W}$$
is positive definite. 
If it exists, let $k_1$ be the smallest nonnegative integer 
smaller or equal to $\nu$ such that 
$\bs{W} \trans \bs{D}^{(k_1)} \bs{W}$ is non-null and define 
$\nouveauD:=a_{k_1} \bs{D}^{(k_1)}$ 
and $g(\theta)=\theta^{-2k_1}$ ($k_1<\nu$) or 
$\nouveauD:=\tilde{a}_\nu \bs{D}^{(\nu)}$ 
and $g(\theta):=\log(\theta) \theta^{-2\nu}$ ($k_1=\nu$). 
If not, then define $\nouveauD:=\tilde{a}_{\nu} \bs{\tilde{D}}^{(\nu)}$ 
and $g(\theta):=\theta^{-2\nu}$. 
In any case, $\bs{W} \trans \nouveauD \bs{W}$ is non-null. \medskip 

If $\bs{W} \trans \nouveauD \bs{W}$ is nonsingular, then

\begin{itemize}
\item either $k_1$ exists and is strictly smaller than $\nu-1$, in which case define $\nouveauD^{\star}:=a_{k_1+1} \bs{D}^{(k_1+1)}$ and $g^{\star}(\theta):=\theta^{-2}$;
\item or $k_1$ exists and is equal to $\nu-1$, in which case define $\nouveauD^{\star}:=\tilde{a}_\nu \bs{D}^{(\nu)}$ and $g^{\star}(\theta):=\log(\theta)\theta^{-2}$;
\item or $k_1$ exists and is equal to $\nu$, in which case define $\nouveauD^{\star}:=\tilde{a}_\nu \bs{\tilde{D}}^{\nu}$ and $g^{\star}(\theta):=\log(\theta)^{-1}$;
\item or $k_1$ does not exist, in which case define $\nouveauD^{\star}:= \tilde{a}_{\nu+1} \matrice{D}^{(\nu+1)}$ and $g^{\star}(\theta):=\log(\theta) \theta^{-2}$.
\end{itemize}

If $\bs{W} \trans \nouveauD \bs{W}$ is singular, then $k_1$ necessarily exists:

\begin{itemize}
\item either $k_1$ is strictly smaller than $\nu$. Then there are two possibilities. The first is that there exists a smallest integer $k_2 \in [\![k_1+1,\nu]\!]$ such that $\bs{W} \trans \bs{D}^{(k_2)} \bs{W}$ is non-null, in which case define $\nouveauD^{\star}:=a_{k_2} \bs{D}^{(k_2)}$ and $g^{\star}(\theta):=\theta^{-2(k_2-k_1)}$ ($k_2<\nu$) or  $\nouveauD^{\star}:=\tilde{a}_\nu \bs{D}^{(\nu)}$ and $g^{\star}(\theta):=\log(\theta)\theta^{-2(\nu-k_1)}$ ($k_2=\nu$). The second is that no such $k_2$ exists, but then $\bs{W} \trans \bs{\tilde{D}}^{(\nu)} \bs{W}$ is necessarily non-null, so define $\nouveauD^{\star}:=\tilde{a}_\nu \bs{\tilde{D}}^{(\nu)}$ and $g^{\star}(\theta):=\theta^{-2(\nu-k_1)}$.
\item or $k_1$ is equal to $\nu$. Then $\bs{W} \trans \bs{\tilde{D}}^{(\nu)} \bs{W}$ is necessarily non-null, so define $\nouveauD^{\star}:=\tilde{a}_\nu \bs{\tilde{D}}^{(\nu)}$ and $g^{\star}(\theta):=\log(\theta)^{-1}$.
\end{itemize}

Finally, define

\begin{equation}
\bs{R}_g(\theta) := g(\theta)^{-1} g^{\star}(\theta)^{-1} \left( \bs{W} \trans \co \bs{W} - g(\theta) \bs{W} \trans \nouveauD \bs{W} -  g(\theta) g^{\star}(\theta) \bs{W} \trans \nouveauD^{\star} \bs{W} \right).
\end{equation}

In all situations, $\| \bs{R}_g(\theta) \| = o(g^{\star}(\theta))$ and $\| \frac{d}{d\theta} \bs{R}_g(\theta) \| = o(g^{\star \prime}(\theta))$.

\end{proof}

\subsection{Behavior of \texorpdfstring{$\de$}{Lg}}

In order to prove Proposition \ref{Prop:ref_prior_behavior}, we need results about the asymptotic behavior of $\de$ whan $\theta \to + \infty$. Lemma \ref{Lem:majoration_derivee_RQ_SE} concerns Rational Quadratic and Squared Exponential kernels, while Lemma \ref{Lem:majoration_derivee_Matern} concerns Matérn kernels.

\begin{lem} \label{Lem:majoration_derivee_RQ_SE}
  For Rational Quadratic and Squared Exponential isotropic correlation kernels,
  consider the decomposition of $\co$ when $\theta$ is large given in Lemma \ref{Lem:RQ_SE_asymptotic_expansion}.
  Letting $N$ be the smallest nonnegative integer such that $\cap_{k=0}^N \Ker \matrice{D}^{(k)}$ is the trivial vector space,
  for large enough $\theta$,
  the matrix $ \theta^{-1} \co - \de$
  is positive definite and $\forall \vecteur{\xi} \in \R^n$, 
  $\vecteur{\xi} \trans \left( \theta^{-1} \co - \de \right) \vecteur{\xi} \leqslant 2(N+1) \theta^{-1} \vecteur{\xi} \trans \co \vecteur{\xi}$.
\end{lem}

\begin{proof}
  First, notice that as long as $\theta$ is large enough, Lemma \ref{Lem:developpement_asymptotique_defini_positif} is applicable with: 
  \begin{itemize}
    \item $\co$ playing the role of $\matrice{M}(\theta^{-1})$;
    \item $\matrice{D}^{(k)}$ playing the role of $\matrice{A}_k$ for every nonnegative integer $k \leqslant N$;
    \item $\frac{a_k}{\theta^{2k}}$ playing the role of $a_k(\theta^{-1})$ for every nonnegative integer $k \leqslant N$;
    \item $\sum_{k=N+1}^\infty \frac{a_k}{\theta^{2k}} \matrice{D}^{(k)}$ playing the role of $\matrice{B}(\theta^{-1})$.
  \end{itemize}
  Lemma \ref{Lem:developpement_asymptotique_defini_positif} is applicable because $\cap_{k=0}^N \Ker \matrice{D}^{(k)}$ is the trivial vector space. Lemma \ref{Lem:developpement_asymptotique_defini_positif} in turn makes Lemma  \ref{Lem:encadrement_developpement_asymptotique} applicable. \medskip 

  Define the vector subspaces $V_0,\dots,V_N$ with respect to $\matrice{D}^{(0)},\dots, \matrice{D}^{(N)}$ as required by Lemma \ref{Lem:encadrement_developpement_asymptotique}. 
  For any $\epsilon>0$, as long as $\theta$ is large enough, for all $\vecteur{\xi} = \vecteur{\xi}_0 + \dots + \vecteur{\xi}_N \in   \R^n = V_0 \overset{\perp}{\oplus} \dots \overset{\perp}{\oplus} V_{N}$,

  \begin{equation} \label{Eq:minoration_corr}
    \vecteur{\xi} \trans \co \vecteur{\xi}
    \geqslant
    (1-\epsilon) \sum_{k=0}^N \frac{a_k}{\theta^{2k}} \vecteur{\xi}_k \matrice{D}^{(k)} \vecteur{\xi}_k.
  \end{equation}

  Now let us consider the derivative $\de$. For large enough $\theta$:

  \begin{equation}
    \de = 
    \sum_{k=1}^\infty -2k \frac{a_k}{\theta^{2k+1}} \matrice{D}^{(k)}.
  \end{equation}

  Therefore
  
  \begin{equation}
    \theta^{-1} \co - \de
    =
    \theta^{-1} \sum_{k=0}^\infty (1+2k) \frac{a_k}{\theta^{2k}} \matrice{D}^{(k)}.
  \end{equation}

  Once again, Lemma \ref{Lem:encadrement_developpement_asymptotique} is applicable. For any $\epsilon>0$, as long as $\theta$ is large enough, for all $\vecteur{\xi} = \vecteur{\xi}_0 + \dots + \vecteur{\xi}_N \in   \R^n = V_0 \overset{\perp}{\oplus} \dots \overset{\perp}{\oplus} V_{N}$,

  \begin{align}
    \vecteur{\xi} \trans \left( \theta^{-1} \co - \de \right) \vecteur{\xi}
    & \leqslant
    (1+\epsilon) \sum_{k=0}^N (1+2k) \frac{a_k}{\theta^{2k}} \vecteur{\xi}_k \matrice{D}^{(k)} \vecteur{\xi}_k \nonumber \\
    & \leqslant
    (1+\epsilon) (1+2N) \sum_{k=0}^N \frac{a_k}{\theta^{2k}} \vecteur{\xi}_k \matrice{D}^{(k)} \vecteur{\xi}_k. \label{Eq:majoration_F}
  \end{align}  

  Combining Equations \eqref{Eq:minoration_corr} and \eqref{Eq:majoration_F} yields that for any $\epsilon>0$, as long as $\theta$ is large enough, for all $\vecteur{\xi} \in \R^n$,

  \begin{equation}
    \vecteur{\xi} \trans \left(\theta^{-1} \co - \de\right) \vecteur{\xi}
    \leqslant
    \frac{1+\epsilon}{1-\epsilon} (1+2N) \vecteur{\xi} \trans \co \vecteur{\xi}.
  \end{equation}

  If $\epsilon$ is taken small enough, $\frac{1+\epsilon}{1-\epsilon} (1+2N) \leqslant 2(N+1)$, which yields the result.
\end{proof}

\begin{lem} \label{Lem:majoration_derivee_Matern}
  For Matérn kernels, for all $\theta \in (0,+\infty)$, the matrix $r \theta^{-1} \co - \de$ is symmetric positive definite.
  Furthermore, for any $\bs{\xi} \in \R^n$,
  \begin{equation}
    \bs{\xi} \trans \left( r \theta^{-1} \co - \de \right)\bs{\xi} \leqslant (2\nu+r) \theta^{-1} \bs{\xi} \trans \co \bs{\xi}.
  \end{equation}
\end{lem}
  
\begin{proof}
  We adopt the notations of the proof of Lemma \ref{Lem:Matern_spectre}. \medskip
  
  For any $\theta \in (0,+\infty)$, for any $\bs{\xi} = (\xi_1,...,\xi_n) \trans \in \R^n$, 
  
  \begin{equation}
  \begin{split}
  \frac{d}{d\theta} I_\theta (\bs{\xi}) &= (-2)\left(\frac{r}{2} + \nu\right) \theta
  \int_{\R^r} \| \bs{s} \|^2 \left( 4 \nu + \theta^2 \| \bs{s} \|^2 \right)^{-\frac{r}{2} - \nu - 1} \left| \sum_{j=1}^n \xi_j e^{i \langle \left. \bs{s} \right| \bs{x}^{(j)} \rangle } \right|^2 d \bs{s} \\
  &= -(2\nu + r) \theta^{-1}
  \int_{\R^r} \frac{\theta^2 \| \bs{s} \|^2}{4 \nu + \theta^2 \| \bs{s} \|^2} \left( 4 \nu + \theta^2 \| \bs{s} \|^2 \right)^{-\frac{r}{2} - \nu} \left| \sum_{j=1}^n \xi_j e^{i \langle \left. \bs{s} \right| \bs{x}^{(j)} \rangle } \right|^2 d \bs{s}.
  \end{split}
  \end{equation}
  
  Since $\theta^2 \| \bs{s} \|^2 \leqslant 4\nu + \theta^2 \| \bs{s} \|^2$, for any $\theta \in (0,+\infty)$ and any non-null vector $\bs{\xi} \in \R^n$,
  
  \begin{equation} \label{Eq:Ineq_deriv_Matern}
  0 < - \frac{d}{d\theta} I_\theta (\bs{\xi})
  \leqslant (2\nu + r) \theta^{-1} I_\theta (\bs{\xi}).
  \end{equation}
  
  Combining Lemma \ref{Lem:Matern_spectre} and  Lemma \ref{Lem:derivative_spectral_decomposition} with Equation (\ref{Eq:Ineq_deriv_Matern}) yields the result.
\end{proof}

\bibliography{biblio}

\end{document}